\documentclass[11pt,twoside]{article}

\usepackage{amsmath, amsthm,amsfonts,amssymb,mathtools,mathdots,scalerel}
\usepackage[utf8]{inputenc}
\usepackage[T1]{fontenc}
\usepackage{lmodern, tikz-cd}

\usepackage{enumitem, tocloft, csquotes}
\usepackage[hidelinks]{hyperref}
\usepackage{tikz, fancyhdr, xparse, xcolor}
\usepackage[british]{babel}
\usepackage[a4paper, top=4.5cm, bottom=4.5cm, left=4cm, right=4cm, asymmetric]{geometry}

\usepackage{ahtitle,titlesec}
\usepackage{etoolbox}
\makeatletter
\patchcmd{\ttlh@hang}{\parindent\z@}{\parindent\z@\leavevmode}{}{}
\patchcmd{\ttlh@hang}{\noindent}{}{}{}
\makeatother

\usetikzlibrary{arrows, decorations.markings, shapes.geometric, decorations.pathmorphing, decorations.pathreplacing, intersections, patterns,calc, backgrounds}

\tikzset{
	0c/.style={circle, draw, fill, inner sep=.7pt},
	1c/.style={->, shorten <=2pt, shorten >=2pt},
	1clong/.style={->},
	edge/.style={shorten <=2pt, shorten >=2pt},
	equal/.style={shorten <=2pt, shorten >=2pt, double},
	1cinc/.style={right hook->, shorten <=2pt, shorten >=2pt},
	1csurj/.style={->>, shorten <=2pt, shorten >=2pt},
	1cincl/.style={left hook->, shorten <=2pt, shorten >=2pt},
	arloop/.tip={Glyph[glyph math command=looparrowleft, swap]},
	1cloop/.style={arloop->, shorten <=2pt, shorten >=2pt},
	2c/.style={double, shorten <=6pt, shorten >=8pt, decoration={markings,mark=at position -6pt with {\arrow[scale=1.75]{>}}}, preaction={decorate}},
	3c1/.style={double, double distance=3pt, shorten <=9pt, shorten >=11pt},
    	3c2/.style={shorten <=9pt, shorten >=10pt},
	3c3/.style={shorten <=9pt, shorten >=10pt, decoration={markings,mark=at position -8pt with {\arrow[scale=2]{>}}},preaction={decorate}},
	follow/.style={->, >=stealth, ultra thick, shorten <=3pt, shorten >=3pt, color=gray!70},
	arlabel/.style={scale=.8}
}

\newcommand\eqdef{\coloneqq}
\newcommand\nbd{\nobreakdash-\hspace{0pt}}
\newcommand\idd[1]{\mathrm{id}_{#1}}
\newcommand\bigid[1]{\mathrm{Id}_{#1}}
\newcommand\invrs[1]{#1^{-1}}
\newcommand\restr[2]{{#1}{\raisebox{0pt}{$|_{#2}$}}}
\newcommand\powset[1]{\mathcal{P}(#1)}

\newcommand\incl{\hookrightarrow}
\newcommand\incliso{\stackrel{\sim}{\hookrightarrow}}

\newcommand\surj{\twoheadrightarrow}
\newcommand\pfun{\rightharpoonup}
\newcommand\cfun{\looparrowright}

\newcommand\undl[1]{\underline{#1}}
\newcommand\slice[2]{{#1}/{\raisebox{-2pt}{$#2$}}}

\newcommand\gray{\,{\otimes}\,}
\newcommand\join{\,{\star}\,}

\newcommand\opp[1]{{#1}^\mathrm{op}}
\newcommand\coo[1]{{#1}^\mathrm{co}}
\newcommand\oppall[1]{{#1}^\circ}
\newcommand\oppn[2]{D_{#1}{#2}}

\newcommand\hasse[1]{\mathcal{H}#1}
\newcommand\hasseo[1]{\mathcal{H}^o#1}
\newcommand\dmn[1]{\mathrm{dim}(#1)}
\newcommand\clos[1]{\mathrm{cl}#1}
\NewDocumentCommand \bord{g g} {\IfNoValueTF{#2}{%
	\IfNoValueTF{#1}{\partial}{\partial_{#1}}}{\partial_{#1}^{#2}}}
\NewDocumentCommand \sbord{g g} {\IfNoValueTF{#2}{%
	\IfNoValueTF{#1}{\Delta}{\Delta_{#1}}}{\Delta_{#1}^{#2}}}
\newcommand\cp[1]{\,{\scriptstyle\#}_{#1}\,}

\newcommand\mol[2]{\mathcal{M}o\ell{#2}_{#1}}
\newcommand\submol{\sqsubseteq}
\newcommand\supmol{\sqsupseteq}
\newcommand\closub[1]{\mathcal{C}\ell(#1)}

\newcommand\infl[1]{O{(#1)}}
\newcommand\eps[1]{\varepsilon_{#1}}
\newcommand\cls[1]{\mathcal{#1}}

\newcommand\rev[1]{\mathrm{rev}(#1)}
\newcommand\lunit[3]{L^{#3}_{{#1} \incl {#2}}}
\newcommand\lmap[3]{\ell^{#3}_{{#1} \incl {#2}}}
\newcommand\runit[3]{R^{#3}_{{#1} \incl {#2}}}
\newcommand\rmap[3]{r^{#3}_{{#1} \incl {#2}}}

\newcommand\extr[2]{E_{#1}(\Delta^{#2})}
\newcommand\extrtil[2]{E'_{#1}(\Delta^{#2})}
\newcommand\horn[2]{\Lambda_{{#1} \incl {#2}}}

\newcommand\face[1]{\mathfrak{F}{#1}}

\newcommand\celto{\Rightarrow}
\newcommand\compos[1]{\langle#1\rangle}

\newcommand\atom{{\raisebox{-.02em}{%
\begin{tikzpicture}[baseline={(current bounding box.south)}]%
	\node[circle, draw, line width=.05em, inner sep=.25em] at (0,0) {};%
	\node[circle, fill, inner sep=.06em] at (0,0) {};%
	\node[circle, fill, inner sep=.06em] at (0,.35em) {};%
\end{tikzpicture}}}}

\newcommand\degg[1]{\mathrm{Deg}{#1}}
\newcommand\equi[1]{\mathrm{Eqv}{#1}}
\newcommand\winv[1]{\mathrm{Inv}{#1}}
\newcommand\cdgm[1]{\mathrm{Dgm}^\mathcal{C}{#1}}
\newcommand\dvd[1]{\mathrm{Dvd}{(#1)}}
\newcommand\divis[2]{\mathrm{Sol}{((#1), #2)}}

\newcommand\monad[1]{\mathsf{T}_{#1}}
\newcommand\algmon[1]{{#1}\cat{Alg}}
\newcommand\imonad{\mathsf{I}}
\newcommand\mmonad{\mathsf{M}}
\newcommand\mimonad{\mathsf{M}\oplus\mathsf{I}}
\newcommand\tmonad{\mathsf{T}}

\newcommand\equifun[2]{\mathcal{E}^{#1}(#2)}
\newcommand\twothree[2]{\mathcal{T}^{#1}(#2)}
\newcommand\winfun[1]{\mathcal{I}(#1)}
\newcommand\pasting[2]{\mathcal{P}st^{#1}(#2)}

\newcommand\lunitor[3]{\lambda^{#3}_{{#1} \submol {#2}}}
\newcommand\runitor[3]{\rho^{#3}_{{#1} \submol {#2}}}

\newcommand\skel[2]{\sigma_{\leq #1}#2}
\newcommand\coskel[2]{\tau_{\leq #1}#2}

\newcommand\nerve[1]{N_{\scalebox{.7}{\atom}} {#1}}
\newcommand\sing[1]{S_{\scalebox{.7}{\atom}} {#1}}

\newcommand\delres[1]{#1_\Delta}
\newcommand\pin[3]{\pi_{#1}^{#2}(#3)}
\newcommand\subdiv{\mathrm{Sd}}
\newcommand\exfun{\mathrm{Ex}}
\newcommand\realis[1]{|#1|}

\newcommand\cat[1]{\mathbf{#1}}
\newcommand\psh[2]{\mathrm{PSh}_{#1}(#2)}

\newcommand\ogpos{\cat{ogPos}}
\newcommand\ogposin{\cat{ogPos}_\textit{in}}
\newcommand\pos{\cat{Pos}}
\newcommand\set{\cat{Set}}

\newcommand\omegagph{\omega\cat{Gph}}
\newcommand\omegagphref{\omega\cat{Gph}_\textit{ref}}
\newcommand\omegacat{\omega\cat{Cat}}
\newcommand\pomegacat{\cat{p}\omega\cat{Cat}}

\newcommand\rdcpx{\cat{DCpx}^\cls{R}}
\newcommand\rdcpxin{\cat{DCpx}^\cls{R}_\textit{in}}
\newcommand\dcpxin{\cat{DCpx}^\cls{C}_\textit{in}}
\newcommand\dcpx{\cat{DCpx}^\cls{C}}
\newcommand\dcpxfun{\cat{DCpx}^\cls{C}_\textit{fun}}
\newcommand\dcpxomega{\cat{DCpx}^\cls{C}_{\omega}}

\newcommand\rmolin{\cat{Mol}^\cls{R}_\textit{in}}

\newcommand\atomin{\atom_\textit{in}}
\newcommand\dgmset{\atom\cat{Set}}
\newcommand\cpol{\cat{Pol}^\cls{C}}
\newcommand\nucat{\omegacat^\cls{C}_\textit{nu}}
\newcommand\dgmcat{\atom\cat{Cat}}

\newcommand\deltacat{\Delta}
\newcommand\sset{\cat{sSet}}
\newcommand\cghaus{\cat{cgHaus}}

\newtheoremstyle{ittheorem}
  {\topsep}   
  {\topsep}   
  {\itshape}  
  {0pt}       
  {\itshape \bfseries} 
  { ---}         
  {5pt plus 1pt minus 1pt} 
  {}          

\newtheoremstyle{itdfn}
  {\topsep}   
  {\topsep}   
  {}  
  {0pt}       
  {\bfseries} 
  {.}         
  {5pt plus 1pt minus 1pt} 
  {}          

\makeatletter
  \renewcommand\@upn{\textit}
\makeatother

\theoremstyle{ittheorem}
\newtheorem{thm}{Theorem}[section]
\newtheorem{prop}[thm]{Proposition}
\newtheorem{cor}[thm]{Corollary}
\newtheorem{lem}[thm]{Lemma}
\newtheorem{conj}[thm]{Conjecture}
\newtheorem*{simp}{Simpson's homotopy hypothesis}
\theoremstyle{itdfn}
\newtheorem{dfn}[thm]{}
\theoremstyle{remark}
\newtheorem{rmk}[thm]{Remark}
\newtheorem{comm}[thm]{Comment}
\newtheorem{exm}[thm]{Example}

\setlist{leftmargin=20pt,itemsep=0pt,topsep=1ex}

\titleformat{\section}
 {\large\scshape}{\thesection.}{1em}{}

\titleformat{\subsection}
 {\normalsize\itshape}{\thesubsection.}{1em}{}
\titlespacing*{\subsection}
{0pt}{1.5ex plus 1ex minus .2ex}{1.5ex plus .2ex}

\renewcommand{\cftsecpagefont}{\mdseries}

\makeatletter \renewcommand{\cftsecfillnum}[1]{%
  {\cftsecleader}\nobreak
  \makebox[\@pnumwidth][\cftpnumalign]{\cftsecpagefont \oldstylenums{#1}}\cftsecafterpnum\par
} \makeatother

\newcommand\runtitle{diagrammatic sets}
\newcommand\runauthor{amar hadzihasanovic}

\title{Diagrammatic sets\\ and rewriting in weak higher categories}

\author{Amar Hadzihasanovic}

\institution{IRIF, CNRS, Universit\'e de Paris}

\keywords{diagrammatic sets, higher categories, rewriting}

\begin{document}


\maketitle 

\noindent\makebox[\textwidth][r]{%
\begin{minipage}[t]{.7\textwidth}
\small \textit{Abstract.}
We revisit Kapranov and Voevodsky's idea of spaces modelled on combinatorial pasting diagrams, now as a framework for higher-dimensional rewriting and the basis of a model of weak $\omega$\nbd categories. In the first part, we elaborate on Steiner's theory of directed complexes as a combinatorial foundation. We individuate convenient classes of directed complexes and develop the theory of diagrammatic sets relative to one such class. We study a notion of equivalence internal to a diagrammatic set, and single out as models of weak $\omega$\nbd categories those diagrammatic sets whose every composable diagram is connected by an equivalence to a single cell. We then define a semistrict model providing algebraic composites and study the embedding of strict $\omega$\nbd categories into this model. Finally, we prove a version of the homotopy hypothesis for the $\infty$-groupoids in the weak model, and exhibit a specific mistake in a proof by Kapranov and Voevodsky that had previously been refuted indirectly.
\end{minipage}}

\vspace{20pt}

\makeaftertitle

\normalsize


\noindent\makebox[\textwidth][c]{%
\begin{minipage}[t]{.5\textwidth}
\setcounter{tocdepth}{1}
\tableofcontents
\end{minipage}}

\section*{Introduction}

In 1991, Mikhail Kapranov and Vladimir Voevodsky published an article in the \emph{Cahiers} claiming that strict $\omega$\nbd categories model all homotopy types in the sense of Grothendieck's homotopy hypothesis \cite{kapranov1991infty}. In 1998, Carlos Simpson proved that their result is false and that strict 3-categories already do not model all 3-types \cite{simpson1998homotopy}. Due to the notorious intricacy of $\omega$\nbd categorical combinatorics and the paucity of details in Kapranov--Voevodsky, he could not point to a specific fatal mistake in their proof.

Homotopy theorists have since adopted a variety of models of weak higher categories satisfying strong forms of the homotopy hypothesis. Meanwhile, strict higher categories have remained popular in \emph{higher-dimensional and diagrammatic rewriting}, in particular
\begin{itemize}
	\item polygraph-based rewriting theory \cite{guiraud2016polygraphs}, stemming from Albert Burroni's work \cite{burroni1993higher}, and
	\item applied category theory \cite{fong2019invitation}, whose characteristic use of string diagrams relies on low-dimensional coherence and strictification theorems.
\end{itemize}
The use of higher categories in this context has some peculiarities which explain the persistence of strict models, such as a focus on \emph{presented} higher categories and the explicit manipulation of cells including surgery (cutting and pasting) and orientation-reversal.\footnote{In a certain sense, higher category theory applied to rewriting is closer to the combinatorial roots of homotopy theory than to its modern model-independent manifestations.}

To perform such manipulations, one needs a sufficiently rich combinatorial theory of \emph{shapes} of diagrams. To interpret this theory in a model of higher categories, one needs a \emph{pasting theorem}, granting that each diagram has a unique interpretation, given a well-formed assignment of cells to its constituents. 

Several strong pasting theorems were proved in the late 1980s and early 1990s for strict $\omega$\nbd categories \cite{johnson1989combinatorics, power1991pasting, steiner1993algebra}. To the best of our knowledge, these have no analogues for the weak models used in homotopy theory, which tend to privilege simple shapes.

As a result, in the current landscape of models there is a trade-off between the homotopy hypothesis and pasting theorems, between the ability to do homotopical algebra and the ability to do rewriting in higher categories. This reaches a critical point as soon as we apply rewriting techniques to a context where the ``strict Eckmann--Hilton argument'' applies.\footnote{By which a strict 3-category with one 0-cell and no nondegenerate 1-cells is a model of a commutative monoidal category, instead of a braided monoidal category.} A particular \emph{\'echec} in the author's work is detailed in \cite[Section 2.3]{hadzihasanovic2017algebra}.

This article aims to offer a way out of this trade-off by setting up a combinatorial framework for higher-dimensional rewriting which doubles as the basis of a model of higher categories satisfying the homotopy hypothesis. For this purpose, it revisits an idea from Kapranov--Voevodsky, namely, the notion of \emph{diagrammatic set}.

\subsection*{Diagrammatic sets}

Kapranov and Voevodsky recognised that the combinatorics developed in the service of pasting theorems for strict $\omega$\nbd categories could be turned to a different purpose: defining \emph{shape categories} that vastly generalise the simplex category or the cube category, while keeping some of their important properties, such as the unique factorisation of morphisms into co-degeneracies and co-faces. 

A diagrammatic set is a presheaf on such a shape category. Kapranov and Voevodsky based their definition on Michael Johnson's theory of \emph{composable pasting schemes}. This was a flawed choice, for reasons discussed in \cite[Appendix A.2]{henry2019non}: composable pasting schemes are not closed under a number of natural rewriting operations.

Instead, we base our revisitation on Richard Steiner's theory of \emph{directed complexes} \cite{steiner1993algebra}, continuing the elaboration that we started in \cite{hadzihasanovic2018combinatorial}. The basic data structure is a finite graded poset together with an orientation, that is, an edge-labelling of the poset's Hasse diagram in the set $\{+,-\}$. The intuition is that of an \emph{oriented face poset} of a directed cell complex, by analogy with the face poset of a cell complex.

The model of a cell shape is called an \emph{atom}, and the model of a diagram shape is called a \emph{molecule}. Each molecule has a \emph{boundary} which is subdivided into an \emph{input} and an \emph{output}. We obtain more or less well-behaved classes of directed complexes by demanding that the input and output boundaries of atoms belong to a certain class $\cls{C}$ of molecules. 

The paradigmatic class that we consider is the class $\cls{S}$ of molecules \emph{with spherical boundary}, and correspondingly the class of \emph{regular} directed complexes. Regular directed complexes have the property that their underlying poset is the face poset of a regular CW complex: their atoms ``are homeomorphic to'' balls and have boundaries ``homeomorphic to'' spheres.\footnote{We first tried to capture regularity with the notion that we call \emph{constructibility} in \cite{hadzihasanovic2018combinatorial}. Then we learnt and adopted a more general notion from Simon Henry \cite{henry2018regular}.} 

We realised that all our proofs could be made relative to a class $\cls{C} \subseteq \cls{S}$ closed under a number of operations, all natural from the point of view of rewriting. These are: reversing the orientation of all atoms of a certain dimension; forming a new atom whose boundaries are molecules in $\cls{C}$ with isomorphic boundaries; pasting two molecules along the entire input or output boundary of one or the other; forming ``units'' and ``unitors'' over a molecule in $\cls{C}$; and taking \emph{Gray products} and \emph{joins} of two molecules.\footnote{Gray products and joins are important in the theory of higher categories as a means of describing higher-order morphisms and (co)limits, respectively. In \cite[Chapter 2]{hadzihasanovic2017algebra} we argued that Gray products also have a r\^ole in universal algebra and rewriting.}

We call such a class \emph{convenient}, and define diagrammatic sets relative to a fixed convenient class $\cls{C}$.

\subsection*{Models of higher categories}

Diagrammatic sets have a rich combinatorics of face and degeneracy operations but no notion of composition. To define weak composition in a diagrammatic set, we adopt the ``non-algebraic'' approach of having $(n+1)$\nbd dimensional cells exhibit weak composites of $n$\nbd dimensional diagrams. This is shared with the complicial \cite{verity2008weak} and opetopic \cite{baez1998higher} models of higher categories.

First, we define a notion of \emph{equivalence} in a diagrammatic set. An equivalence is a diagram $e$ with the property that certain well-formed ``equations'' of the form $x \cup e = y$ in the indeterminate $x$ admit a ``solution'' exhibited by a higher-dimensional diagram which is itself an equivalence. This coinductive definition turns out to identify the same class of diagrams as a characterisation based on Eugenia Cheng's notion of pseudo-invertibility \cite{cheng2007omega}. We prove that equivalences have various expected properties, such as a form of two-out-of-three.

We say that a diagrammatic set \emph{has weak composites} if every \emph{composable} diagram --- one whose shape is a molecule in the convenient class $\cls{C}$ --- is connected by an equivalence to a single cell. A diagrammatic set with weak composites is our model of a weak $\omega$\nbd category.

We also define a semistrict algebraic model which has composition of composable diagrams as an associative algebraic operation. This continues a project that we set up in \cite[Section 6]{hadzihasanovic2018weak}. The underlying idea is that degeneracy operations and composition operations are separately well-behaved in conjunction with face operations over diagrams with spherical boundary, and that it is their unrestricted interaction that produces ``strict Eckmann--Hilton''. Then a promising strategy towards semi-strictification results is to separate the two algebraic structures and fine-tune their interaction.

With this aim, we identify diagrammatic sets with algebras of a monad $\imonad$ on a category of \emph{combinatorial polygraphs} that only have face operations, and define notions of \emph{non-unital $\omega$\nbd categories} as algebras of another monad $\mmonad$ encoding composition. We define \emph{diagrammatic $\omega$\nbd categories} as combinatorial polygraphs with separate $\imonad$\nbd algebra and $\mmonad$\nbd algebra structures, satisfying a condition that we call \emph{unit-merge compatibility}.\footnote{This is the same compatibility that determined a distributive law between the analogues of $\mmonad$ and $\imonad$ in \cite[Section 6]{hadzihasanovic2018weak}. For now we do not see a credible pathway to a distributive law of $\mmonad$ over $\imonad$ in the present context.} 

We show that the underlying diagrammatic set of a diagrammatic $\omega$\nbd category has weak composites, and set up a framework for the proof of a semi-strictification result --- every diagrammatic set with weak composites is weakly equivalent to a diagrammatic $\omega$\nbd category in a suitable sense --- by turning it into a combinatorial problem about morphisms of directed complexes. Unfortunately, a full proof seems beyond our current understanding.

We also prove that, subject to an additional condition on $\cls{C}$ which may or may not hold of all convenient classes, the category of strict $\omega$\nbd categories fully embeds into the category of diagrammatic $\omega$\nbd categories.

\subsection*{Homotopy hypothesis}

Kapranov and Voevodsky used diagrammatic sets as an intermediate step towards a proof of the homotopy hypothesis for strict $\omega$\nbd categories. First they claimed to prove the homotopy hypothesis for \emph{Kan diagrammatic sets}\footnote{These are to diagrammatic sets what Kan complexes are to simplicial sets.} \cite[Theorem 2.9]{kapranov1991infty}. Then they claimed that the homotopy theory of Kan diagrammatic sets is equivalent to the homotopy theory of $\infty$\nbd groupoids in strict $\omega$\nbd categories [Theorem 3.7, \emph{ibid.}].

Because Simpson's refutation did not point to a specific mistake, it has remained open whether one half or the other or both are incorrect. This work vindicates in part the first half by reproving a version of its main result. We prove specifically the minimal version of the homotopy hypothesis that Simpson refuted for strict $\omega$\nbd categories \cite[Theorem 4.4.2]{simpson2009homotopy}, to which we refer as \emph{Simpson's homotopy hypothesis}.

The setting is a category whose objects we interpret as higher categories and morphisms as functors, together with a full subcategory of higher groupoids. There are notions of \emph{$n$\nbd cells} in a higher category for each $n \in \mathbb{N}$. Each $n$\nbd cell has a \emph{boundary}, subdivided into two or more $(n-1)$\nbd cells. A \emph{composite} of two $n$\nbd cells exists whenever their boundaries satisfy a matching condition. There are higher-dimensional \emph{units} on each cell.

We have a notion of homotopy groups of a higher groupoid, with a combinatorial presentation along the following lines. For each $n > 0$, if $X$ is a higher groupoid pointed with a 0-cell $v$, then $\pi_n(X,v)$ is a set of $n$\nbd cells in $X$ whose boundary consists of units on $v$, quotiented by a relation $x \sim y$ generated by $(n+1)$\nbd cells whose boundary contains both $x$ and $y$. The group multiplication and unit are induced by the composition and units in $X$. A set $\pi_0(X)$ of connected components is similarly defined. These definitions are functorial.

Finally, each higher groupoid $X$ has a functorial \emph{Postnikov tower}, that is, a functorial sequence of morphisms under $X$
\begin{equation*}
	\ldots \to \coskel{n}{X} \to \coskel{n-1}{X} \to \ldots \to \coskel{0}{X}
\end{equation*}
such that $\pi_m(\coskel{n}{X},v)$ is trivial for all $m > n$ and 0-cells $v$, and $X \to \coskel{n}{X}$ induces isomorphisms of homotopy groups for all $k \leq n$. Higher groupoids for which $X \to \coskel{n}{X}$ is an isomorphism model $n$\nbd groupoids.

\begin{simp}
There exist 
\begin{enumerate}
	\item a functor $R$ from (pointed) higher groupoids to (pointed) spaces and
	\item natural isomorphisms between $\pi_0(X)$ and $\pi_0(RX)$ and between $\pi_n(X,v)$ and $\pi_n(R(X,v))$ for all $n > 0$.
\end{enumerate}
The functor $R$ is surjective on homotopy types, that is, for each space $X'$ there exists a higher groupoid $X$ such that $X'$ and $RX$ are weakly equivalent.
\end{simp}

This is a relatively weak version of the homotopy hypothesis and we believe that stronger forms are within reach. Still, we focussed on this version for two reasons. The first is that it suffices to localise the fatal mistake to the second half of \cite{kapranov1991infty}. The second is that we were able to give a fairly simple combinatorial proof, relying on the concrete reduction of a Kan diagrammatic set to the Kan complex it contains. Thus the proof doubles as a proof-of-concept of the combinatorial capabilities of our framework.

Incidentally, as a result of this work, we are able to confidently pinpoint the crucial mistake in Kapranov--Voevodsky. As suspected already by others,\footnote{See the answers to the MathOverflow question \url{https://mathoverflow.net/q/234492}.} the mistake is in the proof of Lemma 3.4, specifically in the sentence
\begin{displayquote}
	--- \emph{Therefore there exists a sequence of ``elementary'' transformations (rebracketings, introducing and collapsing of degenerate cells) starting from $(S, \gamma_S)$ and ending in $(T, \gamma_T)$.}
\end{displayquote}
We exhibit a counterexample to this claim at the end of the article.

\subsection*{Related work}

This work is strongly related and indebted to Simon Henry's recent articles \cite{henry2019non,henry2018regular}. Like ours, Henry's work stems from Kapranov and Voevodsky's article and Simpson's response. Its principal aim is the construction of semistrict algebraic models of homotopy types. In this respect, the version of the homotopy hypothesis proved for ``fibrant regular $\omega$\nbd categories'' in \cite{henry2018regular} is stronger than Simpson's homotopy hypothesis that we proved for Kan diagrammatic sets.

On the other hand, Henry does not present a model of higher categories that specialises to a model of higher groupoids: the underlying structure of regular $\omega$\nbd categories lacks units, either algebraic or non-algebraic. Due to the absence of units, Henry's framework is also not a solution to the problem of rewriting in weak higher categories, where units are constantly needed to model ``nullary'' inputs or outputs.

While Henry's approach is to start from the algebra of strict $\omega$\nbd categories and place minimal constraints that render it susceptible of a combinatorial exploration, our approach is to start from ``tame'' combinatorics and prove their soundness for various fragments of the theory of $\omega$\nbd categories. The two complement each other and a synthesis seems within reach. 

Specifically, we think that Conjecture \ref{conj:algebraically_free} implies that Henry's \emph{regular polygraphs} are equivalent to what we call combinatorial $\cls{S}$\nbd polygraphs and that our \emph{subdivisions} model Henry's \emph{generic morphisms}, in which case Henry's regular $\omega$\nbd categories are equivalent to our non-unital $\cls{S}$\nbd $\omega$\nbd categories.

We are aware of only one other approach to rewriting in weak higher categories, namely, the \emph{associative} model of $n$\nbd categories developed by Christoph Dorn, David Reutter, and Jamie Vicary as a generalisation of Gray categories to $n > 3$ and a foundation for the \texttt{homotopy.io} proof assistant \cite{dorn2018associative,reutter2019high}. It is still open whether this model satisfies the homotopy hypothesis in higher dimensions. 

\subsection*{Structure of the article}

Section \ref{sec:complexes}, Section \ref{sec:operations}, and Section \ref{sec:classes} are our elaboration on Steiner's theory of directed complexes and contain the combinatorial foundation for the rest of the article. They are followed by a brief Section \ref{sec:diagrammatic} consisting of the essential definitions relative to diagrammatic sets. Section \ref{sec:equivalences} starts with the definition of equivalence and proves its various properties; this is mostly technical. Section \ref{sec:omegacats} presents both our weak and our semistrict model of higher categories. Section \ref{sec:nerves} studies the embedding of strict $\omega$\nbd categories into the semistrict model. Section \ref{sec:homotopy} is devoted to the proof of Simpson's homotopy hypothesis and ends with the mistake in Kapranov--Voevodsky.

There are many possible paths for reading. The reader interested only in the combinatorics of higher-categorical shapes may read the first three sections, then skip to Section \ref{sec:nerves} and add the first part of Section \ref{sec:homotopy}.

Everyone else should at least read Section \ref{sec:diagrammatic}. Unless they are willing to treat the shape category as a black box and uninterested in the formal description of diagrams, they should also read Section \ref{sec:complexes}. 

If they only care about homotopy types, they may skip directly to Section \ref{sec:homotopy}: it uses a few results about equivalences and weak composites, otherwise it is mostly independent. The counterexample to Kapranov--Voevodsky, however, also relies on Section \ref{sec:nerves}.

The reader interested in rewriting should read Section \ref{sec:operations} and the definition of convenient classes in Section \ref{sec:classes}, as should anyone who wants to follow any of the proofs in the rest of the article. 

The reader interested in models of higher categories should read at least the definitions and results in Section \ref{sec:equivalences} and the first part of Section \ref{sec:omegacats}. The reader interested in \emph{algebraic} models of higher categories should add Section \ref{sec:classes} and the rest of Section \ref{sec:omegacats}.

\subsection*{Outlook and open problems}

We state two numbered conjectures, Conjecture \ref{conj:weak_equivalence} and Conjecture \ref{conj:algebraically_free}. The first would imply that the semistrict model is a ``harmless'' strictification of the weak model, the second that strict $\omega$\nbd categories embed into every version of the semistrict model. Both would be very desirable features of the semistrict model. Both proofs seem to require, mainly, a further understanding of the combinatorics of directed complexes and their morphisms.

Another welcome development would be the proof of a stronger homotopy hypothesis, such as the construction of a model structure on diagrammatic sets whose fibrant objects are the Kan diagrammatic sets, Quillen equivalent to the classical model category of spaces. A related problem is constructing a model structure whose fibrant objects are diagrammatic sets with weak composites.

Then there is the question of comparing ours to other models of weak higher categories. The most straightforward comparison may be with Dominic Verity's complicial model, using the relations between diagrammatic sets and simplicial sets established in Section \ref{sec:homotopy}. Emily Riehl has announced in \cite{riehl2018complicial} a proof of the Barwick--Schommer-Pries axioms \cite{barwick2011unicity} for the complicial model: relating the models may lead to further proofs of equivalence of models at the level of their homotopy theory.

To conclude, much of this work originates from the problem of formalising the ``smash product of algebraic theories'' of \cite[Section 2.3]{hadzihasanovic2017algebra}. We are confident that diagrammatic sets answer that problem, and plan to return to it soon.

\section{Regular directed complexes} \label{sec:complexes}


\subsection{Oriented graded posets}

\begin{dfn}[Graded poset]
Let $P$ be a finite poset with order relation $\leq$. For all elements $x, y \in P$, we say that $y$ \emph{covers} $x$ if $x < y$ and, for all $y' \in X$, if $x < y' \leq y$ then $y' = y$. 

The \emph{Hasse diagram} of $P$ is the finite directed graph $\hasse{P}$ with $\hasse{P}_0 \eqdef P$ as set of vertices and $\hasse{P}_1 \eqdef \{y \to x \mid y \text{ covers } x\}$ as set of edges. 

Let $P_\bot$ be $P$ extended with a least element $\bot$. We say that $P$ is \emph{graded} if, for all $x \in P$, all directed paths from $x$ to $\bot$ in $\hasse{P_\bot}$ have the same length. If this length is $n+1$, we let $\dmn{x} \eqdef n$ be the \emph{dimension} of $x$.
\end{dfn}

\begin{dfn}[Closed and pure subsets]
Let $P$ be a poset and $U \subseteq P$. The \emph{closure} of $U$ is the subset $\clos{U} \eqdef \{x \in P \mid \exists y \in U \; x \leq y\}$ of $P$. We say that $U$ is \emph{closed} if $U = \clos{U}$. 

Suppose $P$ is graded and $U \subseteq P$ is closed. Then $U$ is graded with the partial order inherited from $P$. The \emph{dimension} $\dmn{U}$ of $U$ is $\max\{\dmn{x} \mid x \in U\}$ if $U$ is inhabited, $-1$ otherwise. In particular, $\dmn{\clos{\{x\}}} = \dmn{x}$. 

We say that $U$ is \emph{pure} if its maximal elements all have dimension $\dmn{U}$.
\end{dfn}

\begin{dfn}[Oriented graded poset]
An \emph{orientation} on a finite poset $P$ is an edge-labelling $o\colon \hasse{P}_1 \to \{+,-\}$ of its Hasse diagram.

An \emph{oriented graded poset} is a finite graded poset with an orientation.
\end{dfn}

\begin{dfn} We will often let variables $\alpha, \beta$ range implicitly over $\{+,-\}$.
\end{dfn}

\begin{dfn}[Boundaries]
Let $P$ be an oriented graded poset and $U \subseteq P$ a closed subset. Then $U$ inherits an orientation from $P$ by restriction. 

For all $\alpha \in \{+,-\}$ and $n \in \mathbb{N}$, we define
\begin{align*}
	\sbord{n}{\alpha} U & \eqdef \{x \in U \mid \dmn{x} = n \text{ and if $y \in U $ covers $x$, then $o(y \to x) = \alpha$} \}, \\
	\bord{n}{\alpha} U & \eqdef \clos{(\sbord{n}{\alpha} U)} \cup \{ x \in U \mid \text{for all $y \in U$, if $x \leq y$, then $\dmn{y} \leq n$} \}, \\
	\sbord{n} U & \eqdef \sbord{n}{+}U \cup \sbord{n}{-}U, \quad \qquad \quad \bord{n}U \eqdef \bord{n}{+}U \cup \bord{n}{-} U.
\end{align*}
We call $\bord{n}{-}U$ the \emph{input $n$\nbd boundary}, and $\bord{n}{+}U$ the \emph{output $n$\nbd boundary} of $U$. 

If $U$ is $(n+1)$\nbd dimensional, we write $\sbord{}{\alpha}U \eqdef \sbord{n}{\alpha}U$ and $\bord{}{\alpha}U \eqdef \bord{n}{\alpha}U$. For each $x \in P$, we write $\sbord{n}{\alpha}x \eqdef \sbord{n}{\alpha}\clos{\{x\}}$ and $\bord{n}{\alpha}x \eqdef \bord{n}{\alpha}\clos{\{x\}}$.
\end{dfn}

\begin{lem} \label{lem:higherthandim}
Let $U$ be a closed subset of an oriented graded poset. Then $\bord{m}{\alpha}U = U$ for all $m \geq \dmn{U}$.
\end{lem}
\begin{proof}
If $U$ is $n$-dimensional, then for all $x \in U$ and $m \geq n$, for all $y \in U$ such that $x \leq y$, it holds that $\dmn{y} \leq m$, so $x \in \bord{m}{+}U = \bord{m}{-}U$. 
\end{proof}

\begin{lem} \label{lem:minimaldim}
Let $U$ be a closed inhabited subset of an an oriented graded poset. Then 
\begin{equation*}
	\dmn{U} = \mathrm{min}\{n \in \mathbb{N} \mid \bord{n}{+}U = \bord{n}{-}U = U\}.
\end{equation*}
\end{lem}
\begin{proof}
If $\dmn{U} = n$, by Lemma \ref{lem:higherthandim} we have $\bord{n}{+}U = \bord{n}{-}U = U$. For all $k < n$, $\bord{k}{\alpha}x$ is at most $k$\nbd dimensional, so it cannot cointain any $n$\nbd dimensional element of $U$.
\end{proof}

\begin{dfn}[Maps and inclusions]
A \emph{map} $f\colon P \to Q$ of oriented graded posets is a function of their underlying sets that satisfies
\begin{equation*}
	\bord{n}{\alpha}f(x) = f(\bord{n}{\alpha}x)
\end{equation*}
for all $x \in P$, $n \in \mathbb{N}$, and $\alpha \in \{+,-\}$. 

An \emph{inclusion} $\imath\colon P \incl Q$ is an injective map. An inclusion is an \emph{isomorphism} if it is also surjective.

Oriented graded posets and maps form a category $\ogpos$. We denote by $\ogposin$ the subcategory of oriented graded posets and inclusions.
\end{dfn}

\begin{lem}
Let $f\colon P \to Q$ be a map of oriented graded posets. Then $f$ is a closed, order-preserving, dimension-non-increasing function of the underlying posets.
\end{lem}
\begin{proof}
Let $x \in P$, and let $m$ be larger than the dimensions of both $x$ and $f(x)$. By Lemma \ref{lem:higherthandim}, $\clos\{f(x)\} = \bord{m}{\alpha}f(x) = f(\bord{m}{\alpha}x) = f(\clos\{x\})$. Since $y \leq x$ if and only if $y \in \clos\{x\}$, this proves that $f$ is both closed and order-preserving.

If $\dmn{x} = n$, then $\bord{n}{\alpha}f(x) = f(\bord{n}{\alpha}x) = f(\clos{\{x\}}) = \clos{\{f(x)\}}$. By Lemma \ref{lem:minimaldim}, the dimension of $f(x)$ is at most $n$.
\end{proof}

\begin{rmk}
It follows that there is a forgetful functor from $\ogpos$ to the category $\pos$ of posets and order-preserving functions.
\end{rmk}

\begin{lem} \label{lem:sameinclusions}
Let $\imath\colon P \incl Q$ be an inclusion of oriented graded posets. Then $\imath$ is order-reflecting, hence a closed embedding of the underlying posets, and preserves dimensions and orientations.
\end{lem}
\begin{proof}
Let $x, y$ be such that $\imath(x) \leq \imath(y)$. Then $\imath(x) \in \clos\{\imath(y)\} = \imath(\clos\{y\})$, so $\imath(x) = \imath(x')$ for some $x' \leq y$. Since $\imath$ is injective, $x = x'$. 

It follows that $\imath$ is a closed embedding of graded posets; in particular it preserves the covering relation and dimensions. It follows that $\imath(\sbord x) = \sbord \imath(x)$ for all $x \in P$, and since $\imath(\bord{}{\alpha}x) = \bord{}{\alpha}\imath(x)$, necessarily $\imath(\sbord{}{\alpha}x) = \sbord{}{\alpha}\imath(x)$, which is equivalent to $\imath$ preserving orientations.
\end{proof}

\begin{rmk}
Conversely, a closed embedding of posets that is compatible with the orientations is an inclusion of oriented graded posets. This was the definition of inclusion used in \cite{hadzihasanovic2018combinatorial}. 
\end{rmk}

\begin{lem} \label{lem:boundary_to_boundary}
Let $\imath\colon P \incliso Q$ be an isomorphism of oriented graded posets and suppose $\dmn{P} = \dmn{Q}$. Then $\imath(\bord{}{\alpha}P) = \bord{}{\alpha}Q$.
\end{lem}
\begin{proof}
Let $x \in \sbord{}{\alpha}P$ and $n \eqdef \dmn{P}$. Because $\imath$ preserves dimensions, $\imath(x)$ is $(n-1)$\nbd dimensional. Suppose $\imath(x)$ is covered by $y$. Because $\imath$ is surjective, $y = \imath(y')$ for some $y'$, and because it is order-reflecting, $x$ is covered by $y'$, necessarily with orientation $\alpha$. Since $\imath$ preserves orientations, $y$ covers $\imath(x)$ with orientation $\alpha$. This proves that $\imath(x) \in \sbord{}{\alpha}Q$. 

Now let $x \in P$ be such that $x \leq y$ implies $\dmn{y} < n$. If $\imath(x) \leq y$, then $y = \imath(y')$ for some $y'$ of the same dimension and $x \leq y'$. It follows that $\dmn{y}  = \dmn{y'} < n$.
\end{proof}

\begin{prop} 
Every map $f\colon P \to Q$ of oriented graded posets factors as a surjective map $P \surj \widehat{P}$ followed by an inclusion $\widehat{P} \incl Q$. This factorisation is unique up to isomorphism.
\end{prop}
\begin{proof}
The underlying map of posets admits an essentially unique factorisation as a surjective map followed by a closed embedding. Given a closed embedding $\widehat{P} \incl Q$, there is a unique orientation on $\widehat{P}$ which makes it an inclusion of oriented graded posets. Thus the factorisation lifts to $\ogpos$.
\end{proof}


\subsection{Directed complexes}

\begin{dfn}[Atoms and molecules]
Let $P$ be an oriented graded poset. We define a family of closed subsets of $P$, the \emph{molecules} of $P$, by induction on proper subsets. If $U$ is a closed subset of $P$, then $U$ is a molecule if either
\begin{itemize}
	\item $U$ has a greatest element, in which case we call it an \emph{atom}, or
	\item there exist molecules $U_1$ and $U_2$, both properly contained in $U$, and $n \in \mathbb{N}$ such that $U_1 \cap U_2 = \bord{n}{+}U_1 = \bord{n}{-}U_2$ and $U = U_1 \cup U_2$.
\end{itemize}
We define $\submol$ to be the smallest partial order relation such that, if $U_1$ and $U_2$ are molecules and $U_1 \cap U_2 = \bord{n}{+}U_1 = \bord{n}{-}U_2$, then $U_1, U_2 \submol U_1 \cup U_2$.

We say \emph{$n$\nbd molecule} for an $n$\nbd dimensional molecule. We say that $P$ itself is a molecule if $P \subseteq P$ is a molecule.
\end{dfn}

\begin{lem} \label{lem:molecule_codim}
Let $U$ be an $n$\nbd molecule in an oriented graded poset, $n > 0$, and let $x \in U$ be an $(n-1)$\nbd dimensional element.
\begin{itemize}
	\item If $x \in \sbord{}{+}U \cap \sbord{}{-}U$, then $x$ is not covered by any element,
	\item if $x \in \sbord U \setminus (\sbord{}{+}U \cap \sbord{}{-}U)$, then $x$ is covered by a single element, and 
	\item if $x \notin \sbord U$, then $x$ is covered by exactly two elements with opposite orientations.
\end{itemize}
\end{lem}
\begin{proof}
We proceed by induction on submolecules of $U$. When $U$ is an atom, $\sbord{}{+}U$ and $\sbord{}{-}U$ are disjoint, $x \in \sbord U$, and it is covered only by the greatest element.

Suppose $U = U_1 \cup U_2$ for some proper subsets with $U_1 \cap U_2 = \bord{k}{+}U_1 = \bord{k}{-}U_2$. If $k < n-1$, then $U_1 \cap U_2$ contains no $(n-1)$\nbd dimensional elements. Since no $(n-1)$\nbd dimensional element of $U$ can be covered both by an element of $U_1$ and by an element of $U_2$, it follows that $\sbord{}{\alpha}U$ is the disjoint union of $\sbord{n-1}{\alpha}U_1$ and $\sbord{n-1}{\alpha}U_2$. The statement then holds by the inductive hypothesis.

Let $k = n-1$. In this case, if $x \in U_1 \cap U_2$, then $x \in \sbord{n-1}{+}U_1 = \sbord{n-1}{-}U_2$. If $x \in \sbord{}{+}U \cap \sbord{}{-}U$, then by definition $x$ cannot be covered by any element.

Suppose that $x \in \sbord{}{+}U \setminus \sbord{}{-}U$. Then $x$ is covered at least by one element; suppose it is in $U_1$. Then $x \in \sbord{n-1}{+}U_1 \setminus \sbord{n-1}{-}U_1$, and by the inductive hypothesis it is only covered by one element of $U_1$. Suppose it is also covered by another element, necessarily in $U_2$. Then $x \in \sbord{n-1}{+}U_2 \setminus \sbord{n-1}{-}U_2$, but $\sbord{n-1}{+}U_1 = \sbord{n-1}{-}U_2$, a contradiction. 

The other possibilities with $x \in \sbord U \setminus (\sbord{}{+}U \cap \sbord{}{-}U)$ are handled in a similar way. Finally, suppose $x \notin \sbord U$. If $x \notin U_1 \cap U_2$, then $x \in U_1 \setminus \sbord{n-1}{}U_1$ or $x \in U_2 \setminus \sbord{n-1}{}U_2$, and the inductive hypothesis applies. If $x \in U_1 \cap U_2$, then $x \in \sbord{n-1}{+}U_1 \setminus \sbord{n-1}{-}U_1$, so $x$ is covered by a single element of $U_1$ with orientation $+$, and $x \in \sbord{n-1}{-}U_2 \setminus \sbord{n-1}{+}U_2$, so $x$ is covered by a single element of $U_2$ with orientation $-$.
\end{proof}

\begin{dfn}[Directed complex]
An oriented graded poset $P$ is a \emph{directed complex} if, for all $x \in P$ and $\alpha, \beta \in \{+,-\}$, 
\begin{enumerate}
	\item $\bord{}{\alpha}x$ is a molecule, and
	\item $\bord{}{\alpha}(\bord{}{\beta}x) = \bord{n-2}{\alpha}x$ if $n \eqdef \dmn{x} > 1$.
\end{enumerate}
\end{dfn}

\begin{rmk}
There is a slight difference with Steiner's definition \cite{steiner1993algebra}. There, the underlying structure of a directed complex is encoded by the functions $x \mapsto \sbord{}{\alpha}x$. The axioms ensure that 
\begin{center}
	``$y \to x$ if and only if $x \in \sbord{}{+}y \cup \sbord{}{-}y$''
\end{center}
is the covering relation of a graded poset, but they allow for the possibility that $\sbord{}{+}y$ and $\sbord{}{-}y$ are not disjoint, which is excluded in our setting.
\end{rmk}

\begin{dfn}[Globe]
For each $n \in \mathbb{N}$, let $O^n$ be the poset with a pair of elements $\undl{k}^+, \undl{k}^-$ for each $k < n$ and a greatest element $\undl{n}$, with the partial order defined by $\undl{j}^\alpha \leq \undl{k}^\beta$ if and only if $j \leq k$. This is a graded poset, with $\dmn{\undl{n}} = n$ and $\dmn{\undl{k}^\alpha} = k$ for all $k < n$. 

With the orientation $o(y \to \undl{k}^\alpha) \eqdef \alpha$ if $y$ covers $\undl{k}^\alpha$, $O^n$ becomes a directed complex; in fact, it is the smallest $n$\nbd dimensional directed complex. We call $O^n$ the \emph{$n$\nbd globe}.
\end{dfn}

\begin{dfn}[$\omega$-Graph]
Let $\cat{O}_\textit{in}$ be the full subcategory of $\ogposin$ whose objects are the globes. For all $n$ and $k < n$ there are exactly two inclusions $\imath^+, \imath^-\colon O^k \hookrightarrow O^n$, defined by $\imath^\alpha(\undl{k}) \eqdef \undl{k}^\alpha$ and $\imath^\alpha(\undl{j}^\beta) \eqdef \undl{j}^\beta$ for all $j < k$.

An \emph{$\omega$\nbd graph}, also known as globular set \cite[Section 1.4]{leinster2004higher}, is a presheaf on $\cat{O}_\textit{in}$. With their morphisms of presheaves, $\omega$\nbd graphs form a category $\omegagph$. 

We represent an $\omega$\nbd graph $X$ as a diagram
\begin{equation*}
\begin{tikzpicture}[baseline={([yshift=-.5ex]current bounding box.center)}]
	\node (0) at (0,0) {$X_0$};
	\node (1) at (2,0) {$X_1$};
	\node (2) at (4,0) {$\ldots$};
	\node (3) at (6,0) {$X_n$};
	\node (4) at (8,0) {$\ldots$};
	\draw[1c] (1.west |- 0,.15) to node[auto,swap,arlabel] {$\bord{}{+}$} (0.east |- 0,.15);
	\draw[1c] (1.west |- 0,-.15) to node[auto,arlabel] {$\bord{}{-}$} (0.east |- 0,-.15);
	\draw[1c] (2.west |- 0,.15) to node[auto,swap,arlabel] {$\bord{}{+}$} (1.east |- 0,.15);
	\draw[1c] (2.west |- 0,-.15) to node[auto,arlabel] {$\bord{}{-}$} (1.east |- 0,-.15);
	\draw[1c] (3.west |- 0,.15) to node[auto,swap,arlabel] {$\bord{}{+}$} (2.east |- 0,.15);
	\draw[1c] (3.west |- 0,-.15) to node[auto,arlabel] {$\bord{}{-}$} (2.east |- 0,-.15);
	\draw[1c] (4.west |- 0,.15) to node[auto,swap,arlabel] {$\bord{}{+}$} (3.east |- 0,.15);
	\draw[1c] (4.west |- 0,-.15) to node[auto,arlabel] {$\bord{}{-}$} (3.east |- 0,-.15);
\end{tikzpicture}
\end{equation*}
of sets and functions, where $X_n \eqdef X(O^n)$ and $\bord{}{\alpha} \eqdef X(\imath^\alpha) \colon X_n \to X_{n-1}$. 

The elements of $X_n$ are the \emph{$n$\nbd cells} of $X$. Identifying $\cat{O}_\textit{in}$ with a full subcategory of $\omegagph$ via the Yoneda embedding, we have a correspondence between $n$\nbd cells of $X$ and morphisms $O^n \to X$.

For each $x \in X_n$ and $k < n$, we let
\begin{equation*}
	\bord{k}{\alpha}x \eqdef \underbrace{\bord{}{\alpha}(\ldots(\bord{}{\alpha}}_{n-k} x)).
\end{equation*}
For all $n$-cells $x$ and $j < k < n$, we have $\bord{j}{\alpha}(\bord{k}{\beta}x) = \bord{j}{\alpha}x$ (\emph{globularity}).
\end{dfn}

\begin{dfn}[Reflexive $\omega$\nbd graph]
Let $\cat{O}$ be the full subcategory of $\ogpos$ whose objects are the globes. For all $n$ and $k < n$, there is a unique surjective map $\tau\colon O^n \surj O^k$, defined by $\tau(\undl{n}),\tau(\undl{j}^\alpha) \eqdef \undl{k}$ if $j \geq k$ and $\tau(\undl{j}^\alpha) \eqdef \undl{j}^\alpha$ if $j < k$. 

A \emph{reflexive $\omega$\nbd graph} is a presheaf on $\cat{O}$. With their morphisms of presheaves, $\omega$\nbd graphs form a category $\omegagphref$. 

In a reflexive $\omega$\nbd graph, we write $\eps{} \eqdef X(\tau)\colon X_{n-1} \to X_n$. For each $x \in X_k$ and $n > k$, let
\begin{equation*}
	\eps{n}x \eqdef \underbrace{(\eps{}\ldots\eps{})}_{n-k} x.
\end{equation*}
\end{dfn}

\begin{dfn}[$\omega$-Category] 
A \emph{partial $\omega$\nbd category} is a reflexive $\omega$\nbd graph $X$ together with partial $k$-composition operations
\begin{equation*}
	\cp{k}\colon X_n \times X_n \pfun X_n
\end{equation*}
for all $n \in \mathbb{N}$ and $k < n$, satisfying the following axioms:
\begin{enumerate}
	\item for all $n$\nbd cells $x$ and all $k < n$, 
	\begin{equation*}
		x \cp{k} \eps{n}(\bord{k}{+}x) = x = \eps{n}(\bord{k}{-}x) \cp{k} x,
	\end{equation*}
	where the two $k$\nbd compositions are always defined;
	\item for all $(n+1)$\nbd cells $x, y$ and all $k < n$, whenever the left-hand side is defined,
	\begin{align*}
		\bord{}{-}(x \cp{n} y) & = \bord{}{-} x, \\
		\bord{}{+}(x \cp{n} y) & = \bord{}{+} y, \\
		\bord{}{\alpha}(x \cp{k} y) & = \bord{}{\alpha} x \cp{k} \bord{}{\alpha} y;
	\end{align*}
	\item for all cells $x, y, x', y'$ and all $n$ and $k < n$, whenever the left-hand side is defined, 
	\begin{align*}
		\eps{}(x \cp{n} y) & = \eps{}x \cp{n} \eps{}y, \\
		(x \cp{n} x') \cp{k} (y \cp{n} y') & = (x \cp{k} y) \cp{n} (x' \cp{k} y');
	\end{align*}
	\item for all cells $x, y, z$ and all $n$, whenever either side is defined,
	\begin{equation*}
		(x \cp{n} y) \cp{n} z = x \cp{n} (y \cp{n} z);
	\end{equation*}
	\item for all cells $x, y$, if $x \cp{n} y$ is defined, then $\bord{n}{+}x = \bord{n}{-}y$.
\end{enumerate} 
An {$\omega$\nbd category} is a partial $\omega$\nbd category where the converse of the last condition holds, that is, if $\bord{n}{+}x = \bord{n}{-}y$ then $x \cp{n} y$ is defined.

A \emph{functor} of partial $\omega$\nbd categories is a morphism of the underlying reflexive graphs that respects the composition operations. Partial $\omega$\nbd categories and functors form a category $\pomegacat$, with a full subcategory $\omegacat$ on $\omega$\nbd categories. 
\end{dfn}

\begin{prop} \label{prop:partial_omegacat}
Let $P$ be a directed complex. There is a partial $\omega$\nbd category $\mol{}{P}$ where
\begin{enumerate}
	\item the set $\mol{n}{P}$ of $n$-cells is the set of molecules $U \subseteq P$ with $\dmn{U} \leq n$,
	\item $\bord{}{\alpha}\colon \mol{n}{P} \to \mol{n-1}{P}$ is $U \mapsto \bord{n-1}{\alpha}U$,
	\item $\eps{}\colon \mol{n-1}{P} \to \mol{n}{P}$ is $U \mapsto U$,
	\item $U \cp{n} V$ is defined when $U \cap V = \bord{n}{+}U = \bord{n}{-}V$, and in that case it is equal to $U \cup V$.
\end{enumerate}
\end{prop}
\begin{proof}
This is part of \cite[Proposition 2.9]{steiner1993algebra}.
\end{proof}

\begin{rmk}
In particular, if $U$ is a molecule in a directed complex, $\bord{n}{\alpha}U$ is a molecule for all $n \in \mathbb{N}$ and $\alpha \in \{+,-\}$.
\end{rmk}

\begin{lem} \label{lem:composition_form}
Let $U$ be an $n$\nbd molecule in a directed complex. Either $U$ is an atom, or
\begin{equation*}
	U = V_1 \cp{k} \ldots \cp{k} V_m
\end{equation*}
for some $m > 1$, $k < n$, and molecules $V_1, \ldots, V_m$ such that for $1 \leq i \leq j \leq m$
\begin{enumerate}[label=(\alph*)]
	\item $V_i$ contains exactly one atom $U_i$ with $\dmn{U_i} > k$ and 
	\item if $\dmn{U_i}, \dmn{U_j} > k+1$, then $i = j$.
\end{enumerate}
\end{lem}
\begin{proof}
The proof of \cite[Proposition 4.2]{steiner2004omega}, stated for $\omega$\nbd categories with a set of composition-generators, also applies to the partial $\omega$\nbd categories $\mol{}{U}$, which are composition-generated by the atoms.
\end{proof}

\begin{lem} \label{lem:atom_boundary}
Let $U$ be a molecule in a directed complex, $n \eqdef \dmn{U}$, and suppose $U$ contains a unique element $x$ of dimension $n$. Then $\bord{}{\alpha} x \submol \bord{}{\alpha} U$ for all $\alpha \in \{+,-\}$.
\end{lem}
\begin{proof}
If $U$ is an atom, $\bord{}{\alpha} x = \bord{}{\alpha}U$. Otherwise, let $U = V_1 \cp{k} \ldots \cp{k} V_m$ be the decomposition given by Lemma \ref{lem:composition_form}. Because at most one of the $V_i$ can contain $x$, necessarily $k < n-1$. Without loss of generality, suppose $x \in V_1$. Since $V_1$ is a proper submolecule of $U$, the inductive hypothesis applies, and $\bord{}{\alpha} x \submol \bord{}{\alpha} V_1$. By the axioms of $\omega$\nbd categories, $\bord{n-1}{\alpha}U = \bord{n-1}{\alpha}V_1 \cp{k} \ldots \cp{k} \bord{n-1}{\alpha}V_m$. We conclude that $\bord{}{\alpha} x \submol \bord{}{\alpha} U$. \end{proof}


\subsection{Regularity}

The following definition is based on \cite[Section 2.4]{henry2018regular}.

\begin{dfn}[Spherical boundary] Let $U$ be a molecule in a directed complex. By induction on $n \eqdef \dmn{U}$, we say that $U$ \emph{has spherical boundary} if $n = 0$, or if $n > 0$ and
\begin{enumerate}
	\item $\bord{}{-}U$ and $\bord{}{+}U$ are $(n-1)$\nbd molecules with spherical boundary, and
	\item $\bord{}{-}U \cap \bord{}{+}U = \bord (\bord{}{+}U) = \bord (\bord{}{-}U)$. 
\end{enumerate}
\end{dfn}

\begin{rmk} \label{rmk:spherical_k}
Unravelling the induction, and using globularity, we obtain the equivalent definition: $U$ has spherical boundary if, for all $k < n$,
\begin{equation*}
	\bord{k}{+}U \cap \bord{k}{-}U = \bord{k-1}{}U.
\end{equation*}
Moreover $\bord{k}{+}U$ is always a $k$\nbd molecule with spherical boundary.
\end{rmk}

\begin{dfn}[Regular directed complex]
A directed complex $P$ is \emph{regular} if, for all $x \in P$, the atom $\clos\{x\}$ has spherical boundary.
\end{dfn}

\begin{exm}
By \cite[Theorem 1]{hadzihasanovic2018combinatorial}, constructible molecules have spherical boundary. In particular, every constructible directed complex is a regular directed complex. However, the converse is not true: see [Remark 38, \emph{ibid.}] for a counterexample.
\end{exm}

\begin{dfn}
We write $\rdcpx$ and $\rdcpxin$ for the full subcategories of $\ogpos$ and $\ogposin$, respectively, on the regular directed complexes.
\end{dfn}

\begin{prop} 
The directed complex $1$ with a single element is the terminal object, and the empty directed complex $\emptyset$ is the initial object of $\rdcpx$.
\end{prop}
\begin{prop} \label{prop:rdcpxin_colimits}
The category $\rdcpxin$ has pushouts, created by the forgetful functor to $\pos$ and preserved by the inclusion in $\rdcpx$.
\end{prop}
\begin{proof} 
Left to the reader.
\end{proof}

\begin{cor}
Every regular directed complex is the colimit of the diagram of inclusions of its atoms.
\end{cor}
\begin{proof}
This is true of the underlying diagram in $\pos$, and the colimit can be constructed with pushouts and the initial object.
\end{proof}

\begin{lem} \label{lem:disjoint_sbord}
Let $U$ be an $n$\nbd molecule with spherical boundary. Then
\begin{enumerate}[label=(\alph*)]
	\item $U$ is pure,
	\item if $n > 0$, then $\sbord{}{+}U$ and $\sbord{}{-}U$ are disjoint and inhabited, and
	\item each $x \in \sbord U$ is covered by a single element.
\end{enumerate}
\end{lem}
\begin{proof}
If $n = 0$, there is nothing to prove. Let $n > 0$, and suppose $U$ has a maximal element of dimension $k < n$. Then $x$ is not covered by any element of dimension $>k$, so it belongs to $\sbord{k}{+}U \cap \sbord{k}{-}U$, but $\bord{k}{+}U \cap \bord{k}{-}U = \bord{k-1}{}U$ by Remark \ref{rmk:spherical_k}. The latter is at most $(k-1)$\nbd dimensional, a contradiction. This proves that $U$ is pure.

It follows from purity of $U$ that $\sbord{}{+}U$ and $\sbord{}{-}U$ are disjoint. They are inhabited because $\bord{}{+}U$ and $\bord{}{-}U$ have spherical boundary, so they are pure and $(n-1)$\nbd dimensional. The last point then follows from Lemma \ref{lem:molecule_codim}.
\end{proof}

\begin{lem} \label{lem:molecule_right_dimension}
Let $U$ be an $n$\nbd molecule in a regular directed complex. Then $\bord{k}{\alpha}U$ is $k$\nbd dimensional for all $k < n$.
\end{lem}
\begin{proof}
Replace ``constructible molecule'' with ``molecule with spherical boundary'' in the proof of \cite[Proposition 19]{hadzihasanovic2018combinatorial}.
\end{proof}

\begin{center}
\setlength{\fboxsep}{.6em}
\colorbox{gray!20}{From here on we assume implicitly that molecules are regular.}
\end{center}

\begin{lem} \label{lem:path_to_boundary}
Let $U$ be an $n$\nbd molecule, $x \in U$ an $n$\nbd dimensional element. There is a finite sequence $x = x_0 \to y_0 \to \ldots \to x_m \to y_m$ of elements of $U$ such that
\begin{enumerate}
	\item the $x_i$ are $n$\nbd dimensional and the $y_i$ are $(n-1)$\nbd dimensional, 
	\item $y_i \in \sbord{}{+}x_i \cap \sbord{}{-}x_{i+1}$ for $i < m$, and $y_m \in \sbord{}{+}U$.
\end{enumerate}
\end{lem}
\begin{proof}
By induction on proper submolecules of $U$: if $U$ is an atom, it is equal to $\clos\{x\}$. By Lemma \ref{lem:disjoint_sbord}, $\sbord{}{+}x$ is inhabited, so we can pick $y_0 \in \sbord{}{+}x = \sbord{}{+}U$.

Otherwise, $U = U_1 \cup U_2$ with $U_1 \cap U_2 = \bord{k}{+}U_1 = \bord{k}{-}U_2$. Suppose $x \in U_1$; by the inductive hypothesis there is a sequence $x \to y_0 \to \ldots \to x_m \to y_m$ reaching $\sbord{}{+}U_1$. If $k < n-1$, then $\sbord{}{+}U_1 \subseteq \sbord{}{+}U$, and we are done. Otherwise, by Lemma \ref{lem:molecule_codim}, either $y_m$ is not covered by any element of $U_2$, in which case $y_m \in \sbord{}{+}U$, or $y_m \in \sbord{}{-}x_{m+1}$ for some $x_{m+1} \in U_2$. By the inductive hypothesis, there is a sequence $x_{m+1} \to y_{m+1} \to \ldots \to x_p \to y_p$ reaching $\sbord{}{+}U_2 \subseteq \sbord{}{+}U$. Stringing together the two sequences, we conclude.
\end{proof}

\begin{prop} \label{prop:molecule_noauto}
Let $U$ be a molecule and $\imath\colon U \incliso U$ an isomorphism. Then $\imath$ is the identity.
\end{prop}
\begin{proof}
We proceed by induction on $n \eqdef \dmn{U}$. If $n = 0$, the statement is clearly true, so suppose $n > 0$. 

By Lemma \ref{lem:boundary_to_boundary}, $\imath(\bord{}{\alpha}U) = \bord{}{\alpha}U$. Because $\bord{}{\alpha}U$ is a molecule of dimension $k < n$, we know by the inductive hypothesis that the restriction of $\imath$ to $\bord{}{\alpha}U$ is the identity. If $x \in U$ is a maximal element such that $\imath(x) = x$, by the same reasoning we obtain that $\imath$ is the identity on $\clos\{x\}$. Thus it suffices to prove that $\imath$ fixes the maximal elements. 

If $U$ has a greatest element, this is obvious. Otherwise, let $x$ be a maximal element. If $\dmn{x} < n$, then $x \in \bord{}{\alpha}U$ and we have already established that $\imath(x) = x$. 

If $\dmn{x} = n$, we can construct a sequence $x = x_0 \to y_0 \to \ldots \to x_m \to y_m$ as in Lemma \ref{lem:path_to_boundary}. Such a sequence is mapped by $\imath$ to one with the same property, and $y_m = \imath(y_m)$ because $y_m \in \sbord{}{+} U$. By Lemma \ref{lem:molecule_codim}, $y_m$ is only covered by $x_m$, so necessarily $\imath(x_m) = x_m$. It follows that $\imath(y_{m-1}) = y_{m-1}$. Then $y_{m-1}$ is only covered by $x_{m-1}$ with orientation $+$, and proceeding backwards we find that $\imath(x) = x$. 
\end{proof}

\begin{dfn}
Let $\rmolin$ be the full subcategory of $\rdcpxin$ on molecules of any dimension. For each regular directed complex $P$, let $\slice{\rmolin}{P}$ be the comma category whose objects are inclusions $U \hookrightarrow P$ of molecules into $P$ and morphisms are commutative triangles
\begin{equation*}
\begin{tikzpicture}[baseline={([yshift=-.5ex]current bounding box.center)}]
	\node (0) at (-1.25,1.25) {$U$};
	\node (1) at (0,0) {$P$};
	\node (2) at (1.25,1.25) {$V$};
	\draw[1cinc] (0) to (2);
	\draw[1cincl] (0) to (1);
	\draw[1cinc] (2) to (1);
	\node at (1.5,0) {.};
\end{tikzpicture}
\end{equation*}
\end{dfn}

\begin{cor} \label{cor:slice_preorder}
Let $P$ be a regular directed complex. Then $\slice{\rmolin}{P}$ is a preorder.
\end{cor}
\begin{proof}
As a consequence of Lemma \ref{prop:molecule_noauto}, two inclusions $\imath_1, \imath_2\colon U \incl P$ are equal if and only if they have the same image in $P$. This establishes an equivalence between $\slice{\rmolin}{P}$ and a subposet of the subset lattice of $P$.
\end{proof}


\section{Operations on molecules} \label{sec:operations}

\subsection{Pasting and substitution}

\begin{dfn}[Pasting of molecules] \label{dfn:molecule_pasting}
Let $U_1, U_2$ be molecules and suppose that $\bord{k}{+}U_1$ and $\bord{k}{-}U_2$ are isomorphic. By Corollary \ref{cor:slice_preorder}, given an isomorphic copy $V$ of the two, there is a \emph{unique} span of inclusions $V \incl U_1$ and $V \incl U_2$ whose images are, respectively, $\bord{k}{+}U_1$ and $\bord{k}{-}U_2$.

In this case, let $U_1 \cp{k} U_2$ be the pushout
\begin{equation*}
\begin{tikzpicture}[baseline={([yshift=-.5ex]current bounding box.center)}]
	\node (0) at (0,1.5) {$V$};
	\node (1) at (2.5,0) {$U_1 \cp{k} U_2$};
	\node (2) at (0,0) {$U_1$};
	\node (3) at (2.5,1.5) {$U_2$};
	\draw[1cinc] (0) to (3);
	\draw[1cincl] (0) to (2);
	\draw[1cinc] (2) to (1);
	\draw[1cincl] (3) to (1);
	\draw[edge] (1.6,0.2) to (1.6,0.7) to (2.3,0.7);
\end{tikzpicture}
\end{equation*}
in $\rdcpx$. Then $U_1 \cp{k} U_2$ is a molecule, decomposing as $U_1 \cup U_2$ with $U_1 \cap U_2 = \bord{k}{+}U_1 = \bord{k}{-}U_2$. 

This establishes $\cp{k}$ as a partial operation defined \emph{up to unique isomorphism} on molecules: we can call $U_1 \cp{k} U_2$ ``the pasting'' of $U_1$ and $U_2$ in the same way as we speak of ``the limit'' of a functor. 

Moreover, $\omega$\nbd categorical equations involving $\bord{n}{\alpha}$ and $\cp{n}$, which hold strictly about molecules \emph{in a directed complex} by Proposition \ref{prop:partial_omegacat}, hold up to unique isomorphism when they are interpreted as operations on molecules as individual directed complexes. For example, if $k < n$, there is a unique isomorphism between $\bord{n}{\alpha}(U_1 \cp{k} U_2)$ and $\bord{n}{\alpha}U_1 \cp{k} \bord{n}{\alpha}U_2$. 

In the rest of the article, we will deliberately mix the two perspectives, writing $U_1 \cp{k} U_2$ both for a decomposition inside a larger regular directed complex and for a pasting of individual molecules.
\end{dfn}

\begin{lem} \label{lem:boundary_isomorphism}
Let $U, V$ be molecules with spherical boundary. Suppose that $\bord{}{\alpha}U$ is isomorphic to $\bord{}{\alpha}V$ for all $\alpha \in \{+,-\}$. Then there is a unique isomorphism $\bord U \incliso \bord V$ restricting to isomorphisms $\bord{}{\alpha}U \incliso \bord{}{\alpha}V$.
\end{lem}
\begin{proof}
The boundaries of $U$ and $V$ are molecules, so there are unique isomorphisms $\imath_\alpha\colon \bord{}{\alpha}U \incliso \bord{}{\alpha}V$. They restrict to isomorphisms $\bord{}{\beta}(\bord{}{\alpha}U) \incliso \bord{}{\beta}(\bord{}{\alpha}V)$, which are also unique, so the restrictions of $\imath_+$ and $\imath_-$ to $\bord (\bord{}{\alpha}U)$ are equal after inclusion into $V$. 

Since $U$ has spherical boundary, this allows us to define a map $\bord U \to V$ restricting to $\imath_\alpha$ on $\bord{}{\alpha} U$, whose image is $\bord V$. Since $V$ has spherical boundary this map is injective.
\end{proof}

\begin{dfn}[Substitution]
Let $V$ and $W$ be $n$\nbd molecules with spherical boundary, let $U$ be an $n$\nbd molecule, and let $V \submol U$. Then $U \setminus (V \setminus \bord V)$ is a closed subset of $U$.

Suppose that $\bord{}{\alpha}V$ is isomorphic to $\bord{}{\alpha}W$ for all $\alpha \in \{+,-\}$. From Lemma \ref{lem:boundary_isomorphism} we obtain a unique isomorphism $\imath\colon \bord U \incliso \bord V$.

We define $U[W/V]$ to be the pushout
\begin{equation*}
\begin{tikzpicture}[baseline={([yshift=-.5ex]current bounding box.center)}]
	\node (0) at (0,1.5) {$\bord V$};
	\node (1) at (2.5,0) {$U[W/V]$};
	\node (2) at (0,0) {$W$};
	\node (3) at (2.5,1.5) {$U \setminus (V \setminus \bord V)$};
	\draw[1cinc] (0) to (3);
	\draw[1cincl] (0) to (2);
	\draw[1cinc] (2) to (1);
	\draw[1cincl] (3) to (1);
	\draw[edge] (1.6,0.2) to (1.6,0.7) to (2.3,0.7);
\end{tikzpicture}
\end{equation*}
in $\rdcpx$, and call it the \emph{substitution} of $W$ for $V \submol U$.
\end{dfn}

\begin{prop} \label{prop:substitution}
Suppose $U, V, W$ are $n$\nbd molecules such that $U[W/V]$ is defined. Then
\begin{enumerate}[label=(\alph*)]
	\item $U[W/V]$ is an $n$\nbd molecule with $W \submol U[W/V]$,
	\item there is a unique isomorphism $\bord U \incliso \bord U[W/V]$ restricting to isomorphisms $\bord{}{\alpha}U \incliso \bord{}{\alpha}U[W/V]$,
	\item if $V \submol V' \submol U$ for an $n$\nbd molecule $V'$, then $W \submol V'[W/V] \submol U[W/V]$, and
	\item if $U$ has spherical boundary, so does $U[W/V]$.
\end{enumerate}
\end{prop}
\begin{proof}
By induction on increasing $V'$ with $V \submol V' \submol U$. If $V' = V$, then $V'[W/V] = W$, which is an $n$\nbd molecule by assumption. The isomorphism $\bord V' \incliso \bord V'[W/V]$ also exists by assumption. 

Otherwise, $V'$ has a proper decomposition $V'_1 \cp{k} V'_2$, with $V \submol V'_i$. Without loss of generality, suppose $i = 1$. By the inductive hypothesis, $V'_1[W/V]$ is an $n$\nbd molecule with the same boundary as $V'_1$. Then $V'_1[W/V] \cp{k} V'_2$ is defined and uniquely isomorphic to $V'[W/V]$. Moreover $W \submol V'_1[W/V] \submol V'[W/V]$. 

The extension of the boundary isomorphism from $V'_1$ to $V_1$ can be derived from the axioms of $\omega$\nbd categories as discussed in \S \ref{dfn:molecule_pasting}. We conclude by the fact that chains of proper submolecules between $V$ and $U$ are finite.

Finally, the fact that $U[W/V]$ has spherical boundary when $U$ does is an immediate consequence of the isomorphism between the boundary of $U$ and the boundary of $U[W/V]$.
\end{proof}

\begin{lem} \label{lem:spherical_moves}
Let $U$ be an $n$\nbd molecule. Then
\begin{equation*}
	U = V_1 \cp{n-1} \ldots \cp{n-1} V_m
\end{equation*}
for some $n$\nbd molecules $V_1, \ldots, V_m$ such that, for $1 \leq i \leq m$,
\begin{enumerate}
	\item $V_i$ contains exactly one $n$\nbd atom $U_i$ and
	\item $\bord{}{\alpha}V_i$ is isomorphic to $\bord{}{-\alpha}V_i[\bord{}{\alpha}U_i/\bord{}{-\alpha}U_i]$.
\end{enumerate}
Moreover, if $\bord{}{\alpha}U$ has spherical boundary, so do all the $\bord{}{\alpha}V_i$.
\end{lem}
\begin{proof}
If $U$ contains a single $n$\nbd atom, then let $V_1 \eqdef U$; if $m > 1$, then Lemma \ref{lem:composition_form} provides a decomposition of the desired form.

If $U_i$ is the unique $n$\nbd atom contained in $V_i$, by Lemma \ref{lem:atom_boundary} $\bord{}{\alpha}U_i \submol \bord{}{\alpha}V_i$ for all $\alpha \in \{+,-\}$. The $\bord{}{\alpha}U_i$ have spherical boundary and $\bord{}{\beta}\bord{}{+}U_i = \bord{}{\beta}\bord{}{-}U_i$ for all $\beta \in \{+,-\}$, so $\bord{}{-\alpha}V_i[\bord{}{\alpha}U_i/\bord{}{-\alpha}U_i]$ is defined and has boundaries isomorphic to those of $\bord{}{-\alpha}V_i$. 

Every element outside of $\bord{}{\alpha}U_i \submol \bord{}{\alpha}V_i$ is not covered by any $n$\nbd dimensional element, so it belongs to $\bord (\bord{}{\alpha}V_i)$. Thus the inclusion 
\begin{equation*}
	\bord{}{\alpha}U_i \incl \bord{}{-\alpha}V_i[\bord{}{\alpha}U_i/\bord{}{-\alpha}U_i]
\end{equation*}
extends to an isomorphic inclusion $\bord{}{\alpha}V_i \incliso \bord{}{-\alpha}V_i[\bord{}{\alpha}U_i/\bord{}{-\alpha}U_i]$ via the identity of $\bord (\bord{}{\alpha}V_i)$ and $\bord (\bord{}{-\alpha}V_i)$. 

It follows from Proposition \ref{prop:substitution} that $\bord{}{+}V_i$ has spherical boundary if and only if $\bord{}{-}V_i$ has spherical boundary. Because $\bord{}{-}U = \bord{}{-}V_1$, $\bord{}{+}U = \bord{}{+}V_m$, and $\bord{}{+}V_i = \bord{}{-}V_{i+1}$ for $0 < i < m$, we conclude that if $\bord{}{\alpha}U$ has spherical boundary, so do all the $\bord{}{\alpha}V_i$. 
\end{proof}

\begin{dfn}[Pasting along a submolecule]
Let $U_1, U_2$ be $n$\nbd molecules and suppose that $\bord{}{-\alpha}U_1$ is isomorphic to a submolecule $V \submol \bord{}{\alpha}U_2$. There is a unique span of inclusions $V \incl U_1$ and $V \incl U_2$ whose images are, respectively, $\bord{}{-\alpha}U_1$ and $V \subseteq \bord{}{\alpha}U_2$. Take the pushout
\begin{equation*}
\begin{tikzpicture}[baseline={([yshift=-.5ex]current bounding box.center)}]
	\node (0) at (0,1.5) {$V$};
	\node (1) at (2.5,0) {$U_1 \cup U_2$};
	\node (2) at (0,0) {$U_1$};
	\node (3) at (2.5,1.5) {$U_2$};
	\draw[1cinc] (0) to (3);
	\draw[1cincl] (0) to (2);
	\draw[1cinc] (2) to (1);
	\draw[1cincl] (3) to (1);
	\draw[edge] (1.6,0.2) to (1.6,0.7) to (2.3,0.7);
\end{tikzpicture}
\end{equation*}
in $\rdcpx$. We claim that $U_1 \cup U_2$ is an $n$\nbd molecule with $U_1, U_2 \submol U_1 \cup U_2$. 

Without loss of generality, suppose $\alpha = -$. Note that $U_1 \cup U_2$ is isomorphic to $(U_1 \cup \bord{}{-}U_2) \cp{n-1} U_2$, so it suffices to show that $U_1 \cup \bord{}{-}U_2$ is an $n$\nbd molecule with $\bord{}{+}(U_1 \cup \bord{}{-}U_2) = \bord{}{-}U_2$ and $U_1 \submol U_1 \cup \bord{}{-}U_2$. This is proved by the same argument used for \cite[Lemma 14]{hadzihasanovic2018combinatorial}.

We call $U_1 \cup U_2$ the \emph{pasting of $U_1$ and $U_2$ along the submolecule} $V \submol \bord{}{\alpha}U_2$.
\end{dfn}

\begin{prop} \label{prop:spherical_pasting}
Let $U_1, U_2$ be $n$\nbd molecules such that the pasting $U_1 \cup U_2$ of $U_1$ and $U_2$ along $V \submol \bord{}{\alpha}U_2$ is defined. Then
\begin{enumerate}[label=(\alph*)]
	\item if $U_2$ has spherical boundary, so does $U_1 \cup U_2$, and
	\item if $U_1$ also has spherical boundary, then $\bord{}{\alpha}(U_1 \cup U_2)$ is isomorphic to $\bord{}{\alpha}U_2[\bord{}{\alpha}U_1/V]$.
\end{enumerate}
\end{prop}
\begin{proof}
Left to the reader.
\end{proof}

\begin{dfn}[$- \celto -$ construction]
Let $U, V$ be $n$\nbd molecules with spherical boundary such that $\bord{}{\alpha}U$ is isomorphic to $\bord{}{\alpha}V$ for all $\alpha \in \{+,-\}$. From Lemma \ref{lem:boundary_isomorphism} we obtain a unique isomorphism $\bord U \incliso \bord V$. 

Form the pushout $U \cup V$ of the span of inclusions $\bord U \incl U$, $\bord U \incl V$ whose images are $\bord U$ and $\bord V$, respectively. We define $U \celto V$ to be the oriented graded poset obtained from $U \cup V$ by adjoining a greatest element $\top$ with $\bord{}{-}\top \eqdef U$ and $\bord{}{+}\top \eqdef V$. Then $U \celto V$ is an $(n+1)$\nbd dimensional atom with spherical boundary.
\end{dfn}

\begin{dfn}[$\compos{-}$ construction]
Let $U$ be a molecule with spherical boundary. Then $\bord{}{-}U \celto \bord{}{+}U$ is defined, and we denote it by $\compos{U}$. 
\end{dfn}


\subsection{Gray products}

\begin{dfn}[Gray product]
Let $P, Q$ be oriented graded posets. The \emph{Gray product} $P \gray Q$ of $P$ and $Q$ is the cartesian product $P \times Q$ of their underlying posets --- a graded poset --- with the following orientation. Write $x \gray y$ for a generic element of $P \gray Q$. For all $x'$ covered by $x$ in $P$ and all $y'$ covered by $y$ in $Q$,
\begin{align*}
	o(x \gray y \to x' \gray y) & \eqdef o_P(x \to x'), \\
	o(x \gray y \to x \gray y') & \eqdef (-)^{\dmn{x}}o_Q(y \to y'),
\end{align*}
where $o_P$ and $o_Q$ are the orientations of $P$ and $Q$, respectively. 

Up to isomorphism, this defines an associative operation with unit $1$.
\end{dfn}

\begin{lem} \label{lem:gray_molecules}
Let $P, Q$ be directed complexes, $U \subseteq P$ an $n$\nbd molecule and $V \subseteq Q$ an $m$\nbd molecule. Then $U \gray V \subseteq P \gray Q$ is an $(n+m)$\nbd molecule. For all $k \in \mathbb{N}$ and $\alpha \in \{+,-\}$,
\begin{equation*}
	\bord{k}{\alpha}(U \gray V) = \bigcup_{i = 0}^k \bord{i}{\alpha}U \gray \bord{k-i}{(-)^i\alpha}V,
\end{equation*}
which is a molecule with $\bord{i}{\alpha}U \gray \bord{k-i}{(-)^i\alpha}V \submol \bord{k}{\alpha}(U \gray V)$.
\end{lem}
\begin{proof}
This is \cite[Theorem 7.4]{steiner1993algebra}.
\end{proof}

\begin{prop}
Let $U, V$ be molecules with spherical boundary. Then $U \gray V$ has spherical boundary.
\end{prop}
\begin{proof}
Let $n \eqdef \dmn{U}$ and $m \eqdef \dmn{V}$. For $k < n + m$, 
\begin{align*}
	\bord{k}{+}(U \gray V) \cap \bord{k}{-}(U \gray V) & = \bigcup_i \bigg(\bord{i}{+}U \gray \bord{k-i}{(-)^i}V\bigg) \cap \bigcup_j \bigg(\bord{j}{-}U \gray \bord{k-j}{-(-)^j}V\bigg) = \\
	& = \bigcup_{i,j} \bigg(\bord{i}{+}U \cap \bord{j}{-}U\bigg) \gray \bigg(\bord{k-i}{(-)^i}V \cap \bord{k-j}{-(-)^j}V\bigg)
\end{align*}
by elementary properties of unions, intersections, and cartesian products. We analyse separately different parts of this union.

\begin{itemize}
	\item Case $i < j$. Then $\bord{i}{+}U \cap \bord{j}{-}U = \bord{i}{+}U$ and $\bord{k-i}{(-)^i}V \cap \bord{k-j}{-(-)^j}V = \bord{k-j}{-(-)^j}V$, and also $\bord{k-j}{-(-)^j}V \subseteq \bord{k-i-1}{(-)^i}V$ for all $j > i+1$, so
	\begin{equation} \label{eq:union_j_gt_i}
		\bigcup_{i<j} \bigg(\bord{i}{+}U \cap \bord{j}{-}U\bigg) \gray \bigg(\bord{k-i}{(-)^i}V \cap \bord{k-j}{-(-)^j}V\bigg) = \bigcup_i \bord{i}{+}U \gray \bord{k-i-1}{(-)^i}V.
	\end{equation} 
	\item Case $i > j$. By a dual argument,
	\begin{equation} \label{eq:union_i_gt_j}
		\bigcup_{i>j} \bigg(\bord{i}{+}U \cap \bord{j}{-}U\bigg) \gray \bigg(\bord{k-i}{(-)^i}V \cap \bord{k-j}{-(-)^j}V\bigg) = \bigcup_j \bord{j}{-}U \gray \bord{k-j-1}{-(-)^j}V.
	\end{equation} 
	\item Case $i = j$. Then 
	\begin{equation*}
		\bord{i}{+}U \cap \bord{i}{-}U = \begin{cases} \bord{i-1}{} U & \text{if $i < n$}, \\
				U & \text{if $i \geq n$} \end{cases}
	\end{equation*}
	\begin{equation*}
		\bord{k-i}{(-)^i}V \cap \bord{k-i}{-(-)^i}V = \begin{cases} \bord{k-i-1}{} V & \text{if $i > k - m$}, \\
				V & \text{if $i \leq k-m$}, \end{cases}
	\end{equation*}
	because $U$ and $V$ have spherical boundary. Thus the union over $i = j$ is equal to
	\begin{equation*}
		\bigcup_{i \leq k-m} \bord{i-1}{}U \gray V \; \cup \;	\bigcup_{k-m < \\ i < \\ n} \bord{i-1}{}U \gray \bord{k-i-1}{}V \; \cup \; \bigcup_{i\geq n} U \gray \bord{k-i-1}{}V,
	\end{equation*}
	but every term of this union is included in one of (\ref{eq:union_j_gt_i}) or (\ref{eq:union_i_gt_j}).
\end{itemize}
We obtain 
\begin{equation*}
	\bord{k}{+}(U \gray V) \cap \bord{k}{-}(U \gray V) = \bigcup_i \bord{i}{+}U \gray \bord{k-i-1}{(-)^i}V \; \cup \; \bigcup_j \bord{j}{-}U \gray \bord{k-j-1}{-(-)^j}V
\end{equation*}
which by Lemma \ref{lem:gray_molecules} is equal to $\bord{k-1}{+}(U \gray V) \cup \bord{k-1}{-}(U \gray V)$. This proves that $U \gray V$ has spherical boundary.
\end{proof}

\begin{cor}
If $P$ and $Q$ are regular directed complexes, then $P \gray Q$ is a regular directed complex.
\end{cor}

\begin{dfn}
Let $f\colon P \to P'$ and $g\colon Q \to Q'$ be maps of regular directed complexes, and let $f \gray g\colon P \gray Q \to P' \gray Q'$ have the cartesian product of $f$ and $g$ as underlying function. By Lemma \ref{lem:gray_molecules} and the fact that $f, g$ are maps of oriented graded posets,
\begin{align*}
	\bord{k}{\alpha}(f \gray g(x \gray y)) & = \bord{k}{\alpha}(f(x) \gray g(y)) = \bigcup_i \bord{i}{\alpha}f(x) \gray \bord{k-i}{(-)^i\alpha}g(y) = \\
	& = \bigcup_i f(\bord{i}{\alpha}x) \gray g(\bord{k-i}{(-)^i\alpha}y) = \\
	& = f \gray g\bigg(\bigcup_i \bord{i}{\alpha}x \gray \bord{k-i}{(-)^i\alpha}y\bigg) = f \gray g(\bord{k}{\alpha}(x \gray y)).
\end{align*}
Thus $f \gray g$ is a map of oriented graded posets. We deduce that Gray products determine a monoidal structure on $\rdcpx$ such that the forgetful functor to $\pos$ with cartesian products is monoidal.
\end{dfn}

\begin{prop} \label{prop:gray_preserve}
Gray products preserve the initial object and pushouts of inclusions separately in each variable.
\end{prop}
\begin{proof}
By Proposition \ref{prop:rdcpxin_colimits} these colimits are created by the forgetful functor to $\pos$, and the statement is true of cartesian products of posets.
\end{proof}

\begin{dfn}[Directed cylinder] \label{dfn:cylinder}
Let $P$ be a regular directed complex. The \emph{directed cylinder} over $P$ is the Gray product $O^1 \gray P$. 

If $U$ is an $n$\nbd molecule in $P$, then $O^1 \otimes U$ is an $(n+1)$\nbd molecule with $\bord{0}{\alpha}(O^1 \gray U) = \{\undl{0}^\alpha\} \gray \bord{0}{\alpha}U$ and, for $0 < k \leq n$, 
\begin{align*}
	\bord{k}{-}(O^1 \gray U) & = \big(\ldots\big(\{\undl{0}^-\} \gray \bord{k}{-}U \cp{0} O^1 \gray \bord{0}{+}U\big)\cp{1}\ldots\big) \cp{k-1} O^1 \gray \bord{k-1}{+}U, \\
	\bord{k}{+}(O^1 \gray U) & = O^1 \gray \bord{k-1}{-}U \cp{k-1} \big(\ldots \cp{1}\big(O^1 \gray \bord{0}{-}U \cp{0}\{\undl{0}^+\} \gray \bord{k}{+}U\big)\ldots\big).
\end{align*}
If $U$ has a decomposition $U = U_1 \cp{k} U_2$, then $O^1 \gray U$ decomposes as
\begin{equation*}
	(\{\undl{0}^-\} \gray \bord{k+1}{-}U_1 \cup O^1 \gray U_2) \cp{k+1} (O^1 \gray U_1 \cup \{\undl{0}^+\} \gray \bord{k+1}{+}U_2),
\end{equation*}
where the two submolecules decompose as
\begin{align*}
	\big(\ldots\big(\{\undl{0}^-\} \gray \bord{k+1}{-}U_1 \cp{0} O^1 \gray \bord{0}{+}U_2\big)\cp{1} O^1 \gray \bord{1}{+}U_2\ldots\big) \cp{k} O^1 \gray U_2, \\
	O^1 \gray U_1 \cp{k} \big(\ldots O^1 \gray \bord{1}{-}U_1 \cp{1}\big(O^1 \gray \bord{0}{-}U_1 \cp{0}\{\undl{0}^+\} \gray \bord{k+1}{+}U_2\big)\ldots\big).
\end{align*}
\end{dfn}

\begin{dfn}
Let $P$ be a regular directed complex, $V \subseteq P$ a closed subset, and let $\sim_V$ be the equivalence relation on $O^1 \gray P$ defined by
\begin{equation*}
	\undl{0}^- \gray x \sim_V \undl{1} \gray x \sim_V \undl{0}^+ \gray x \text{ for all } x \in V.
\end{equation*}
The quotient $\slice{O^1 \gray P}{\sim_V}$ is graded. It inherits an orientation from $O^1 \gray P$ because $\undl{0}^- \gray y$ covers $\undl{0}^- \gray x$ with orientation $\alpha$ if and only if $\undl{0}^+ \gray y$ covers $\undl{0}^+ \gray x$ with orientation $\alpha$. This makes the quotient map $q\colon O^1 \gray P \surj \slice{O^1 \gray P}{\sim_V}$ a map of oriented graded posets.
\end{dfn}

\begin{lem}
The oriented graded poset $\slice{O^1 \gray P}{\sim_V}$ is a directed complex.
\end{lem}
\begin{proof}
Let $x \in \slice{O^1 \gray P}{\sim_V}$. Either $x = q(\undl{0}^\alpha \gray x')$ for some $x'$, in which case $\clos\{x\}$ is isomorphic to $\clos\{x'\} \subseteq P$, or $x = q(\undl{1} \gray x')$ for some $x' \in P \setminus V$. Then $\bord{k}{\alpha}x = q(\bord{k}{\alpha}(\undl{1} \gray x'))$, so it suffices to show that $q$ sends the molecules $\bord{k}{\alpha}(\undl{1} \gray x') = \bord{k}{\alpha}(O^1 \gray \clos\{x'\})$ to molecules. 

By inspection, the map $q$ preserves the decomposition of $\bord{k}{\alpha}(O^1 \gray \clos\{x'\})$ given in \S \ref{dfn:cylinder}, so it suffices that $q$ send its components to molecules. This is clear for the components of the form $\{\undl{0}^\alpha\} \gray \bord{k}{\alpha}x'$, so the problem is reduced to showing that $q(O^1 \gray \bord{i}{-\alpha}x')$ is a molecule for all $i < k$.

This follows from the general statement that $q(O^1 \gray U)$ is a molecule when $U \subseteq P$ is a molecule. This is immediate when $U$ is an atom. When $U$ has a proper decomposition $U_1 \cp{k} U_2$, by inspection $q$ preserves the decomposition of $O^1 \gray (U_1 \cp{k} U_2)$ given in \S \ref{dfn:cylinder}, whose components are all molecules by the inductive hypothesis. This completes the proof.
\end{proof}

\begin{prop} \label{prop:spherical_quotient}
Let $U$ be an $n$\nbd molecule with spherical boundary and let $V \subseteq \bord U$ be a closed subset. Then $\slice{O^1 \gray U}{\sim_V}$ is an $(n+1)$\nbd molecule with spherical boundary.
\end{prop}
\begin{proof}
By Lemma \ref{lem:gray_molecules}, $O^1 \gray U$ is an $(n+1)$\nbd molecule with 
\begin{equation*}
	\bord{k}{\alpha}(O^1 \gray U) = \{\undl{0}^\alpha\} \gray \bord{k}{\alpha}U \cup O^1 \gray \bord{k-1}{-\alpha}U
\end{equation*}
for each $k \leq n$. Let $\tilde{U} \eqdef \slice{O^1 \gray U}{\sim_V}$. We have $\bord{k}{\alpha}\tilde{U} = q(\bord{k}{\alpha}(O^1 \gray U))$, so $\bord{k}{+}\tilde{U} \cap \bord{k}{-}\tilde{U}$ is equal to
\begin{align*}
	& \left(q(\{\undl{0}^+\} \gray \bord{k}{+}U) \cap q(\{\undl{0}^-\} \gray \bord{k}{-}U)\right) \cup \left(q(\{\undl{0}^+\} \gray \bord{k}{+}U) \cap q(O^1 \gray \bord{k-1}{+}U)\right) \cup \\
	& \cup \left(q(O^1 \gray \bord{k-1}{-}U) \cap q(\{\undl{0}^-\} \gray \bord{k}{-}U)\right) \cup \left(q(O^1 \gray \bord{k-1}{-}U) \cap q(O^1 \gray \bord{k-1}{+}U)\right).
\end{align*}
The last three components may be rewritten as
\begin{equation} \label{eq:spherical_quotient}
	q(\{\undl{0}^+\} \gray \bord{k-1}{+}U) \cup q(\{\undl{0}^-\} \gray \bord{k-1}{-}U) \cup q(O^1 \gray \bord{k-2}{} U),
\end{equation}
where in the last term we used the fact that $U$ has spherical boundary. This is equal to $\bord{k-1}{+}\tilde{U} \cup \bord{k-1}{-}\tilde{U}$.

The first component is equal to
\begin{equation*} 
	q\big(O^1 \otimes (V \cap \bord{k}{+}U \cap \bord{k}{-}U)\big). 
\end{equation*}
When $k = n$, this is included in the first two components of (\ref{eq:spherical_quotient}) by definition of $q$ and the fact that $V \subseteq \bord U$. When $k < n$, this is included in the third component of (\ref{eq:spherical_quotient}) because $U$ has spherical boundary. 

In either case, $\bord{k}{+}\tilde{U} \cap \bord{k}{-}\tilde{U} = \bord{k-1}{}\tilde{U}$, that is, $\tilde{U}$ has spherical boundary.
\end{proof}

\begin{cor}
The directed complex $\slice{O^1 \gray P}{\sim_V}$ is regular.
\end{cor}
\begin{proof}
Let $x \in \slice{O^1 \gray P}{\sim_V}$. If $x = q(\undl{0}^\alpha \gray x')$ for some $x' \in P$, then $\clos\{x\}$ is isomorphic to $\clos\{x'\} \in P$, which has spherical boundary. If $x = q(\undl{1} \gray x')$ for some $x' \in P \setminus V$, then $\clos\{x\}$ is isomorphic to $\slice{O^1 \otimes \clos\{x'\}}{\sim_{V'}}$ where $V' \eqdef V \cap \clos\{x'\} \subseteq \bord x'$. This has spherical boundary by Proposition \ref{prop:spherical_quotient}.
\end{proof}

\begin{dfn}[$\infl{-}$ construction] 
Let $U$ be an $n$\nbd molecule with spherical boundary. We define 
\begin{equation*}
	\infl{U} \eqdef \slice{O^1 \gray U}{\sim_{\bord U}},
\end{equation*}
which is an $(n+1)$\nbd molecule with spherical boundary by Proposition \ref{prop:spherical_quotient}. 

If $!\colon O^1 \surj 1$ is the unique map to the terminal object, the natural map $O^1 \gray U \surj U$ obtained by composing $! \gray \idd{U}$ with the unique isomorphism $1 \gray U \incliso U$ descends to the quotient, factoring as
\begin{equation*}
\begin{tikzpicture}[baseline={([yshift=-.5ex]current bounding box.center)}]
	\node (0) at (-1.5,1.25) {$O^1 \gray U$};
	\node (1) at (0,0) {$\infl{U}$};
	\node (2) at (1.5,1.25) {$U$};
	\draw[1csurj] (0) to (2);
	\draw[1csurj] (0) to (1);
	\draw[1csurj] (1) to node[auto,swap,arlabel] {$\tau$} (2);
\end{tikzpicture}
\end{equation*}
for a unique surjective map $\tau$. This has the property that
\begin{equation*}
\begin{tikzpicture}[baseline={([yshift=-.5ex]current bounding box.center)}]
	\node (0) at (-1.5,1.25) {$U$};
	\node (1) at (0,0) {$\infl{U}$};
	\node (2) at (1.5,1.25) {$U$};
	\draw[1cinc] (0) to node[auto,arlabel] {$\idd{U}$} (2);
	\draw[1cinc] (0) to node[auto,swap,arlabel] {$\imath^{\alpha}$} (1);
	\draw[1csurj] (1) to node[auto,swap,arlabel] {$\tau$} (2);
\end{tikzpicture}
\end{equation*}
commutes for all $\alpha \in \{+,-\}$, where $\imath^\alpha$ is the isomorphic inclusion of $U$ into $\bord{}{\alpha}\infl{U}$. 
\end{dfn}

\begin{exm}
If $U$ is an atom, $\infl{U}$ is also an atom and is isomorphic to $U \celto U$.
\end{exm}

\begin{exm}
Let $O^0(U) \eqdef U$ and $O^n(U) \eqdef \infl{O^{n-1}(U)}$ for each $n > 0$. Then $O^n(1)$ is isomorphic to $O^n$. The maps $\imath^\alpha\colon U \incl \infl{U}$ and $\tau\colon \infl{U} \surj U$ specialise to the maps of $\cat{O}$ with the same name.
\end{exm}


\subsection{Suspension, joins, and duals}

\begin{dfn}[Suspension]
Let $P$ be an oriented graded poset. The \emph{suspension} of $P$ is the oriented graded poset $\Sigma P$ whose elements are $\{\Sigma x \mid x \in P \} + \{\bot^-, \bot^+\}$, with the partial order defined by
\begin{enumerate}
	\item $\bot^\alpha < \Sigma x$ for all $x \in P$, and
	\item $\Sigma x \leq \Sigma y$ if and only if $x \leq y$ in $P$,
\end{enumerate}
and the orientation 
\begin{equation*}
	o(\Sigma x \to \bot^\alpha) \eqdef \alpha, \quad o(\Sigma y \to \Sigma x) \eqdef o(y \to x)
\end{equation*}
for all $x, y \in P$ such that $y$ covers $x$.
\end{dfn}

\begin{lem} \label{lem:suspension}
Let $U$ be an $n$\nbd molecule. Then $\Sigma U$ is an $(n+1)$\nbd molecule with $\bord{0}{\alpha}\Sigma U = \{\bot^\alpha\}$ and $\bord{k}{\alpha}\Sigma U = \Sigma \bord{k-1}{\alpha} U$ for all $k > 0$. If $U$ has spherical boundary, so does $\Sigma U$.
\end{lem}
\begin{proof}
Left to the reader.
\end{proof}

\begin{dfn}
Let $f\colon P \to Q$ be a map of regular directed complexes. There is a function $\Sigma f\colon \Sigma P \to \Sigma Q$ defined by $\bot^\alpha \mapsto \bot^\alpha$ and $\Sigma x \mapsto \Sigma (f(x))$ for all $x \in P$.

By Lemma \ref{lem:suspension}, $\Sigma P$ and $\Sigma Q$ are regular directed complexes. For all $n \in \mathbb{N}$,
\begin{equation*}
	\bord{n}{\alpha} \Sigma f(\bot^\beta) = \bord{n}{\alpha} \bot^\beta = \{\bot^\beta\} = \Sigma f(\bord{n}{\alpha} \bot^\beta), 
\end{equation*}
and for all $x \in P$, if $n > 0$,
\begin{equation*}
	\bord{n}{\alpha} \Sigma f(\Sigma x) = \bord{n}{\alpha} \Sigma (f(x)) = \Sigma \bord{n-1}{\alpha} f(x) = \Sigma (f(\bord{n-1}{\alpha} x)) = \Sigma f(\bord{n}{\alpha} \Sigma x),
\end{equation*}
while $\bord{0}{\alpha} \Sigma f(\Sigma x) = \{\bot^\alpha\} = \Sigma f(\bord{0}{\alpha} \Sigma x)$. 

Thus $\Sigma f$ is a map of regular directed complexes. The assignment $f \mapsto \Sigma f$ respects composition and identities, so it defines an endofunctor $\Sigma$ on $\rdcpx$. 
\end{dfn}

\begin{dfn}[Join] 
If $P$ is an oriented graded poset, extend the orientation of $P$ to $P_\bot$ by $o(x \to \bot) \eqdef +$ for all 0-dimensional $x \in P$.

Let $P, Q$ be oriented graded posets. The \emph{join} $P \join Q$ of $P$ and $Q$ is the unique oriented graded poset such that $(P \join Q)_\bot$ is isomorphic to $P_\bot \gray Q_\bot$. Up to isomorphism, the join defines an associative operation with unit $\emptyset$.

Elements of $P \join Q$ have one of the three forms
\begin{itemize}
	\item $x$ corresponding to $x \gray \bot$ for some $x \in P$,
	\item $y$ corresponding to $\bot \gray y$ for some $y \in Q$, or
	\item $x \join y$ corresponding to $x \gray y$ for some $x \in P$ and $y \in Q$.
\end{itemize}
The first two cases determine inclusions $P, Q \incl P \join Q$. If $U \subseteq P$ and $V \subseteq Q$ are closed subsets, then $U \join V$ is a closed subset of $P \join Q$.
\end{dfn}

\begin{dfn}
There is an injective function (not a map of oriented graded posets) $P_\bot \to \Sigma P$ defined by $\bot \mapsto \bot^+$ and $x \mapsto \Sigma x$ for all $x \in P$. This induces an injective function $j\colon P \join Q \to \Sigma P \gray \Sigma Q$. 
\end{dfn}

\begin{prop}
Let $P, Q$ be directed complexes, $U \subseteq P$ an $n$\nbd molecule and $V \subseteq Q$ an $m$\nbd molecule. Then $U \join V \subseteq P \join Q$ is an $(n+m+1)$\nbd molecule. For all $k \in \mathbb{N}$, $\bord{k}{\alpha}(U \join V)$ is a molecule with
\begin{equation*}
	\bord{k}{-}(U \join V) = 
	\begin{cases} \bord{k}{-}U \cup \displaystyle\bigcup_{i=1}^k \bord{i-1}{-}U \join \bord{k-i}{-(-)^i} V & \text{if $k$ is even}, \\ 
	\displaystyle\bigcup_{i=1}^k \bord{i-1}{-}U \join \bord{k-i}{-(-)^i} V & \text{if $k$ is odd}, \end{cases}
\end{equation*}
\begin{equation*}
	\bord{k}{+}(U \join V) = 
	\begin{cases} \bord{k}{+}V \cup \displaystyle\bigcup_{i=1}^k \bord{i-1}{+}U \join \bord{k-i}{(-)^i} V & \text{if $k$ is even}, \\ 
	\bord{k}{+}U \cup \bord{k}{+}V \cup \displaystyle\bigcup_{i=1}^k \bord{i-1}{+}U \join \bord{k-i}{(-)^i} V & \text{if $k$ is odd}. \end{cases}
\end{equation*}
If $U$ and $V$ have spherical boundary, so does $U \join V$.
\end{prop}
\begin{proof}
In \cite[Proposition 7.8]{steiner1993algebra}, it is proved that if $W \subseteq \Sigma P \gray \Sigma Q$ is an $(n+1)$\nbd molecule not contained in $\{\bot^-\} \gray Q \cup  P \gray \{\bot^-\}$, then $\invrs{j}(W)$ is an $n$\nbd molecule and $\bord{k}{\alpha}\invrs{j}(W) = \invrs{j}(\bord{k+1}{\alpha}W)$. But
\begin{equation*}
	\bord{k}{\alpha}(U \join V) = \bord{k}{\alpha}\invrs{j}(\Sigma U \gray \Sigma V) = \invrs{j}(\bord{k+1}{\alpha}(\Sigma U \gray \Sigma V)),
\end{equation*}
which suffices to prove that $U \join V$ and all its boundaries are molecules. By Lemma \ref{lem:gray_molecules} $\bord{k}{\alpha}(U \join V)$ is then equal to
\begin{equation*}
	\invrs{j}\bigg(\bigcup_{i = 0}^{k+1} \bord{i}{\alpha}\Sigma U \gray \bord{k+1-i}{(-)^i\alpha}\Sigma V\bigg)
\end{equation*}
which we rewrite as
\begin{equation*}
	\invrs{j}\bigg(\{\bot^\alpha\} \gray \Sigma \bord{k}{\alpha}V \cup \bigcup_{i = 1}^{k} \Sigma \bord{i-1}{\alpha} U \gray \Sigma \bord{k-i}{(-)^i\alpha} V \cup \Sigma \bord{k}{\alpha} U \gray \{\bot^{(-)^{k+1}\alpha}\} \bigg).
\end{equation*}
Then
\begin{equation*}
	\invrs{j} \bigg( \bigcup_{i = 1}^{k} \Sigma \bord{i-1}{\alpha} U \gray \Sigma \bord{k-i}{(-)^i\alpha} V \bigg) = \bigcup_{i=1}^k \bord{i-1}{\alpha}U \join \bord{k-i}{(-)^i\alpha} V, 
\end{equation*}
while 
\begin{equation*}
	\invrs{j}(\{\bot^\alpha\} \gray \Sigma \bord{k}{\alpha}V) = \begin{cases} \emptyset & \text{if $\alpha = -$}, \\
	\bord{k}{+}V & \text{if $\alpha = +$}, \end{cases}
\end{equation*}
\begin{equation*}
	\invrs{j}(\Sigma \bord{k}{\alpha} U \gray \{\bot^{(-)^{k+1}\alpha}\}) = \begin{cases} \emptyset & \text{if $\alpha = (-)^k$}, \\
	\bord{k}{\alpha}U & \text{if $\alpha = (-)^{k+1}$}. \end{cases}
\end{equation*}
Finally, if $U$ and $V$ have spherical boundary, so does $\Sigma U \gray \Sigma V$. Then $U \join V$ has spherical boundary by the compatibility of $\invrs{j}$ with boundaries and intersections.
\end{proof}

\begin{cor}
If $P$ and $Q$ are regular directed complexes, then $P \join Q$ is a regular directed complex.
\end{cor}

\begin{dfn}
Let $f\colon P \to P'$ and $g\colon Q \to Q'$ be maps of regular directed complexes. Then $\Sigma f \gray \Sigma g$ sends the image of $j\colon P \join Q \to \Sigma P \gray \Sigma Q$ to the image of $j\colon P' \join Q' \to \Sigma P' \gray \Sigma Q'$. Because $j$ is injective, it has a partial inverse defined on its image, and it makes sense to define
\begin{equation*}
	f \join g (z) \eqdef \invrs{j}(\Sigma f \gray \Sigma g(j(z)))
\end{equation*}
for each $z \in P \join Q$. Then
\begin{align*}
	f \join g (\bord{k}{\alpha}z) & = \invrs{j}((\Sigma f \gray \Sigma g)\clos j(\bord{k}{\alpha}z)) = \invrs{j}((\Sigma f \gray \Sigma g)\bord{k+1}{\alpha}j(z)) = \\
	 & = \invrs{j}(\bord{k+1}{\alpha}(\Sigma f \gray \Sigma g)j(z)) = \bord{k}{\alpha}(f\join g)(z)
\end{align*}
so $f \join g$ is a map of oriented graded posets. The assignment $(f, g) \mapsto f \join g$ is functorial because $(f, g) \mapsto \Sigma f \gray \Sigma g$ is. We deduce that joins determine a second monoidal structure on $\rdcpx$. 
\end{dfn}

\begin{prop}
Joins preserve pushouts of inclusions separately in each variable.
\end{prop}
\begin{proof}
The endofunctor $\Sigma$ preserves pushout diagrams of inclusions. The statement then follows from Proposition \ref{prop:gray_preserve} by the definition of joins.
\end{proof}

\begin{rmk}
On the other hand, $\Sigma$ does not preserve the initial object and neither does the join operation.
\end{rmk}

\begin{dfn}[Duals] Let $P$ be an oriented graded poset and $J \subseteq \mathbb{N} \setminus \{0\}$. The \emph{$J$\nbd dual} $\oppn{J}{P}$ of $P$ is the oriented graded poset with the same underlying poset as $P$ and the orientation $o'$ defined by
\begin{equation*}
o'(y \to x) := \begin{cases}
		-o(y \to x) & \text{if $\dmn{y} \in J$}, \\
		o(y \to x) & \text{if $\dmn{y} \not\in J$},
	\end{cases}
\end{equation*}
for all $x, y \in P$ such that $y$ covers $x$. 

We write $\opp{P}$, $\coo{P}$, and $\oppall{P}$ in the cases $J = \{2n-1\}_{n>0}$, $J = \{2n\}_{n>0}$, and $J = \mathbb{N} \setminus \{0\}$, respectively. For all $n > 0$ we write $\oppn{n}{P}$ for $\oppn{\{n\}}{P}$.

If $P$ is a regular directed complex, so is $\oppn{J}{P}$, and if $f\colon P \to Q$ is a map of regular directed complexes, so is $\oppn{J}{f}\colon \oppn{J}{P} \to \oppn{J}{Q}$ with the same underlying function. The proof is by induction on molecules and left to the reader. Thus $\oppn{J}{}$ determines an endofunctor on $\rdcpx$.
\end{dfn}

\begin{prop}
Let $P, Q$ be regular directed complexes. Then,
\begin{enumerate}[label=(\alph*)]
	\item $x \gray y \mapsto y \gray x$ is an isomorphism between $\opp{(P \gray Q)}$ and $\opp{Q} \gray \opp{P}$ and between $\coo{(P \gray Q)}$ and $\coo{Q} \gray \coo{P}$, and
	\item $x \join y \mapsto y \join x$ is an isomorphism between $\opp{(P \join Q)}$ and $\opp{Q} \join \opp{P}$.
\end{enumerate}
Consequently $x \gray y \mapsto x \gray y$ is an isomorphism between $\oppall{(P \gray Q)}$ and $\oppall{P} \gray \oppall{Q}$. These isomorphisms are natural for maps of regular directed complexes.
\end{prop}
\begin{proof}
Identical to the proof of \cite[Proposition 15]{hadzihasanovic2018combinatorial}.
\end{proof}

\begin{dfn}[Reverse map]
Let $U, V$ be atoms with $\dmn{U} > \dmn{V}$ and let $p\colon U \surj V$ be a surjective map. Let $n \eqdef \dmn{U}$. The \emph{reverse} of $p$ is the map $\rev{p}\colon \oppn{n}{U} \surj V$ defined by
\begin{enumerate}
	\item $\restr{\rev{p}}{\bord{}{\alpha}\oppn{n}{U}} \eqdef \restr{p}{\bord{}{-\alpha}U}$ and
	\item $\rev{p}$ sends the greatest element of $\oppn{n}{U}$ to the greatest element of $V$.
\end{enumerate}
This is well-defined because $p(\bord{}{+}U) = p(\bord{}{-}U) = V$. 

More in general, suppose that $U$ and $V$ are molecules. For all $x \in U$ with $\dmn{x} = n$, we have $\dmn{p(x)} < n$, so $\rev{\restr{p}{\clos\{x\}}}\colon \oppn{n}{\clos\{x\}} \surj \clos\{p(x)\}$ is defined. The reverse of $p$ is then the map $\rev{p}\colon \oppn{n}{U} \surj V$ defined by
\begin{enumerate}
	\item $\restr{\rev{p}}{\bord{}{\alpha}\oppn{n}{U}} \eqdef \restr{p}{\bord{}{-\alpha}U}$ and
	\item $\restr{\rev{p}}{\oppn{n}{\clos\{x\}}} \eqdef \rev{\restr{p}{\clos\{x\}}}$ for all $x \in U$ with $\dmn{x} = n$.
\end{enumerate}
Note that $\rev{\rev{p}} = p$.
\end{dfn}

\section{Classes of molecules} \label{sec:classes}

\subsection{Algebraic classes}

\begin{dfn}[Class of molecules] 
A \emph{class of molecules} $\cls{C}$ is a class of objects of $\rmolin$, closed under isomorphism, such that
\begin{enumerate}
	\item if $U$ is in $\cls{C}$, then $\bord{}{+} U$ and $\bord{}{-} U$ are in $\cls{C}$, and
	\item if $U$ is in $\cls{C}$ and $x \in U$, then $\clos\{x\}$ is in $\cls{C}$.
\end{enumerate}
\end{dfn}

\begin{exm}
The following are classes of molecules:
\begin{itemize}
	\item the class $\cls{R}$ of all (regular) molecules;
	\item the class $\cls{S}$ of molecules with spherical boundary; 
	\item the class $\cls{K}$ of constructible molecules \cite{hadzihasanovic2018combinatorial}.
\end{itemize}
\end{exm}

\begin{exm}
Given a regular directed complex $P$, let $\hasseo{P}$ be the directed graph obtained from $\hasse{P}$ by reversing all the edges labelled $-$. We say that $P$ is \emph{totally loop-free} \cite[Definition 2.14]{steiner1993algebra} if $\hasseo{P}$ is acyclic as a directed graph. Totally loop-free molecules form a class $\cls{LF}$: if $U$ is totally loop-free, so is any closed subset of $U$.
\end{exm}

\begin{exm}
Let $\cls{A}$ be a class of atoms with the property that, if $U \in \cls{A}$ and $x \in U$, then $\clos\{x\} \in \cls{A}$. There is a class of molecules $\cls{C}_\cls{A}$ defined by the property that $U \in \cls{C}_\cls{A}$ if and only if $\clos\{x\} \in \cls{A}$ for all $x \in U$. 

Notice that $\cls{R} = \cls{C}_\cls{A}$ for the class of atoms with spherical boundary.
\end{exm}

\begin{exm}
The class $\cls{O}$ of all globes is a class of molecules. Due to the requirement that a class be closed under boundary operators, this is the largest class of molecules whose members are all atoms.
\end{exm}

\begin{dfn}[$\cls{C}$-Directed complex]
Let $\cls{C}$ be a class of molecules. A \emph{$\cls{C}$\nbd directed complex} is a regular directed complex $P$ with the property that $\clos\{x\}$ is in $\cls{C}$ for all $x \in P$.

We write $\dcpx$ and $\dcpxin$ for the full subcategories of $\rdcpx$ and $\rdcpxin$, respectively, on the $\cls{C}$\nbd directed complexes.
\end{dfn}

\begin{dfn}[$\cls{C}$-Functors and subdivisions]
Let $P$, $Q$ be two $\cls{C}$\nbd directed complexes. The closed subsets of $P$ form a bounded lattice $\closub{P}$. A \emph{$\cls{C}$\nbd functor} $f\colon P \cfun Q$ is a function $f\colon \closub{P} \to \closub{Q}$ such that
\begin{enumerate}
	\item $f$ preserves all unions and binary intersections,
	\item $\bord{n}{\alpha}f(\clos\{x\}) = f(\bord{n}{\alpha}x)$, and
	\item $f(\clos\{x\})$ is a $\cls{C}$\nbd molecule
\end{enumerate}
for all $x \in P$, $n \in \mathbb{N}$, and $\alpha \in \{+,-\}$. 

A \emph{subdivision} is a $\cls{C}$\nbd functor $f\colon P \cfun Q$ such that $f(P) = Q$, that is, a $\cls{C}$\nbd functor that preserves all intersections.
\end{dfn}

\begin{rmk}
A $\cls{C}$\nbd functor is completely determined by what it does on atoms: 
\begin{equation*}
	f(U) = f(\bigcup_{x \in U} \clos\{x\}) = \bigcup_{x \in U} f(\clos\{x\}).
\end{equation*}
\end{rmk}

\begin{dfn} \label{dfn:inclusion_cfunctor}
Every inclusion $\imath\colon P \incl Q$ of $\cls{C}$\nbd directed complexes induces a $\cls{C}$\nbd functor defined on closed subsets by $U \mapsto \imath(U)$. This is well-defined because $\imath$ is closed, and it preserves unions and binary intersections as the direct image of an injective map. The other two conditions follow from the definition of map of $\cls{C}$\nbd directed complexes. This $\cls{C}$\nbd functor is a subdivision if and only if $\imath$ is an isomorphism.
\end{dfn}

\begin{exm} 
Let $U$ be a $\cls{C}$\nbd molecule and $x \in U$ with $n \eqdef \dmn{x} = \dmn{U}$. Suppose $V$ is a $\cls{C}$\nbd molecule with $\bord{}{\alpha} V$ isomorphic to $\bord{}{\alpha} x$ for all $\alpha \in \{+,-\}$. Then $U[V/\clos\{x\}]$ is defined, and there is a subdivision $U \cfun U[V/\clos\{x\}]$ defined by 
\begin{equation*}
	\clos\{x\} \mapsto V, \quad \quad \clos\{y\} \mapsto \clos\{y\} \text{ if } y \neq x. 
\end{equation*}
As a special case, if $U$ has spherical boundary and $\compos{U}$ is a $\cls{C}$\nbd molecule, there is a unique subdivision $\compos{U} \cfun U$.
\end{exm}

\begin{exm} \label{exm:unique_subdivision}
Let $U$ be a $\cls{C}$\nbd molecule with spherical boundary, $n \eqdef \dmn{U}$. If the $n$\nbd globe $O^n$ is a $\cls{C}$\nbd molecule, there is a unique subdivision $O^n \cfun U$, defined by
\begin{equation*}
	\clos\{\undl{n}\} \mapsto U, \quad \quad \clos\{\undl{k}^\alpha\} \mapsto \bord{k}{\alpha}U
\end{equation*}
for each $k < n$ and $\alpha \in \{+,-\}$.
\end{exm}

\begin{prop} \label{prop:cfun_preserve_molecules}
Let $f\colon P \cfun Q$ be a $\cls{C}$\nbd functor and $U \subseteq P$ a molecule. Then
\begin{enumerate}[label=(\alph*)]
	\item $f(U)$ is a molecule with $\dmn{U} = \dmn{f(U)}$,
	\item $\bord{}{\alpha}f(U) = f(\bord{}{\alpha}U)$, and
	\item if $U$ decomposes as $U_1 \cp{k} U_2$, then $f(U)$ decomposes as $f(U_1) \cp{k} f(U_2)$.
\end{enumerate}
\end{prop}
\begin{proof}
By induction on $n \eqdef \dmn{U}$ and proper submolecules. First, notice that $\bord{n}{\alpha}f(\clos\{x\}) = f(\bord{n}{\alpha}x) = f(\clos\{x\})$ for all $x \in P$ with $\dmn{x} = n$, so by Lemma \ref{lem:minimaldim} the dimension of $f(\clos\{x\})$ is at most $n$. If $n = 0$, since $f(\clos\{x\})$ is a molecule, it must be a 0-molecule. 

Let $n > 0$. If $U = \clos\{x\}$, then by definition $f(U)$ is a molecule with $\bord{n-1}{\alpha}f(U) = f(\bord{}{\alpha}U)$. Suppose that $\dmn{f(U)} < n$. By the inductive hypothesis, $f(\bord{}{+}U)$ and $f(\bord{}{-}U)$ are $(n-1)$\nbd molecules, and since
\begin{equation*}
	f(U) = \bord{n-1}{\alpha}f(U) = f(\bord{}{\alpha}U)
\end{equation*}
they are equal to each other and to $f(U)$. Then 
\begin{equation*}
	f(U) = f(\bord{}{+}U) \cap f(\bord{}{-}U) = f(\bord{}{+}U \cap \bord{}{-}U) = f(\bord{n-2}{}U)
\end{equation*}
because $x$ has spherical boundary, and by the inductive hypothesis the last term is $(n-2)$\nbd dimensional: a contradiction. So $\dmn{f(U)} = n$.

Finally, suppose that $U$ has a proper decomposition $U_1 \cp{k} U_2$. By the inductive hypothesis, $f(U_1)$ and $f(U_2)$ are molecules of the same dimension as $U_1$ and $U_2$, and
\begin{equation*}
	f(U_1) \cap f(U_2) = f(U_1 \cap U_2) = f(\bord{k}{+}U_1) = f(\bord{k}{-}U_2)
\end{equation*}
with the last two terms equal to $\bord{k}{+}f(U_1)$ and $\bord{k}{-}f(U_2)$, respectively. Thus $f(U_1) \cp{k} f(U_2)$ is defined and equal to $f(U_1 \cp{k} U_2)$. 
\end{proof}

\begin{cor} \label{cor:cfunctor_functor}
A $\cls{C}$\nbd functor $f\colon P \cfun Q$ induces a functor of partial $\omega$\nbd categories $\mol{}{P} \to \mol{}{Q}$, defined on molecules by $U \mapsto f(U)$.
\end{cor}
\begin{proof}
Immediate.
\end{proof}

\begin{cor} \label{cor:functor_preserve_spherical}
Let $f\colon P \cfun Q$ be a $\cls{C}$\nbd functor and $U \subseteq P$ a molecule. If $U$ has spherical boundary, so does $f(U)$.
\end{cor}
\begin{proof}
Follows from preservation of molecules, boundaries, and binary intersections.
\end{proof}

\begin{cor}
The following coincide:
\begin{enumerate}
	\item $\cls{C}$\nbd directed complexes and $(\cls{C} \cap \cls{S})$\nbd directed complexes, and
	\item $\cls{C}$\nbd functors and $(\cls{C} \cap \cls{S})$\nbd functors.
\end{enumerate}
\end{cor}
\begin{proof}
Regular atoms and atoms with spherical boundary coincide, so a regular directed complex is always a $\cls{S}$\nbd directed complex, and by Corollary \ref{cor:functor_preserve_spherical} every $\cls{C}$\nbd functor sends atoms to molecules with spherical boundary.
\end{proof}

\begin{lem} \label{lem:functor_preserve_dimension}
Let $f\colon P \cfun Q$ be a $\cls{C}$\nbd functor and $U \subseteq P$ a closed subset. Then $\dmn{U} = \dmn{f(U)}$.
\end{lem}
\begin{proof}
We have that $\dmn{f(U)}$ is equal to
\begin{equation*}
	\mathrm{max}\{\dmn{y} \mid y \in \bigcup_{x \in U} f(\clos\{x\})\} = \mathrm{max}\{\dmn{f(\clos\{x\})} \mid x \in U\},
\end{equation*}
and the latter is equal to $\dmn{U}$ by Proposition \ref{prop:cfun_preserve_molecules}.
\end{proof}

\begin{prop}
Let $f\colon P \cfun Q$ be a $\cls{C}$\nbd functor. Then $f$ is injective as a function $\closub{P} \to \closub{Q}$.
\end{prop}
\begin{proof}
Suppose that $f(U) = f(V)$ for a pair of closed subsets $U, V \subseteq P$. We may assume that $U \subseteq V$; otherwise, $f(U \cap V) = f(U) \cap f(V) = f(U)$ and we can replace $U$ with $U \cap V$. 

Now, $V$ is equal to the closure of its set of maximal elements. Suppose that there is $x \in V$ which is maximal and not in $U$, and let $n \eqdef \dmn{x}$. Then $\clos\{x\} \cap U$ is closed and must have dimension $k < n$, so $f(\clos\{x\} \cap U)$ has dimension $k < n$ by Lemma \ref{lem:functor_preserve_dimension}. But 
\begin{equation*}
	f(\clos\{x\} \cap U) = f(\clos\{x\}) \cap f(U) = f(\clos\{x\}) \cap f(V) = f(\clos\{x\})
\end{equation*}
because $\clos\{x\} \subseteq V$, and $\dmn{f(\clos\{x\})} = \dmn{x} = n$, a contradiction.
\end{proof}

\begin{rmk}
Let $f\colon P \cfun Q$ and $g\colon Q \cfun R$ be two $\cls{C}$\nbd functors. The composite $f;g\colon \closub{P} \to \closub{R}$ may not determine a $\cls{C}$\nbd functor: while Proposition \ref{prop:cfun_preserve_molecules} guarantees that $g(f(\clos\{x\}))$ is a molecule for all $x \in P$, there is no guarantee that it is a $\cls{C}$\nbd molecule. The following is a sufficient condition for $\cls{C}$\nbd functors to compose.
\end{rmk}

\begin{dfn}[Algebraic class of molecules]
A class of molecules $\cls{C}$ is \emph{algebraic} if, for all $\cls{C}$\nbd functors $f\colon P \cfun Q$ and $\cls{C}$\nbd molecules $U$ in $P$, $f(U)$ is a $\cls{C}$\nbd molecule in $Q$.
\end{dfn}

\begin{exm}
The class $\cls{S}$ of molecules with spherical boundary is algebraic by Corollary \ref{cor:functor_preserve_spherical}.
\end{exm}

\begin{exm}
The class $\cls{K}$ of constructible molecules is algebraic. It suffices to observe that if a $\cls{K}$\nbd functor sends constructible atoms to constructible molecules, then it preserves the constructible submolecule relation, whose definition involves only binary unions, binary intersections, and boundary operators.
\end{exm}

\begin{exm}
By \cite[Theorem 2.18]{steiner1993algebra}, if the atoms of a directed complex $P$ are totally loop-free, then all molecules are totally loop-free. It follows that the class $\cls{LF}$ of totally loop-free molecules is algebraic. 
\end{exm}

\begin{exm}
If $\cls{A}$ is a class of atoms, then every molecule in a $\cls{C}_\cls{A}$\nbd directed complex is a $\cls{C}_\cls{A}$\nbd molecule, so $\cls{C}_\cls{A}$ is algebraic.
\end{exm}

\begin{dfn}
Let $\cls{C}$ be an algebraic class of molecules. Two $\cls{C}$\nbd functors $f\colon P \cfun Q$ and $g\colon Q \cfun R$ compose to a $\cls{C}$\nbd functor $f;g\colon P \cfun R$. This composition is associative and unital.

If $\cls{C}$ is algebraic, we let $\dcpxfun$ denote the category of $\cls{C}$\nbd molecules and $\cls{C}$\nbd functors. By \S\ref{dfn:inclusion_cfunctor} we identify $\dcpxin$ with a subcategory of $\dcpxfun$.
\end{dfn}

\begin{prop} 
Let $\cls{C}$ be an algebraic class of molecules. Every $\cls{C}$\nbd functor $f\colon P \cfun Q$ of $\cls{C}$\nbd molecules factors as a subdivision $P \cfun \widehat{P}$ followed by an inclusion $\widehat{P} \incl Q$. This factorisation is unique up to isomorphism.
\end{prop}
\begin{proof}
The image of $f\colon \closub{P} \to \closub{Q}$ is a subset of $\closub{f(P)}$: factorising $f$ through $\closub{f(P)}$ produces a subdivision $f\colon P \cfun f(P)$, which followed by the subset inclusion of $f(P)$ into $Q$ produces one factorisation. 

Let $\widehat{f}\colon P \cfun \widehat{P}$, $\imath\colon \widehat{P} \incl Q$ be another such factorisation. We have $\widehat{f}(P) = \widehat{P}$ because $\widehat{f}$ is a subdivision, hence $\imath(\widehat{P}) = \imath \widehat{f}(P) = f(P)$. This determines an isomorphism $\widehat{P} \incliso f(P)$ which makes both triangles commute. 
\end{proof}

\begin{prop} \label{prop:fun_preserve_colimits}
The inclusion $\dcpxin \incl \dcpxfun$ preserves pushouts and the initial object.
\end{prop}
\begin{proof}
Straightforward.
\end{proof}

\begin{rmk}
Both $\dcpx$ and $\dcpxfun$ are extensions of $\dcpxin$ in which the inclusions form the right class of an orthogonal factorisation system: the left class is surjective maps in $\dcpx$ and subdivisions in $\dcpxfun$. 
\end{rmk}


\subsection{Convenient classes}

\begin{dfn}[Convenient class of molecules] Let $\cls{C} \subseteq \cls{S}$ be an algebraic class of molecules with spherical boundary. We say that $\cls{C}$ is a \emph{convenient class of molecules} if it satisfies the following axioms:
\begin{enumerate}
	\item $\cls{C}$ contains 1; \label{ax:nontrivial}
	\item if $U \in \cls{C}$ and $J \subseteq \mathbb{N} \setminus \{0\}$, then $\oppn{J}{U} \in \cls{C}$; \label{ax:dual}
	\item if $U, V \in \cls{C}$ and $U \celto V$ is defined, then $U \celto V \in \cls{C}$; \label{ax:shift}
	\item if $U_1, U_2 \in \cls{C}$ and the pasting $U_1 \cup U_2$ along $V \submol \bord{}{\alpha}U_2$ is defined, then $U_1 \cup U_2 \in \cls{C}$; \label{ax:submolecule}
	\item if $U \in \cls{C}$ and $V \subseteq \bord U$ is a closed subset, then $\slice{O^1 \gray U}{\sim_V} \in \cls{C}$; \label{ax:inflate}
	\item if $U, V \in \cls{C}$, then $U \gray V \in \cls{C}$ and $U \join V \in \cls{C}$. \label{ax:grayjoin}
\end{enumerate}
\end{dfn}

\begin{exm}
By the results of the previous sections, the class $\cls{S}$ of molecules with spherical boundary is convenient.

There ought to be also a minimal convenient class $\cls{C}_0$ obtained as the closure of $\{1\}$ under all the listed operations, together with $\cls{C}_0$\nbd subdivisions and the extraction of boundaries and atoms. At present, we do not have other examples, nor molecules that separate the two classes, so it cannot be excluded that $\cls{S}$ is equal to $\cls{C}_0$.
\end{exm}

\begin{exm}
The class $\cls{LF}$ of totally loop-free molecules does not satisfy axiom \ref{ax:shift} nor axiom \ref{ax:submolecule}. The original definition of diagrammatic sets in \cite{kapranov1991infty} was given relative to a class of totally loop-free shapes which is plausibly equivalent to $\cls{LF}$, but the authors seemed to assume that both \ref{ax:shift} and \ref{ax:submolecule} hold. This was noted as a flaw by Simon Henry who has given counterexamples to both axioms in \cite[Discussion A.2]{henry2019non}. 
\end{exm}

\begin{exm} \label{exm:class_constructible}
The class $\cls{K}$ of constructible molecules does not satisfy axiom \ref{ax:inflate}. Consider the 3-atom $U$ associated to the pasting diagram
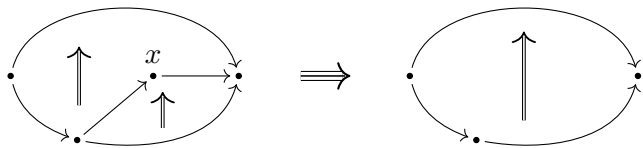
\begin{equation*}
\begin{tikzpicture}[baseline={([yshift=-.5ex]current bounding box.center)}]
\begin{scope}
	\node[0c] (0) at (-1.5,0) {};
	\node[0c] (1) at (1.5,0) {};
	\node[0c] (a) at (-.625,-.85) {};
	\node[0c] (m) at (.375,0) [label=above:$x$] {};
	\draw[1c, out=75,in=105] (0) to (1);
	\draw[1c, out=-75,in=160] (0) to (a);
	\draw[1c, out=-10, in=-105] (a) to (1);
	\draw[1c] (a) to (m);
	\draw[1c] (m) to (1);
	\draw[2c] (-.6, -.6) to (-.6,.6);
	\draw[2c] (.5, -.9) to (.5,0);
	\draw[3c1] (2,0) to (3.25,0);
	\draw[3c2] (2,0) to (3.25,0);
	\draw[3c3] (2,0) to (3.25,0);
\end{scope}
\begin{scope}[shift={(5.25,0)}]
	\node[0c] (0) at (-1.5,0) {};
	\node[0c] (1) at (1.5,0) {};
	\node[0c] (a) at (-.625,-.85) {};
	\draw[1c, out=75,in=105] (0) to (1);
	\draw[1c, out=-75,in=160] (0) to (a);
	\draw[1c, out=-10, in=-105] (a) to (1);
	\draw[2c] (0, -.8) to (0,.8);
\end{scope}
\end{tikzpicture}

\end{equation*}
and let $V \eqdef \{x\}$. By Proposition \ref{prop:spherical_quotient}, $\tilde{U} \eqdef \slice{O^1 \gray U}{\sim_V}$ is a 4-atom with spherical boundary, but $\bord{}{+}\tilde{U}$ has the form
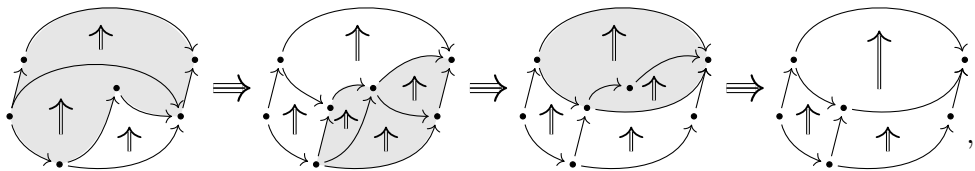
\begin{equation*}
\begin{tikzpicture}[baseline={([yshift=-.5ex]current bounding box.center)},scale=.75]
\begin{scope}
	\path[fill, color=gray!20] (-1.5,0) to [out=-75,in=160] (-.625,-.85) to [out=15,in=-105] (.375,.5) to [out=-60,in=180] (1.5,0) to (1.75,1) to [out=105,in=75,looseness=1.1] (-1.25,1) to (-1.5,0);
	\node[0c] (0) at (-1.5,0) {};
	\node[0c] (1) at (1.5,0) {};
	\node[0c] (a) at (-.625,-.85) {};
	\node[0c] (m) at (.375,.5) {};
	\node[0c] (0b) at (-1.25,1) {};
	\node[0c] (1b) at (1.75,1) {};
	\draw[1c, out=75,in=105] (0) to (1);
	\draw[1c, out=75,in=105] (0b) to (1b);
	\draw[1c, out=-75,in=160] (0) to (a);
	\draw[1c, out=-10, in=-105] (a) to (1);
	\draw[1c] (0) to (0b);
	\draw[1c] (1) to (1b);
	\draw[1c,out=15, in=-105] (a) to (m);
	\draw[1c,out=-60,in=180] (m) to (1);
	\draw[2c] (-.6, -.6) to (-.6,.6);
	\draw[2c] (.6, -.9) to (.6,.1);
	\draw[2c] (.1, .9) to (.1,1.9);
	\draw[3c1] (1.65,.5) to (3.1,.5);
	\draw[3c2] (1.65,.5) to (3.1,.5);
	\draw[3c3] (1.65,.5) to (3.1,.5);
\end{scope}
\begin{scope}[shift={(4.5,0)}]
	\path[fill, color=gray!20] (-.625,-.85) to (-.375,.15) to [out=75,in=165,looseness=1.4] (.375,.5) to [out=45,in=165,looseness=1.1] (1.75,1) to (1.5,0) to [out=-105,in=-10,looseness=1.1] (-.625,-.85);
	\node[0c] (0) at (-1.5,0) {};
	\node[0c] (1) at (1.5,0) {};
	\node[0c] (a) at (-.625,-.85) {};
	\node[0c] (m) at (.375,.5) {};
	\node[0c] (0b) at (-1.25,1) {};
	\node[0c] (1b) at (1.75,1) {};
	\node[0c] (ab) at (-.375,.15) {};
	\draw[1c, out=75,in=105] (0b) to (1b);
	\draw[1c, out=-75,in=160] (0) to (a);
	\draw[1c, out=-75,in=150] (0b) to (ab);
	\draw[1c, out=-10, in=-105] (a) to (1);
	\draw[1c, out=75, in=165] (ab) to (m);
	\draw[1c, out=45, in=165] (m) to (1b);
	\draw[1c] (0) to (0b);
	\draw[1c] (1) to (1b);
	\draw[1c] (a) to (ab);
	\draw[1c,out=15, in=-105] (a) to (m);
	\draw[1c,out=-60,in=180] (m) to (1);
	\draw[2c] (-1, -.6) to (-1,.5);
	\draw[2c] (-.1, -.5) to (-.1,.4);
	\draw[2c] (1.1, 0) to (1.1,1);
	\draw[2c] (.1, .7) to (.1,1.9);
	\draw[2c] (.6, -.9) to (.6,.1);
	\draw[3c1] (1.65,.5) to (3.1,.5);
	\draw[3c2] (1.65,.5) to (3.1,.5);
	\draw[3c3] (1.65,.5) to (3.1,.5);
\end{scope}
\begin{scope}[shift={(9,0)}]
	\path[fill, color=gray!20] (-1.25,1) to [out=-75,in=160] (-.375,.15) to [out=-10,in=-105,looseness=1.1] (1.75,1) to [out=105,in=75,looseness=1.1] (-1.25,1);
	\node[0c] (0) at (-1.5,0) {};
	\node[0c] (1) at (1.5,0) {};
	\node[0c] (a) at (-.625,-.85) {};
	\node[0c] (m) at (.375,.5) {};
	\node[0c] (0b) at (-1.25,1) {};
	\node[0c] (1b) at (1.75,1) {};
	\node[0c] (ab) at (-.375,.15) {};
	\draw[1c, out=75,in=105] (0b) to (1b);
	\draw[1c, out=-75,in=160] (0) to (a);
	\draw[1c, out=-75,in=160] (0b) to (ab);
	\draw[1c, out=-10, in=-105] (a) to (1);
	\draw[1c, out=-10, in=-105] (ab) to (1b);
	\draw[1c, out=75, in=165] (ab) to (m);
	\draw[1c, out=45, in=165] (m) to (1b);
	\draw[1c] (0) to (0b);
	\draw[1c] (1) to (1b);
	\draw[1c] (a) to (ab);
	\draw[2c] (-1, -.6) to (-1,.5);
	\draw[2c] (.4, -.9) to (.4,.1);
	\draw[2c] (.8, 0) to (.8,1);
	\draw[2c] (.1, .7) to (.1,1.9);
	\draw[3c1] (1.65,.5) to (3.1,.5);
	\draw[3c2] (1.65,.5) to (3.1,.5);
	\draw[3c3] (1.65,.5) to (3.1,.5);
\end{scope}
\begin{scope}[shift={(13.5,0)}]
	\node[0c] (0) at (-1.5,0) {};
	\node[0c] (1) at (1.5,0) {};
	\node[0c] (a) at (-.625,-.85) {};
	\node[0c] (0b) at (-1.25,1) {};
	\node[0c] (1b) at (1.75,1) {};
	\node[0c] (ab) at (-.375,.15) {};
	\draw[1c, out=75,in=105] (0b) to (1b);
	\draw[1c, out=-75,in=160] (0) to (a);
	\draw[1c, out=-75,in=160] (0b) to (ab);
	\draw[1c, out=-10, in=-105] (a) to (1);
	\draw[1c, out=-10, in=-105] (ab) to (1b);
	\draw[1c] (0) to (0b);
	\draw[1c] (1) to (1b);
	\draw[1c] (a) to (ab);
	\draw[2c] (-1, -.6) to (-1,.5);
	\draw[2c] (.4, -.9) to (.4,.1);
	\draw[2c] (.25, .2) to (.25,1.8);
	\node at (1.85,-.5) {,};
\end{scope}
\end{tikzpicture}

\end{equation*}
where the shaded area in each diagram indicates the input boundary of the following 3-atom. This is not a constructible molecule because no pair of 3-atoms in it forms a constructible molecule: the first two have non-constructible intersection, while the union of the following two is not a molecule. 

The class $\cls{K}$ satisfies all other axioms with the possible exception of \ref{ax:submolecule}, because for $U_1 \cup U_2$ to be constructible $V$ needs to be a \emph{constructible} submolecule of $\bord{}{\alpha}U_2$, and we do not know whether every submolecule that is constructible is also a constructible submolecule. 

We believe that, if we modify the definition of constructible molecule by systematically replacing the relation ``is a constructible submolecule'' with ``is a submolecule and is constructible'', we should obtain a class $\cls{K}'$ which is closed under all axioms except \ref{ax:inflate}.
\end{exm}

\begin{rmk}
Axiom \ref{ax:nontrivial} is clearly independent from the rest, since it rules out the empty class of molecules. By Example \ref{exm:class_constructible} we strongly suspect that \ref{ax:inflate} is also independent.  We do not know if other axioms are derivable from the rest.
\end{rmk}

\begin{dfn}[Cubes and simplices] \label{dfn:cubes_simplices}
The \emph{$n$\nbd cube} $\square^n$ is defined by
\begin{equation*}
	\square^0 \eqdef 1, \quad \quad \square^n \eqdef O^1 \gray \square^{n-1} \text{ if $n > 0$}.
\end{equation*}
The \emph{$n$\nbd simplex} $\Delta^n$ is defined by
\begin{equation*}
	\Delta^0 \eqdef 1, \quad \quad \Delta^n \eqdef 1 \join \Delta^{n-1} \text{ if $n > 0$}.
\end{equation*}
\end{dfn} 

\begin{lem} \label{lem:all_shapes}
Let $\cls{C}$ be a convenient class of molecules. Then $\cls{C}$ contains all the globes, the cubes, and the simplices.
\end{lem}
\begin{proof}
By axiom \ref{ax:nontrivial}, $\cls{C}$ contains $1 = O^0 = \Delta^0 = \square^0$. Then either by axiom \ref{ax:shift} or by axiom \ref{ax:inflate} the class $\cls{C}$ contains all the globes, and by axiom \ref{ax:grayjoin} it contains all the cubes and the simplices.
\end{proof}

\begin{lem} \label{lem:substitution_convenient}
Let $\cls{C}$ be a convenient class of molecules and $U, V, W \in \cls{C}$ be $n$\nbd molecules such that $U[W/V]$ is defined. Then $U[W/V] \in \cls{C}$.
\end{lem}
\begin{proof}
By axiom \ref{ax:shift} the $(n+1)$\nbd molecules $U \celto U$ and $V \celto W$ are both in $\cls{C}$. Then $\bord{}{-}(V \celto W)$ is isomorphic to $V \submol U$, which is isomorphic to a submolecule of $\bord{}{+}(U \celto U)$: the pasting $\tilde{U}$ of $U \celto U$ and $V \celto W$ along this submolecule is defined and belongs to $\cls{C}$ by axiom \ref{ax:submolecule}. Then $\bord{}{+}\tilde{U} \in \cls{C}$ is isomorphic to $U[W/V]$ by Proposition \ref{prop:spherical_pasting}.
\end{proof}

\begin{lem}
Let $\cls{C}$ be a convenient class of molecules and $U \in \cls{C}$. Then $\compos{U} \in \cls{C}$. 
\end{lem}
\begin{proof}
We have $\bord{}{\alpha}U \in \cls{C}$, therefore $\compos{U} = (\bord{}{-}U \celto \bord{}{+}U) \in \cls{C}$ by axiom \ref{ax:shift}.
\end{proof}

\begin{prop}
Let $\cls{C}$ be a convenient class of molecules and let $P, Q$ be $\cls{C}$\nbd directed complexes. Then $P \gray Q$ and $P \join Q$ are $\cls{C}$\nbd directed complexes.
\end{prop}
\begin{proof}
Immediate from axiom \ref{ax:grayjoin}.
\end{proof}

\begin{cor}
The monoidal structures given by Gray products and joins on $\rdcpx$ restrict to monoidal structures on $\dcpx$.
\end{cor}

\begin{dfn}
Let $\cls{C}$ be a convenient class of molecules, and let $f\colon P \cfun P'$ and $g\colon Q \cfun Q'$ be $\cls{C}$\nbd functors. 

Let $f \gray g\colon P \gray P' \cfun Q \gray Q'$ be defined by $\clos\{x \gray y\} \mapsto f(\clos\{x\}) \gray g(\clos\{y\})$ on the atoms of $P \gray P'$. It follows from Lemma \ref{lem:gray_molecules}, axiom \ref{ax:grayjoin}, and the properties of $f$ and $g$ that $f \gray g$ is a $\cls{C}$\nbd functor.

Similarly, letting $f \join g\colon P \join P' \cfun Q \join Q'$ be equal to $f$ on $P$, to $g$ on $P'$, and send $\clos\{x \join y\}$ to $f(\clos\{x\}) \join g(\clos\{y\})$ for all $x \in P$ and $y \in Q$ defines a $\cls{C}$\nbd functor of $\cls{C}$\nbd directed complexes. 

We deduce that both Gray products and joins determine monoidal structures on $\dcpxfun$. The monoidal structures on $\dcpx$ and $\dcpxfun$ coincide on their common subcategory $\dcpxin$.
\end{dfn}

\section{Diagrammatic sets} \label{sec:diagrammatic}

\begin{center}
\setlength{\fboxsep}{.6em}
\colorbox{gray!20}{We fix once and for all a convenient class of molecules $\cls{C}$.}
\end{center}

\begin{dfn}
We write $\atom$ for a skeleton of the full subcategory of $\dcpx$ on the atoms of every dimension.
\end{dfn}

\begin{dfn}[Diagrammatic set]
A \emph{diagrammatic set} is a presheaf on $\atom$. Diagrammatic sets and their morphisms form a category $\dgmset$.
\end{dfn}

\begin{dfn}
We identify $\atom$ with a full subcategory $\atom \incl \dgmset$ via the Yoneda embedding. With this identification, we use morphisms in $\dgmset$ as our notation of choice for both elements and structural operations of a diagrammatic set $X$:
\begin{itemize}
	\item $x \in X(U)$ becomes $x\colon U \to X$, and
	\item for each map $f\colon V \to U$ in $\atom$, $X(f)(x) \in X(V)$ becomes $f;x\colon V \to X$.
\end{itemize}
\end{dfn}

\begin{dfn}
The embedding $\atom \incl \dgmset$ extends to an embedding $\dcpx \incl \dgmset$ as follows.

Given a set $\Gamma = \{F_i\colon J_i \to \cat{C}\}_{i \in I}$ of colimit diagrams in a category $\cat{C}$, a presheaf $X$ on $\cls{C}$ is \emph{$\Gamma$-continuous} if $X(F_i(-))$ is a limit diagram in $\set$ for all $i \in I$. For all small sets $\Gamma$ the full subcategory $\psh{\Gamma}{\cat{C}}$ on $\Gamma$\nbd continuous presheaves is a reflective subcategory of $\psh{}{\cat{C}}$, and the Yoneda embedding factors through $\psh{\Gamma}{\cat{C}}$ \cite{freyd1972categories}.

Let $\Gamma$ be the set of colimit diagrams in $\dcpx$ comprising the initial object diagram and all pushout diagrams of inclusions. Up to isomorphism there are countably many $\cls{C}$\nbd directed complexes and each is the target of finitely many inclusions, so we can safely assume that $\Gamma$ is small. 

There are no non-trivial $\Gamma$-colimits in $\atom$: every presheaf on $\atom$ is trivially $\Gamma$-continuous. Thus we have a restriction functor $\psh{\Gamma}{\dcpx} \to \dgmset$, which is in fact an equivalence of categories; we leave the proof as an exercise. Through this equivalence, the Yoneda embedding $\dcpx \incl \psh{\Gamma}{\dcpx}$ becomes an embedding $\dcpx \incl \dgmset$.
\end{dfn}

\begin{dfn}[Dual diagrammatic set] Let $X$ be a diagrammatic set and $J \subseteq \mathbb{N} \setminus \{0\}$. The \emph{$J$-dual} of $X$ is the diagrammatic set $\oppn{J}{X} \eqdef X(\oppn{J}{-})$. 

For each $J$, this defines an endofunctor $\oppn{J}{}$ on $\dgmset$. In particular we write $\oppall{X}$, $\opp{X}$, and $\coo{X}$ for $X(\oppall{-})$, $X(\opp{-})$, and $X(\coo{-})$, respectively.
\end{dfn}

\begin{dfn}
By axiom \ref{ax:grayjoin}, Gray products and joins restrict to monoidal structures on $\dcpx$. The extension of these monoidal structures from $\dcpx$ to $\dgmset$ can then be developed using Day's theory \cite{day1970closed} in the precise same way as their extension from constructible directed complexes to constructible polygraphs in \cite[Section 4]{hadzihasanovic2018combinatorial}. We bundle the relevant facts together into a single statement.
\end{dfn}

\begin{prop} The following hold.
\begin{enumerate}[label=(\alph*)]
	\item The monoidal structure $(-\gray-, 1)$ on $\dcpx$ extends to a monoidal biclosed structure on $\dgmset$.
	\item The monoidal structure $(-\join-, \emptyset)$ on $\dcpx$ extends to a monoidal structure on $\dgmset$. There are inclusions $X \incl X \join Y$ and $Y \incl X \join Y$ natural in the diagrammatic sets $X, Y$. 
	\item This monoidal structure is locally biclosed, in the sense that, for all $X$, there are right adjoints to the functors $\dgmset \to \slice{X}{\dgmset}$ defined by
\begin{equation*}
	f\colon Y \to Z \;\; \mapsto
\begin{tikzpicture}[baseline={([yshift=-.5ex]current bounding box.center)}]
	\node (0) at (-1.5,-1.25) {$X \join Y$};
	\node (1) at (0,0) {$X$};
	\node (2) at (1.5,-1.25) {$X \join Z$,};
	\draw[1cincl] (1) to (0);
	\draw[1cinc] (1) to (2);
	\draw[1c] (0) to node[auto,swap,arlabel] {$\idd{X} \join f$} (2);
\end{tikzpicture} 
\end{equation*}
\begin{equation*}
	f\colon Y \to Z \;\; \mapsto
\begin{tikzpicture}[baseline={([yshift=-.5ex]current bounding box.center)}]
	\node (0) at (-1.5,-1.25) {$Y \join X$};
	\node (1) at (0,0) {$X$};
	\node (2) at (1.5,-1.25) {$Z \join X$.};
	\draw[1cincl] (1) to (0);
	\draw[1cinc] (1) to (2);
	\draw[1c] (0) to node[auto,swap,arlabel] {$f \join \idd{X}$} (2);
\end{tikzpicture}
\end{equation*}
	\item Gray products preserve colimits, while joins preserve connected colimits separately in each variable.
	\item There are isomorphisms of diagrammatic sets $\opp{(X \gray Y)} \simeq \opp{Y} \gray \opp{X}$, $\coo{(X \gray Y)} \simeq \coo{Y} \gray \coo{X}$, $\oppall{(X \gray Y)} \simeq \oppall{X} \gray \oppall{Y}$, $\opp{(X \join Y)} \simeq \opp{Y} \join \opp{X}$, natural in $X$ and $Y$. 
\end{enumerate}
\end{prop}

\begin{dfn}[Diagrams and cells] Let $X$ be a diagrammatic set and $U$ a molecule in $\dcpx$. A \emph{diagram of shape $U$ in $X$} is a morphism $x\colon U \to X$. It is \emph{composable} if $U \in \cls{C}$ and a \emph{cell} if $U$ is an atom. 

For all $n \in \mathbb{N}$, we say that $x$ is \emph{$n$\nbd diagram} or an \emph{$n$\nbd cell} when $\dmn{U} = n$.

Given an inclusion $\imath\colon V \incl U$ of molecules, we write $\restr{x}{V}$ for the diagram $\imath;x\colon V \to X$. If $V \submol U$, we write $\restr{x}{V} \submol x$, and call $\restr{x}{V}$ a \emph{subdiagram} of $x$. If $U$ decomposes as $U_1 \cp{k} U_2$, we write $x = x_1 \cp{k} x_2$ for $x_i \eqdef \restr{x}{U_i}$, $i \in \{1,2\}$. 

If $x\colon U \to X$ is a diagram in $X$ and $f\colon X \to Y$ a morphism of diagrammatic sets, we write $f(x)$ for the diagram $x;f\colon U \to Y$. 
\end{dfn}

\begin{dfn}
For simplicity, we assume to have fixed a skeleton of $\dcpx$, so we can compare diagrams for equality even when their shapes are, a priori, only uniquely isomorphic. 
\end{dfn}

\begin{dfn}[Boundaries of diagrams] Let $X$ be a diagrammatic set, $x\colon U \to X$ a diagram, and let $\imath_k^\alpha\colon \bord{k}{\alpha}U \incl U$ be the inclusions of the $k$\nbd boundaries of $U$. The \emph{input $k$\nbd boundary} of $x$ is the diagram $\bord{k}{-}x \eqdef \imath_k^-;x$ and the \emph{output $k$\nbd boundary} of $x$ is the diagram $\bord{k}{+}x \eqdef \imath_k^+;x$. We may omit the index $k$ when $k = \dmn{U} -1$. 

We write $x\colon y^- \celto y^+$ to express that $\bord{k}{\alpha}x = y^\alpha$ for each $\alpha \in \{+,-\}$, and say that $x$ is of \emph{type} $y^- \celto y^+$. We say that two diagrams $x_1, x_2$ are \emph{parallel} if they have the same type.
\end{dfn}

\begin{dfn}[Substitution of diagrams]
Let $U, V, W$ be molecules such that $U[W/V]$ is defined, and let $x\colon U \to X$ and $z\colon W \to X$ be diagrams such that $z$ is parallel to $y \eqdef \restr{x}{V}$. The \emph{substitution of $z$ for $y \submol x$} is the diagram $x[z/y]\colon U[W/V] \to X$ which is equal to $z$ on $W$ and to $x$ on $U \setminus (V \setminus \bord V)$.
\end{dfn}

\begin{dfn}[Pasting of diagrams]
Let $U_1, U_2$ be molecules such that the pasting $U_1 \cup U_2$ of $U_1$ and $U_2$ along $V \submol \bord{}{\alpha}U_2$ is defined. This is a pushout of inclusions $V \incl U_1$ and $V	\incl U_2$. 

If $x_1\colon U_1 \to X$ and $x_2\colon U_2 \to X$ are diagrams with the property that $y \eqdef \restr{x_1}{V} = \restr{x_2}{V}$, there is a unique diagram $x_1 \cup x_2\colon U_1 \cup U_2 \to X$ such that $\restr{x_1 \cup x_2}{U_i} = x_i$ for each $i \in \{1,2\}$. We call it the \emph{pasting of $x_1$ and $x_2$ along $y \submol \bord{}{\alpha}x_2$}.
\end{dfn}

\begin{dfn}[Attaching cells]
Let $y^-\colon U \to X$ and $y^+\colon V \to X$ be parallel composable diagrams. There is a unique morphism $[y^-, y^+]\colon \bord (U \celto V) \to X$ restricting to $y^\alpha$ on $\bord{}{\alpha}(U \celto V)$. 

If $\{(y^-_i\colon U_i \to X, y^+_i\colon V_i \to X)\}_{i \in I}$ is a family of pairwise parallel composable diagrams, we let $X' \eqdef X \cup \{x_i\colon y^-_i \celto y^+_i \}$ be the pushout
\begin{equation*}
\begin{tikzpicture}[baseline={([yshift=-.5ex]current bounding box.center)}]
	\node (0) at (-1,2) {$\coprod_{i \in I} \bord (U_i \celto V_i)$};
	\node (1) at (2.5,0) {$X'$};
	\node (2) at (-1,0) {$X$};
	\node (3) at (2.5,2) {$\coprod_{i \in I} U_i \celto V_i$};
	\draw[1cinc] (0) to (3);
	\draw[1c] (0) to node[auto,swap,arlabel] {$([y_i^-,y_i^+])_{i \in I}$} (2);
	\draw[1cinc] (2) to (1);
	\draw[1c] (3) to node[auto,arlabel] {$(x_i)_{i \in I}$} (1);
	\draw[edge] (1.6,0.2) to (1.6,0.7) to (2.3,0.7);
\end{tikzpicture}
\end{equation*}
in $\dgmset$. We say that $X'$ is the result of \emph{attaching} the cells $\{x_i\colon y_i^- \celto y_i^+\}_{i \in I}$ to $X$.
\end{dfn}

\begin{dfn}[Degenerate diagram]
Let $x\colon U \to X$ be a diagram. We say that $x$ is \emph{degenerate} if there exist a surjective map of molecules $p\colon U \surj V$ in $\dcpx$ and a diagram $y\colon V \to X$ such that $\dmn{V} < \dmn{U}$ and $x = p;y$. 

We write $\degg{X}$ for the set of degenerate composable diagrams of $X$.
\end{dfn}

\begin{dfn} Let $U$ and $V \submol \bord{}{-}U$ be molecules with spherical boundary such that $\dmn{U} = \dmn{V} + 1$. Let $W \eqdef \bord U \setminus (V \setminus \bord V)$. Then 
\begin{equation*}
	\lunit{V}{U}{+} \eqdef \slice{O^1 \gray U}{\sim_W} \,
\end{equation*}
is a molecule with spherical boundary and there are isomorphisms
\begin{equation*}
	\bord{}{-}\lunit{V}{U}{+} \incliso U, \quad \quad \bord{}{+}\lunit{V}{U}{+} \incliso \infl{V} \cup U.
\end{equation*}	
The natural map $O^1 \gray U \surj U$ factors as
\begin{equation*}
\begin{tikzpicture}[baseline={([yshift=-.5ex]current bounding box.center)}]
	\node (0) at (-1.5,1.25) {$O^1 \gray U$};
	\node (1) at (0,0) {$\lunit{V}{U}{+}$};
	\node (2) at (1.5,1.25) {$U$};
	\draw[1csurj] (0) to (2);
	\draw[1csurj] (0) to (1);
	\draw[1csurj] (1) to node[auto,swap,arlabel] {$\lmap{V}{U}{+}$} (2);
\end{tikzpicture}
\end{equation*}
for a unique surjective map $\lmap{V}{U}{+}$, which is a retraction onto $U$ in the sense that 
\begin{equation*}
\begin{tikzpicture}[baseline={([yshift=-.5ex]current bounding box.center)}]
	\node (0) at (-1.5,1.25) {$U$};
	\node (1) at (0,0) {$\lunit{V}{U}{+}$};
	\node (2) at (1.5,1.25) {$U$};
	\draw[1c] (0) to node[auto,arlabel] {$\idd{U}$} (2);
	\draw[1cinc] (0) to node[auto,swap,arlabel] {$\imath$} (1);
	\draw[1csurj] (1) to node[auto,swap,arlabel] {$\lmap{V}{U}{+}$} (2);
\end{tikzpicture}
\end{equation*}
commutes when $\imath$ is the inclusion of $U$ into either boundary of $\lunit{V}{U}{+}$.

Dually if $V \submol \bord{}{+}U$, we write
\begin{equation*}
	\runit{V}{U}{-} \eqdef \slice{O^1 \gray U}{\sim_W}.
\end{equation*}
This has isomorphisms
\begin{equation*}
	\bord{}{-}\runit{V}{U}{-} \incliso U \cup \infl{V}, \quad \quad \bord{}{+}\runit{V}{U}{-} \incliso U,
\end{equation*}	
and factorising $O^1 \gray U \surj U$ we obtain a retraction $\rmap{V}{U}{-}\colon \runit{V}{U}{-} \surj U$ onto $U$. If $n \eqdef \dmn{U} + 1$, we let 
\begin{equation*}
	\lunit{V}{U}{-} \eqdef \oppn{n}{\lunit{V}{U}{+}}, \quad \quad \runit{V}{U}{+} \eqdef \oppn{n}{\runit{V}{U}{-}},
\end{equation*}
and
\begin{equation*}
	\lmap{V}{U}{-} \eqdef \rev{\lmap{V}{U}{+}}, \quad \quad \rmap{V}{U}{+} \eqdef \rev{\rmap{V}{U}{-}}.
\end{equation*}
By axioms \ref{ax:dual} and \ref{ax:inflate}, if $U, V \in \cls{C}$, then $\lunit{V}{U}{\alpha}, \runit{V}{U}{\alpha} \in \cls{C}$.
\end{dfn}

\begin{dfn}[Units]
Let $x\colon U \to X$ be a composable diagram. The \emph{unit} on $x$ is the composable diagram $\eps{}{x} \eqdef \tau;x\colon \infl{U} \to X$. 
\end{dfn}

\begin{dfn}[Unitors] Let $x\colon U \to X$ be a composable $n$\nbd diagram and $y \submol \bord{}{-}x$ a composable $(n-1)$\nbd diagram of shape $V \submol \bord{}{-}U$. The \emph{left unitors of $x$ at $y \submol \bord{}{-}x$} are the degenerate composable $(n+1)$\nbd diagrams
\begin{equation*}
	\lunitor{y}{x}{\alpha} \eqdef \lmap{V}{U}{\alpha}; x\colon \lunit{V}{U}{\alpha} \to X
\end{equation*}
for each $\alpha \in \{+,-\}$. We have
\begin{equation*}
	\lunitor{y}{x}{+}\colon x \celto (\eps{}y \cup x), \quad \quad \lunitor{y}{x}{-}\colon (\eps{}y \cup x) \celto x.
\end{equation*}	
Dually if $y \submol \bord{}{+}x$, the \emph{right unitors of $x$ at $y \submol \bord{}{+}x$} are the degenerate diagrams
\begin{equation*}
	\runitor{y}{x}{\alpha} \eqdef \rmap{V}{U}{\alpha}; x\colon \runit{V}{U}{\alpha} \to X,
\end{equation*}
where 
\begin{equation*}
	\runitor{y}{x}{+}\colon x \celto (x \cup \eps{}y), \quad \quad \runitor{y}{x}{-}\colon (x \cup \eps{}y) \celto x.
\end{equation*}
\end{dfn}

\begin{rmk} \label{rmk:unit_biased}
The definitions are somewhat biased towards the \emph{left} cylinder $O^1 \gray U$. We could as well use quotients of the right cylinder $U \gray O^1$. These belong to $\cls{C}$ by axiom \ref{ax:dual} and the isomorphism between
\begin{equation*}
	\slice{U \gray O^1}{\sim_V} \quad \text{and} \quad  \opp{(\slice{(O^1 \gray \opp{U})}{\sim_{\opp{V}}})}.
\end{equation*}
\end{rmk}

\section{Equivalences} \label{sec:equivalences}

\subsection{Definitions}

\begin{dfn}[Dividend] 
Let $x\colon U \to X$ be a composable $n$\nbd diagram in a diagrammatic set. A \emph{dividend for $x$} is a composable $n$\nbd diagram $x_0\colon U_0 \to X$ together with an inclusion $\imath\colon \bord{}{\alpha}U \incl \bord{}{\alpha}U_0$ such that $\imath(\bord{}{\alpha}U) \submol \bord{}{\alpha}U_0$ and
\begin{equation*}
	\restr{x}{\bord{}{\alpha} U} = \restr{x_0}{\bord{}{\alpha}U}.
\end{equation*}
\end{dfn}

\begin{dfn}[Lax and colax division]
Let $x\colon U \to X$ be a composable $n$\nbd diagram and $(x_0\colon U_0 \to X, \imath\colon \bord{}{\alpha}U \incl \bord{}{\alpha}U_0)$ a dividend for $x$. A \emph{lax division of $(x_0, \imath)$ by $x$} is a triple of
\begin{enumerate}
	\item a composable $n$\nbd diagram $x_\alpha\colon U_\alpha \to X$, 
	\item a subdiagram $y \submol \bord{}{\alpha}x_\alpha$ such that 
	\begin{enumerate}
		\item the pasting $x \cup x_\alpha$ along this subdiagram is defined and parallel to $x_0$,
		\item $\imath$ is equal to the inclusion $\bord{}{\alpha}U \incl \bord{}{\alpha}(U \cup U_\alpha)$,
	\end{enumerate}
	\item a composable $(n+1)$\nbd diagram $h\colon x \cup x_\alpha \celto x_0$.
\end{enumerate}
A \emph{colax division of $(x_0, \imath)$ by $x$} is defined in the same way, except we require a composable $(n+1)$\nbd diagram $h'\colon x_0 \celto x \cup x_\alpha$. 
\end{dfn}

\begin{dfn}
We will sometimes write ``$x_0$ is a dividend for $x$'' and ``$h$ is a (co)lax division of $x_0$ by $x$'', leaving the additional data implicit.
\end{dfn}

\begin{comm}
A dividend $(x_0, \imath)$ for $x$ corresponds ideally to an equation of the form $x \cup x_\alpha = x_0$ to be solved in the indeterminate $x_\alpha$. A (co)lax division of $x_0$ by $x$ is then a (co)lax solution to this equation, where the equality is replaced by a composable diagram in the next higher dimension. 

This is akin to the solution of a lifting problem, except its shape is not fixed. It may be presented as the solution of a lifting problem not in $X$, but in the $\cls{C}$\nbd polygraph $\mmonad X$, as defined in Section \ref{sec:omegacats}.
\end{comm}

\begin{dfn}[Equivalence]
Let $x$ be a composable diagram in a diagrammatic set $X$. Coinductively, we say that $x$ is an \emph{equivalence} if, for all dividends $(x_0, \imath)$ for $x$, there exists a lax division $(x_\alpha, y \submol \bord{}{\alpha}x_\alpha, h\colon x \cup x_\alpha \celto x_0)$ of $x$ by $(x_0, \imath)$ such that $h$ is an equivalence.

We write $\equi{X}$ for the set of equivalences of $X$.
\end{dfn}

\begin{rmk}
The definition may seem biased towards lax division over colax division. We will prove that, in fact, if $(x_0, \imath)$ is a dividend for an equivalence $x$, then both a lax division and a colax division of $x_0$ by $x$ exist and are equivalences.
\end{rmk}

\begin{comm}
Because coinductive definitions and proofs are not very common outside of theoretical computer science, we take some time to explain the definition of equivalence and the corresponding proof method.

Let $X$ be a diagrammatic set and let $\cdgm{X}$ be the set of composable diagrams in $X$. For each composable diagram $x$, let $\dvd{x}$ be the set of its dividends, and for each dividend $(x_0, \imath)$ for $x$, let $\divis{x_0, \imath}{x}$ be the set of lax divisions of $(x_0,\imath)$ by $x$.

For each subset $A \subseteq \cdgm{X}$, let
\begin{align*}
	\equifun{}{A} \eqdef \; & \{ x \in \cdgm{X} \mid \text{for all $(x_0,\imath) \in \dvd{x}$, there exists} \\
	& \text{$(x_\alpha, y \submol \bord{}{\alpha}x_\alpha, h) \in \divis{x_0,\imath}{x}$ such that $h \in A$} \}.
\end{align*}
That is, the set $\equifun{}{A}$ is the set of those diagrams $x$ such that any equation $x \cup x_\alpha = x_0$ has a lax solution exhibited by a diagram in $A$.

If $A \subseteq B$, then $B$ is a larger space of potential solutions, so $\equifun{}{A} \subseteq \equifun{}{B}$. It follows that $A \mapsto \equifun{}{A}$ is an order-preserving endomorphism of the powerset $\powset{\cdgm{X}}$ seen as a poset. Any such endomorphism has a greatest fixed point, constructed as the limit of
\begin{equation*}
	\ldots \subseteq \equifun{n}{\cdgm{X}} \subseteq \ldots \subseteq \equifun{}{\cdgm{X}} \subseteq \cdgm{X}.
\end{equation*}
This greatest fixed point is $\equi{X}$.

This definition provides the following proof method: given $A \subseteq \cdgm{X}$, if $A \subseteq \equifun{}{A}$, then $A \subseteq \equi{X}$. 

Because of the grading of $\cdgm{X}$ given by the dimension of diagrams, such proofs may look like ``inductive proofs without the base case''. Indeed, if $A = \bigcup_{n\in\mathbb{N}} A_n$, in order to prove that $A \subseteq \equifun{}{A}$ we need to show that, for all $n \in \mathbb{N}$, equations involving $x \in A_n$ have (co)lax solutions in $A_{n+1}$. 

Such a proof may be misconstrued as follows. Let $P(n)$ be the statement that $A_n \subseteq (\equi{X})_n$; then $A \subseteq \equi{X}$ is equivalent to $\forall n \, P(n)$. Suppose we have proved that equations for $x \in A_n$ have lax solutions in $A_{n+1}$. Assuming that $P(n+1)$ holds, these solutions are in $(\equi{X})_{n+1}$, so $x \in (\equi{X})_{n}$ and $P(n)$ holds. 

It is tempting to see the implication of $P(n)$ from $P(n+1)$ as an ``inductive step'' in the proof of $\forall n \, P(n)$, but this does \emph{not} correspond to a valid proof principle: think of the case of a uniformly false $P(n)$. What validates the proof is not the fact that the solutions are ``equivalences by assumption'', but the fact that they belong to $A$. 
\end{comm}


\subsection{Closure properties}

\begin{dfn}
Let $A \subseteq \cdgm{X}$ and let $x$ be a composable diagram in $X$. We let $x \in \twothree{}{A}$ if either $x \in A$, or there exist composable diagrams $x_+, x_-, x_0$ and $y \submol \bord{}{\alpha}x_\alpha$ for some $\alpha \in \{+,-\}$ such that 
	\begin{enumerate}
		\item the pasting of $x_+$ and $x_-$ along $y \submol \bord{}{\alpha}x_\alpha$ is defined and
		\item $x_+ \cup x_-$ is parallel to $x_0$,
	\end{enumerate}
	together with $h\colon x_+ \cup x_- \celto x_0$ or $h\colon x_0 \celto x_+ \cup x_-$, where either
	\begin{itemize}
		\item $x = x_0$ and $h, x_+, x_- \in A$, or
		\item $x = x_\alpha$ and $h, x_0, x_{-\alpha} \in A$.
	\end{itemize}
By construction $A \subseteq \twothree{}{A}$. We define $\twothree{\infty}{A} \eqdef \bigcup_{n \in \mathbb{N}} \twothree{n}{A}$. 

Our goal in this section is to prove that 
\begin{equation*}
	\equi{X} = \twothree{\infty}{\equi{X} \cup \degg{X}},
\end{equation*}
that is, equivalences contain all degenerate composable diagrams, and are closed under ``division'' and ``composition'' witnessed by equivalences in the next dimension.
\end{dfn}

\begin{lem} \label{lem:diagram_reverse}
Let $A \subseteq \cdgm{X}$ and let $h\colon x \celto y$ be a diagram in $A \cap \equifun{}{A}$. Then there exists a composable diagram $h^*\colon y \celto x$ in $\twothree{2}{A}$.
\end{lem}
\begin{proof}
Suppose $h$ has shape $U$. Then $h$ together with $\idd{\bord{}{-}U}$ is a dividend for $h$. By assumption, there is a lax division $k\colon h \cup e \celto h$ of $h$ by itself, where $e$ has type $y \celto y$ and $k \in A$. Since $h, k \in A$, it follows that $e \in \twothree{}{A}$. 

Now, $e$ has some shape $U_-$ with $\bord{}{+}U_-$ isomorphic to $\bord{}{+}U$. Then $e$ together with this isomorphic inclusion is a dividend for $h$, so there is a lax division $k'\colon h^* \cup h \celto e$ of $e$ by $h$, where $h^*$ has type $y \celto x$ and $k' \in A$. Since $k', h \in A$ and $e \in \twothree{}{A}$, it follows that $h^* \in \twothree{2}{A}$. 
\end{proof}

\begin{prop} \label{prop:equivalence_reverse}
Let $h\colon x \celto y$ be an equivalence in $X$. Then there exists a composable diagram $h^*\colon y \celto x$ in $\twothree{2}{\equi{X}}$.
\end{prop}
\begin{proof}
Follows from Lemma \ref{lem:diagram_reverse} with $A \eqdef \equi{X}$.
\end{proof}

\begin{lem} \label{lem:pasting_twothree}
Let $A \subseteq \cdgm{X}$ and let $x_+, x_- \in A$ such that the pasting $x_+ \cup x_-$ along $y \submol \bord{}{\alpha}x_\alpha$ is defined. Then $x_+ \cup x_- \in \twothree{}{A \cup \degg{X}}$.
\end{lem}
\begin{proof}
Let $x \eqdef x_+ \cup x_-$ and write the type of the unit $\eps{}{x}$ as $x_+ \cup x_- \celto x$. Then $x_+, x_- \in A$ and $\eps{}{x} \in \degg{X}$, so $x \in \twothree{}{A \cup \degg{X}}$.
\end{proof}

\begin{dfn}[Unbiased set of solutions] 
Let $A \subseteq \cdgm{X}$ and $x \in \equifun{}{A}$. We say that $A$ is \emph{unbiased for $x$} if, for all dividends $(x_0, \imath)$ for $x$, there exists a pair $(x_\alpha, y \submol \bord{}{\alpha}x_\alpha, h)$ and $(x_\alpha, y \submol \bord{}{\alpha}x_\alpha, h')$ of a lax and a colax division of $(x_0, \imath)$ by $x$ with $h, h' \in A$.

Given a set $B \subseteq \equifun{}{A}$, we say that $A$ is \emph{unbiased for $B$} if $A$ is unbiased for each $x \in B$.
\end{dfn}

\begin{comm}
If we interpret $x_\alpha$ as ``the result of dividing $x_0$ by $x$'', unbiasedness of $A$ means that the same result is exhibited both by a lax division and by a colax division in $A$.
\end{comm}

\begin{lem} \label{lem:equi_unbiased}
Let $X$ be a diagrammatic set. Then
\begin{enumerate}[label=(\alph*)]
	\item $\equi{X} \subseteq \equifun{}{\twothree{2}{\equi{X}}}$ and
	\item $\twothree{2}{\equi{X}}$ is unbiased for $\equi{X}$.
\end{enumerate}
\end{lem}
\begin{proof}
We have $\equi{X} = \equifun{}{\equi{X}} \subseteq \equifun{}{\twothree{2}{\equi{X}}}$. If $x$ is an equivalence and $x_0$ a dividend for $x$, there is a lax division $h\colon x \cup x_\alpha \celto x_0$ of $x_0$ by $x$ where $h$ is itself an equivalence. By Proposition \ref{prop:equivalence_reverse}, there is a diagram $h^*\colon x_0 \celto x \cup x_\alpha$ in $\twothree{2}{\equi{X}}$: this is a colax division of $x_0$ by $x$. This proves that $\twothree{2}{\equi{X}}$ is unbiased for $\equi{X}$.
\end{proof}

\begin{lem} \label{lem:cell_to_diagram}
Let $A \subseteq \cdgm{X}$ and let $x$ be a composable $n$\nbd diagram. Suppose that $x_i \in \equifun{}{A}$ for all $n$\nbd cells $x_1, \ldots, x_m \submol x$. Then
\begin{enumerate}[label=(\alph*)]
	\item $x \in \equifun{}{\twothree{m-1}{A \cup \degg{X}}}$ and
	\item if $A$ is unbiased for each $x_i$, then $\twothree{m-1}{A \cup \degg{X}}$ is unbiased for $x$.
\end{enumerate}
\end{lem}
\begin{proof}
Let $U$ be the shape of $x$. Write $U = V_1 \cp{n-1} \ldots \cp{n-1} V_m$ where each $V_i$ contains a single $n$\nbd atom $U_i$ as in Lemma \ref{lem:spherical_moves}, and let $x_i \eqdef \restr{x}{U_i}$. Note that $U = \bigcup_{i=1}^m U_i$ because $U$ is pure. 

Let $(y_0\colon W_0 \to X, \imath\colon \bord{}{\alpha}U \incl \bord{}{\alpha}W_0)$ be a dividend for $x$, and without loss of generality suppose $\alpha = -$. For $1 \leq i \leq m$, we recursively construct
\begin{enumerate} 
	\item dividends $(y_{i-1}\colon W_{i-1} \to X, \imath_i\colon \bord{}{-}U_i \incl \bord{}{-}W_{i-1})$ for $x_i$, 
	\item inclusions of submolecules $\bord{}{-}V_i \incl \bord{}{-}W_{i-1}$, and
	\item lax divisions $h_i\colon x_i \cup y_i \celto y_{i-1}$ of $y_{i-1}$ by $x_i$ with $h_i \in A$
\end{enumerate}
such that $\bord{}{-}W_i$ is isomorphic to $\bord{}{-}W_{i-1}[\bord{}{+}U_i/\bord{}{-}U_i]$.

We have $\bord{}{-}V_1 = \bord{}{-}U$, and let the inclusion $\bord{}{-}V_1 \incl \bord{}{-}W_0$ be $\imath$. If $i > 1$, from the inclusion $\bord{}{-}V_{i-1} \incl \bord{}{-}W_{i-2}$ we obtain an inclusion of submolecules $\bord{}{-}V_{i-1}[\bord{}{+}U_{i-1}/\bord{}{-}U_{i-1}] \incl \bord{}{-}W_{i-1}$. By Lemma \ref{lem:spherical_moves}, its source is isomorphic to $\bord{}{+}V_{i-1}$ which is equal to $\bord{}{-}V_i$. Composing with this isomorphism, we obtain an inclusion $\bord{}{-}V_i \incl \bord{}{-}W_{i-1}$. 

For each $i$, we let $\imath_i$ be its restriction to $\bord{}{-}U_i \submol \bord{}{-}V_i$. This makes $(y_{i-1}, \imath_i)$ a dividend for $x_i$. By assumption, there exists a lax division $h_i\colon x_i \cup y_i \celto y_{i-1}$ of $y_{i-1}$ by $x_i$ with $h_i \in A$. 

If $W_i$ is the shape of $y_i$, $\bord{}{-}W_{i-1}$ is isomorphic to $\bord{}{-}(U_i \cup W_i)$ which by Proposition \ref{prop:spherical_pasting} is isomorphic to $\bord{}{-}W_i[\bord{}{-}U_i/\bord{}{+}U_i]$. Substituting backwards, $\bord{}{-}W_i$ is isomorphic to $\bord{}{-}W_{i-1}[\bord{}{+}U_i/\bord{}{-}U_i]$. This completes the recursion.

Next, for $1 \leq i \leq m$ we construct composable diagrams
\begin{equation*}
	\tilde{h}_i\colon x_1 \cup (x_2 \cup \ldots (x_i \cup y_i)\ldots) \celto y_0
\end{equation*}
such that $y_i \submol \bord{}{-}\tilde{h}_i$ and $\tilde{h}_i \in \twothree{i-1}{A \cup \degg{X}}$. We let $\tilde{h}_1 \eqdef h_1$ which by assumption is in $A$. 

If $i > 1$, the pasting of $h_i$ and $\tilde{h}_{i-1}$ along $y_{i-1} \submol \bord{}{-}\tilde{h}_{i-1}$ is defined, and we let $\tilde{h}_i \eqdef h_i \cup \tilde{h}_{i-1}$. Then $y_i \submol \bord{}{-}h_i \submol \bord{}{-}\tilde{h}_i$. Moreover, $h_i \in A$ and $\tilde{h}_i \in \twothree{i-1}{A \cup \degg{X}}$, so by Lemma \ref{lem:pasting_twothree} $\tilde{h}_i \in \twothree{i}{A \cup \degg{X}}$. 

To conclude, the type of $\tilde{h}_m$ may be rewritten as $x \cup y_m \celto y_0$, so $\tilde{h}_m$ is a lax division of $y_0$ by $x$ and $\tilde{h}_m \in \twothree{m-1}{A \cup \degg{X}}$. 

If $A$ is unbiased for the $x_i$, there is a dual construction of a colax division of $y_0$ by $x$: we construct step by step colax divisions $h'_i\colon y_{i-1} \celto x_i \cup y_i$ of $y_{i-1}$ by $x_i$, then paste them together to construct a colax division $\tilde{h}'_m\colon y_0 \celto x \cup y_m$ in $\twothree{m-1}{A \cup \degg{X}}$. This proves that $\twothree{m-1}{A \cup \degg{X}}$ is unbiased for $x$.
\end{proof}

\begin{lem} \label{lem:degg_unbiased}
Let $X$ be a diagrammatic set. Then
\begin{enumerate}[label=(\alph*)]
	\item $\degg{X} \subseteq \equifun{}{\twothree{\infty}{\degg{X}}}$ and
	\item $\twothree{\infty}{\degg{X}}$ is unbiased for $\degg{X}$.
\end{enumerate}
\end{lem}
\begin{proof}
Let $x$ be a degenerate $n$\nbd cell of shape $U$ in $X$ and let $(x_0\colon U_0 \to X,$ $\imath\colon \bord{}{\alpha}U \incl \bord{}{\alpha}U_0)$ be a dividend for $x$. Without loss of generality suppose $\alpha = -$. We then have $\bord{}{-}x = y$ for some $y \submol \bord{}{-}x_0$. 

By definition, there exist a surjective map of atoms $p\colon U \surj V$ and a cell $\tilde{x}\colon V \to X$ such that $\dmn{V} < n$ and $x = p;\tilde{x}$. Then $p$ has a reverse map $\rev{p}\colon \oppn{n}{U} \surj V$. We let $x^* \eqdef \rev{p};\tilde{x}$, a cell of shape $\oppn{n}{U}$. 

Consider the atom
\begin{equation*}
	W \eqdef \infl{\bord{}{-}U} \celto (U \cp{n-1} \oppn{n}{U});
\end{equation*}
this is defined and by axioms \ref{ax:dual}, \ref{ax:shift}, \ref{ax:submolecule}, and \ref{ax:inflate} it belongs to $\cls{C}$. There is a surjective map $q\colon W \surj V$ defined by
\begin{enumerate}
	\item $\restr{q}{\infl{\bord{}{-}U}} \eqdef \tau;\restr{p}{\bord{}{-}U}$,
	\item $\restr{q}{U} \eqdef p$,
	\item $\restr{q}{\oppn{n}{U}} \eqdef \rev{p}$, and
	\item $q$ sends the greatest element of $W$ to the greatest element of $V$.
\end{enumerate}
Then $q; \tilde{x}$ is a degenerate $(n+1)$\nbd cell of type $\eps{}{y} \celto x\cp{n-1}x^*$. 

Moreover, we have a left unitor $\lunitor{y}{x_0}{+}$ of $x_0$ at $y \submol \bord{}{-}x_0$ which is a degenerate $(n+1)$\nbd diagram of type $x_0 \celto \eps{}{y} \cup x_0$. The pasting $h' \eqdef \lunitor{y}{x_0}{+} \cup q;\tilde{x}$ at $\eps{}{y} \submol \bord{}{+}\lunitor{y}{x_0}{+}$ is defined and its type can be written as
\begin{equation*}
	x_0 \celto x \cup (x^* \cup x_0),
\end{equation*}
so $h'$ is a colax division of $x_0$ by $x$. By Lemma \ref{lem:pasting_twothree}, $h' \in \twothree{}{\degg{X}}$. 

Dually, the pasting of $\rev{q}$ and the left unitor $\lunitor{y}{x_0}{-}$ is defined, has type $x \cup (x^* \cup x_0) \celto x_0$, and belongs to $\twothree{}{\degg{X}}$. The case $\alpha = +$ is dual, using right unitors. This proves that $x \in \equifun{}{\twothree{}{\degg{X}}}$ and that $\twothree{}{\degg{X}}$ is unbiased for $x$. 

Now, suppose $x$ is a degenerate composable $n$\nbd diagram. Then all $n$\nbd cells $x_1, \ldots, x_m \submol x$ are degenerate: by the first part of the proof and by Lemma \ref{lem:cell_to_diagram}, $x \in \equifun{}{\twothree{m}{\degg{X}}}$ and $\twothree{m}{\degg{X}}$ is unbiased for $x$. This completes the proof.
\end{proof}

\begin{lem} \label{lem:twothree_inductive}
Let $A, B \subseteq \cdgm{X}$ and suppose that
\begin{enumerate}
	\item $A \subseteq B \cap \equifun{}{B}$ and
	\item $B$ is unbiased for $A$.
\end{enumerate}
Then $\twothree{}{A} \subseteq \equifun{}{\twothree{7}{B \cup \degg{X}}}$ and $\twothree{7}{B \cup \degg{X}}$ is unbiased for $\twothree{}{A}$.
\end{lem}
\begin{proof}
Let $x \in \twothree{}{A}$ be a composable diagram of shape $U$ and suppose $(y_0\colon W_0 \to X, \imath\colon \bord{}{\alpha}U \incl \bord{}{\alpha}W_0)$ is a dividend for $x$. Without loss of generality, suppose $\alpha = -$. We proceed by case distinction.
\begin{itemize}
	\item If $x \in A$, there is nothing to prove.
	
	\item \emph{There is a composable diagram $h\colon x_+ \cup x_- \celto x$ with $h, x_+, x_- \in A$.} Let $U_\alpha$ be the shape of $x_\alpha$. We have $\bord{}{-}U_+ \submol \bord{}{-}U$, so if we let $\imath_+ \eqdef \restr{\imath}{\bord{}{-}U_+}$, the pair $(y_0, \imath_+)$ is a dividend for $x_+$. \\
	By assumption, there is a colax division $k_+\colon y_0 \celto x_+ \cup y_+$ of $y_0$ by $x_+$ with $k_+ \in B$. If $W_+$ is the shape of $y_+$, by Proposition \ref{prop:spherical_pasting} $\bord{}{-}W_+$ is isomorphic to $\bord{}{-}W_0[\bord{}{+}U_+/\bord{}{-}U_+]$, which contains $\bord{}{-}U[\bord{}{+}U_+/\bord{}{-}U_+]$ as a submolecule. This in turn contains $\bord{}{-}U_-$ as a submolecule. Through these identifications, we obtain an inclusion $\imath_-\colon \bord{}{-}U_- \incl \bord{}{-}W_+$ which makes $(y_+, \imath_-)$ a dividend for $x_-$. \\
	Then there is a colax division $k_-\colon y_+ \celto x_- \cup y_-$ of $y_+$ by $x_-$ with $k_- \in B$. The pasting $k_+ \cup k_-$ along $y_+ \submol \bord{}{+}k_+$ is defined, has type $y_0 \celto (x_+ \cup x_-) \cup y_-$, and by Lemma \ref{lem:pasting_twothree} it belongs to $\twothree{}{B \cup \degg{X}}$. Then the pasting 
	\begin{equation*}
		(k_+ \cup k_-) \cup h\colon y_0 \celto x \cup y_-
	\end{equation*}
	along $(x_+ \cup x_-) \submol \bord{}{+}(k_+ \cup k_-)$ is also defined and produces a colax division of $(y_0, \imath)$ by $x$, which by Lemma \ref{lem:pasting_twothree} belongs to $\twothree{2}{B \cup \degg{X}}$. \\
	Now by Lemma \ref{lem:diagram_reverse} applied to $h \in B \cap \equifun{}{B}$ there is a composable diagram $h^*\colon x \celto x_+ \cup x_-$ in $\twothree{2}{B}$. Moreover, because $B$ is unbiased for $A$, we have lax divisions 
	\begin{equation*}
		k'_+\colon x_+ \cup y_+ \celto y_0, \quad \quad k'_-\colon x_- \cup y_- \celto y_+
	\end{equation*}
	both in $B$. With two pastings, we construct a composable diagram 
	\begin{equation*}
		h^* \cup (k'_- \cup k'_+)\colon x \cup y_- \celto y_0
	\end{equation*}
	which is a lax division of $(y_0, \imath)$ by $x$ and belongs to $\twothree{4}{B \cup \degg{X}}$. This proves that $x \in \equifun{}{\twothree{4}{B \cup \degg{X}}}$ and that $\twothree{4}{B \cup \degg{X}}$ is unbiased for $x$. \\
	The case where we start from $h\colon x \celto x_+ \cup x_-$ is dual.
	
	\item \emph{There is a composable diagram $h\colon x_+ \cup x \celto x_0$ with $h, x_+, x_0 \in A$ and $z \eqdef \bord{}{+}x_+ \submol \bord{}{-}x$.} Then $z \submol \bord{}{-}y_0$ through $\imath$, and we have a left unitor
	\begin{equation*}
		\lunitor{z}{y_0}{-}\colon \eps{}{z} \cup y_0 \celto y_0.
	\end{equation*}
	Let $U_+$ be the shape of $x_+$. Then $(\eps{}{z}, \bord{}{+}\infl{\bord{}{+}U_+} \incliso \bord{}{+}U_+)$ is a dividend for $x_+$, so by assumption there exist a lax division $k\colon x_+^* \cup x_+ \celto \eps{}{z}$ and a colax division $k'\colon \eps{}{z} \celto x_+^* \cup x_+$ of $\eps{}{z}$ by $x_+$ with $k, k' \in B$. Then the pasting
	\begin{equation*}
		\tilde{k}_1 \eqdef k \cup \lunitor{z}{y_0}{-}\colon x_+^* \cup (x_+ \cup y_0) \celto y_0
	\end{equation*}
	along $\eps{}{z} \submol \bord{}{-}\lunitor{z}{y_0}{-}$ is defined and belongs to $\twothree{}{B \cup \degg{X}}$ by Lemma \ref{lem:pasting_twothree}. \\
	Let $U_0$ be the shape of $x_0$. The shape of $x_+ \cup y_0$ is $U_+ \cup W_0$ and $\bord{}{-}(U_+ \cup W_0)$ is isomorphic to $\bord{}{-}W_0[\bord{}{-}U_+/\bord{}{+}U_+]$. Through $\imath$, this contains $\bord{}{-}U[\bord{}{-}U_+/\bord{}{+}U_+]$ as a submolecule, which by Proposition \ref{prop:spherical_pasting} is isomorphic to $\bord{}{-}U_0$. Thus we have an inclusion $\imath_0\colon \bord{}{-}U_0 \incl \bord{}{-}(U_+ \cup W_0)$ which makes $(x_+ \cup y_0, \imath_0)$ a dividend for $x_0$. \\
	 By assumption there is a lax division $\ell\colon x_0 \cup y_1 \celto x_+ \cup y_0$ of $x_+ \cup y_0$ by $x_0$ with $\ell \in B$. We construct the pasting
	 \begin{equation*}
	 	\tilde{k}_2 \eqdef \ell \cup \tilde{k}_1\colon x^*_+ \cup (x_0 \cup y_1) \celto y_0
	 \end{equation*}
	 along $x_+ \cup y_0 \submol \bord{}{-}\tilde{k}_1$, then the pasting
	 \begin{equation*}
	 	\tilde{k}_3 \eqdef h \cup \tilde{k}_2\colon ((x^*_+ \cup x_+) \cup x) \cup y_1 \celto y_0
	 \end{equation*}
	 along $x_0 \submol \bord{}{-}\tilde{k}_2$, then the pasting
	 \begin{equation*}
	 	\tilde{k}_4 \eqdef k' \cup \tilde{k}_3\colon (\eps{}{z} \cup x) \cup y_1 \celto y_0
	 \end{equation*}
	 along $x^*_+ \cup x_+ \submol \bord{}{-}\tilde{k}_3$, and finally the pasting
	 \begin{equation*}
	 	\tilde{k}_5 \eqdef \lunitor{z}{x}{+} \cup \tilde{k}_4\colon x \cup y_1 \celto y_0
	 \end{equation*}
	 along $\eps{}{z} \cup x \submol \bord{}{-}\tilde{k}_4$. Then $\tilde{k}_5$ is a lax division of $(y_0, \imath)$ by $x$ which by Lemma \ref{lem:pasting_twothree} belongs to $\twothree{5}{B \cup \degg{X}}$. \\
	 Using Lemma \ref{lem:diagram_reverse} to reverse $h$, unbiasedness of $B$ to reverse $\ell$, and $\rev{-}$ to reverse the unitors, we can dually construct a composable diagram of type $y_0 \celto x \cup y_1$ in $\twothree{7}{B \cup \degg{X}}$. This proves that $x \in \equifun{}{\twothree{7}{B \cup \degg{X}}}$ and that $\twothree{7}{B \cup \degg{X}}$ is unbiased for $x$. \\
	 The case where we start from $h\colon x_0 \celto x_+ \cup x$ is dual.
	 
	 \item \emph{There is a composable diagram $h\colon x \cup x_- \celto x_0$ with $h, x_-, x_0 \in A$ and $z \eqdef \bord{}{-}x_- \submol \bord{}{+}x$.} Let $U_0$ be the shape of $x_0$. Then $\bord{}{-}U = \bord{}{-}U_0$, and $(y_0, \imath)$ is also a dividend for $x_0$. \\
	 By assumption there exist a lax division $k\colon x_0 \cup y \celto y_0$ and a colax division $k'\colon y_0 \celto x_0 \cup y$ of $y_0$ by $x_0$ with $k, k' \in B$. Then the pasting 
	 \begin{equation*}
	 	h \cup k\colon x \cup (x_- \cup y) \celto y_0
	 \end{equation*}
	 along $x_0 \submol \bord{}{-}k$ is defined, is a lax division of $(y_0, \imath)$ by $x$, and belongs to $\twothree{}{B \cup \degg{X}}$. \\
	 By Lemma \ref{lem:diagram_reverse} applied to $h \in B \cap \equifun{}{B}$, we also have a composable diagram $h^*\colon x_0 \celto x \cup x_-$ in $\twothree{2}{B}$, and the pasting
	 \begin{equation*}
	 	k' \cup h^*\colon y_0 \celto x \cup (x_- \cup y)
	 \end{equation*}
	 along $x_0 \submol \bord{}{+}k'$ is a colax division of $(y_0, \imath)$ by $x$ in $\twothree{3}{B \cup \degg{X}}$. This proves that $x \in \equifun{}{\twothree{3}{B \cup \degg{X}}}$ and that $\twothree{3}{B \cup \degg{X}}$ is unbiased for $x$. \\
	The case where we start from $h\colon x_0 \celto x \cup x_-$ is dual. 
\end{itemize}
This completes the case distinction and the proof.
\end{proof}

\begin{thm} \label{thm:twothree}
$\equi{X} = \twothree{\infty}{\equi{X} \cup \degg{X}}$.
\end{thm}
\begin{proof}
The inclusion $\equi{X} \subseteq \twothree{\infty}{\equi{X} \cup \degg{X}}$ is obvious. Conversely, we prove by induction that, for all $n \in \mathbb{N}$,
\begin{enumerate}
	\item $\twothree{n}{\equi{X} \cup \degg{X}} \subseteq \equifun{}{\twothree{\infty}{\equi{X} \cup \degg{X}}}$ and
	\item $\twothree{\infty}{\equi{X} \cup \degg{X}}$ is unbiased for $\twothree{n}{\equi{X} \cup \degg{X}}$.
\end{enumerate}
The base case is given by the combination of Lemma \ref{lem:equi_unbiased} and Lemma \ref{lem:degg_unbiased}. The inductive step is given by Lemma \ref{lem:twothree_inductive} with $A \eqdef \twothree{n}{\equi{X} \cup \degg{X}}$ and $B \eqdef \twothree{\infty}{\equi{X} \cup \degg{X}}$.

This proves that $\twothree{\infty}{\equi{X} \cup \degg{X}} \subseteq \equifun{}{\twothree{\infty}{\equi{X} \cup \degg{X}}}$. By coinduction, $\twothree{\infty}{\equi{X} \cup \degg{X}} \subseteq \equi{X}$.
\end{proof}

\begin{cor} \label{cor:twothree}
The following facts hold in every diagrammatic set.
\begin{enumerate}[label=(\alph*)]
	\item Every degenerate composable diagram is an equivalence.
	\item If $h\colon x \celto y$ is an equivalence, then there is an equivalence $h^*\colon y \celto x$.
	\item If $x_1$ and $x_2$ are equivalences and the pasting $x_1 \cup x_2$ along $y \submol \bord{}{\alpha}x_2$ is defined, then $x_1 \cup x_2$ is an equivalence.
	\item If $x$ is a composable $n$\nbd diagram and each $n$\nbd cell in $x$ is an equivalence, then $x$ is an equivalence.
\end{enumerate}
\end{cor}
\begin{proof}
The first point follows immediately from Theorem \ref{thm:twothree}, and the other three in combination with Proposition \ref{prop:equivalence_reverse}, Lemma \ref{lem:pasting_twothree}, and Lemma \ref{lem:cell_to_diagram}, respectively.
\end{proof}

\begin{dfn}[Equivalent diagrams]
Let $x, y$ be parallel composable diagrams in a diagrammatic set. We write $x \simeq y$, and say that \emph{$x$ is equivalent to $y$}, if there exists an equivalence $h\colon x \celto y$.
\end{dfn}

\begin{prop} \label{prop:equiv_relation}
The relation $\simeq$ is an equivalence relation.
\end{prop}
\begin{proof}
Let $x, y, z$ be parallel composable $n$\nbd diagrams. The unit $\eps{}{x}\colon x \celto x$ is a degenerate composable diagram, so by Corollary \ref{cor:twothree} it is an equivalence, exhibiting $x \simeq x$. This proves that $\simeq$ is reflexive.

Suppose $x \simeq y$, exhibited by an equivalence $h\colon x \celto y$. By Corollary \ref{cor:twothree} there is an equivalence $h^*\colon y \celto x$, exhibiting $y \simeq x$. This proves that $\simeq$ is symmetric.

Finally, if $x \simeq y$ and $y \simeq z$, exhibited by equivalences $h\colon x \celto y$ and $k\colon y \celto z$, we have a composable diagram $h \cp{n} k\colon x \celto z$ which by Corollary \ref{cor:twothree} is an equivalence, exhibiting $x \simeq z$. This proves that $\simeq$ is transitive.
\end{proof}

\begin{prop} \label{prop:pasting_equivalence}
Let $x, x', y, y'$ be composable diagrams such that $x \simeq x'$, $y \simeq y'$, and $y \submol x$. Then $x' \simeq x[y'/y]$.
\end{prop}
\begin{proof}
Let $h\colon x' \celto x$ and $k\colon y \celto y'$ be equivalences exhibiting $x' \simeq x$ and $y \simeq y'$. The pasting $h \cup k$ along $y \submol \bord{}{+}h$ is defined, has type $x' \celto x[y'/y]$, and is an equivalence by Corollary \ref{cor:twothree}.
\end{proof}

\begin{prop} \label{prop:higher_equivalence}
Let $x, y$ be parallel composable diagrams in a diagrammatic set. If $x \simeq y$ and $x$ is an equivalence, then $y$ is an equivalence.
\end{prop}
\begin{proof}
Suppose $x$ is an equivalence and let $(y_0, \imath)$ be a dividend for $y$. Because $x$ and $y$ are parallel, this is also a dividend for $x$. Then there is a lax division $h\colon x \cup x' \celto y_0$ of $y_0$ by $x$ which is an equivalence. Let $k\colon y \celto x$ be an equivalence exhibiting $y \simeq x$. The pasting $k \cup h$ along $x \submol \bord{}{-}h$ is defined, is a lax division of $y_0$ by $y$, and is an equivalence by Corollary \ref{cor:twothree}. 
\end{proof}

\begin{prop} \label{prop:unitors}
Let $x$ be a composable $n$\nbd diagram of type $y^- \celto y^+$. Then $x \simeq \eps{}{y^-} \cp{n-1} x$ and $x \simeq x \cp{n-1} \eps{}{y^+}$.
\end{prop}
\begin{proof}
The two equivalences are exhibited by the left unitors $\lunitor{y}{x}{\alpha}$ and the right unitors $\runitor{y}{x}{\alpha}$, respectively.
\end{proof}

\begin{prop} \label{prop:what_unit_does}
Let $e\colon x \celto x$ be a composable $n$\nbd diagram. The following are equivalent:
\begin{enumerate}
	\item $e \simeq \eps{}{x}$;
	\item $y \simeq e \cp{n-1} y$ and $z \simeq z \cp{n-1} e$ for all composable $n$\nbd diagrams $y\colon x \celto x'$ and $z\colon x' \celto x$.
\end{enumerate}
\end{prop}
\begin{proof}
Suppose $e \simeq \eps{}{x}$. The second statement follows immediately from Proposition \ref{prop:unitors} combined with Proposition \ref{prop:pasting_equivalence}. Conversely, by Proposition \ref{prop:unitors} we have $e \simeq e \cp{n-1} \eps{}{x} \simeq \eps{}{x}$.
\end{proof}

\begin{comm}
Proposition \ref{prop:what_unit_does} implies that the choice of units and unitors based on left rather than right cylinders, as discussed in Remark \ref{rmk:unit_biased}, is inessential: units constructed with right cylinders satisfy the second condition, so they are equivalent to units constructed with left cylinders.
\end{comm}


\subsection{Weak invertibility}

The following definition is based on \cite{cheng2007omega} and \cite[Section 4.2]{lafont2010folk}.

\begin{dfn}[Weakly invertible diagram]
Let $e\colon x \celto y$ be a composable $n$\nbd diagram in a diagrammatic set $X$. Coinductively, we say that $e$ is \emph{weakly invertible} if there exist a composable diagram $e^*\colon y \celto x$ and weakly invertible diagrams $h\colon e \cp{n-1} e^* \celto \eps{}{x}$ and $h'\colon e^* \cp{n-1} e \celto \eps{}{y}$. In this case, $e^*$ is a \emph{weak inverse} of $e$.

We write $\winv{X}$ for the set of weakly invertible diagrams of $X$. This admits the following proof principle. For each subset $A \subseteq \cdgm{X}$, let
\begin{align*}
	\winfun{A} \eqdef \; & \{ e\colon x \celto y \in \cdgm{X} \mid \text{there exist $e^*\colon y \celto x$ and} \\
	& \text{$h\colon e \cp{n-1} e^* \celto \eps{}{x}, h'\colon e^* \cp{n-1} e \celto \eps{}{y}$ with $h, h' \in A$} \}.
\end{align*}
If $A \subseteq \winfun{A}$, then $A \subseteq \winv{X}$. 
\end{dfn}

\begin{rmk}
Unlike the notion of equivalence, the notion of weakly invertible diagram is clearly self-dual: if $e$ is weakly invertible, then any weak inverse $e^*$ is also weakly invertible.
\end{rmk}

\begin{lem} \label{lem:equivalence_is_invertible}
Let $e$ be an equivalence. Then $e$ is weakly invertible.
\end{lem}
\begin{proof}
Suppose $e$ has shape $U$ and type $x \celto y$. Then $(\eps{}{x}, \bord{}{-}\infl{\bord{}{-}U} \incliso \bord{}{-}U)$ and $(\eps{}{y}, \bord{}{+}\infl{\bord{}{+}U} \incliso \bord{}{+}U)$ are both dividends for $e$, so there exist lax divisions $h\colon e \cp{n-1} e^* \celto \eps{}{x}$ and $k\colon e' \cp{n-1} e \celto \eps{}{y}$ such that $h, k$ are equivalences. Moreover,
\begin{equation*}
	e' \simeq e' \cp{n-1} \eps{}{x} \simeq e' \cp{n-1} e \cp{n-1} e^* \simeq \eps{}{y} \cp{n-1} e^* \simeq e^*
\end{equation*}
by Proposition \ref{prop:pasting_equivalence} and Proposition \ref{prop:unitors}. If $k'$ is an equivalence exhibiting $e^* \simeq e'$, pasting $k'$ and $k$ along $e' \submol \bord{}{-}k$ produces an equivalence 
\begin{equation*}
	h' \eqdef k' \cup k\colon  e^* \cp{n-1} e \celto \eps{}{y}. 
\end{equation*}
This proves that $\equi{X} \subseteq \winfun{\equi{X}}$, so by coinduction $\equi{X} \subseteq \winv{X}$. 
\end{proof}

\begin{dfn}
Let $A \subseteq \cdgm{X}$ and let $x$ be a composable diagram in $X$. We let $x \in \pasting{}{A}$ if either $x \in A$, or $x$ is the pasting of two diagrams $x_1, x_2 \in A$ along a subdiagram $y \submol \bord{}{\alpha}x_2$. We let $\pasting{\infty}{A} \eqdef \bigcup_{n \in \mathbb{N}} \pasting{n}{A}$.
\end{dfn}

\begin{lem} \label{lem:pasting_invertible}
Let $A \subseteq \cdgm{X}$ and suppose that $A \subseteq \winfun{\pasting{\infty}{A}}$. Then $\pasting{}{A} \subseteq \winfun{\pasting{\infty}{A \cup \degg{X}}}$.
\end{lem}
\begin{proof}
Let $e_1, e_2 \in A$ be such that the pasting $e_1 \cup e_2$ along $y_1 \submol \bord{}{\alpha}e_2$ is defined. Without loss of generality, suppose $\alpha = -$ and let $x_i \celto y_i$ be the type of $e_i$. Then the type of $e \eqdef e_1 \cup e_2$ is $x \celto y$ for $x \eqdef x_2[x_1/y_1]$ and $y \eqdef y_2$.

By assumption, if $e_1$ and $e_2$ are $n$\nbd diagrams, there are diagrams 
\begin{equation*}
	e^*_i\colon y_i \celto x_i, \quad h_i\colon e_i \cp{n-1}e^*_i \celto \eps{}{x_i}, \quad h'_i\colon e^*_i \cp{n-1}e_i \celto \eps{}{y_i}
\end{equation*}
such that $h_i, h'_i \in \pasting{\infty}{A}$. The pasting $e^* \eqdef e_2^* \cup e_1^*$ along $y_1 \submol \bord{}{+}e_2$ is defined, and so are 
\begin{equation*}
	e \cp{n-1} e^*\colon x \celto x, \quad \quad e^* \cp{n-1} e\colon y \celto y.
\end{equation*}
Let $k_1$ be the left unitor $\lunitor{x}{e \cp{n-1} e^*}{+} \colon e \cp{n-1} e^* \celto \eps{}{x} \cp{n-1} (e \cp{n-1} e^*)$. We construct the pasting
\begin{equation*}
	k_2 \eqdef k_1 \cup h_2\colon e \cp{n-1} e^* \celto ((\eps{}{x} \cup e_1) \cup \eps{}{x_2}) \cup e^*_1
\end{equation*}
along $e_2 \cp{n-1} e_2^* \submol \bord{}{+}k_1$, then the pasting
\begin{equation*}
	k_3 \eqdef k_2 \cup \runitor{x_2}{\eps{}{x} \cup e_1}{-}\colon e \cp{n-1} e^* \celto \eps{}{x} \cup (e_1 \cp{n-1} e_1^*)
\end{equation*}
along $(\eps{}{x} \cup e_1) \cup \eps{}{x_2} \submol \bord{}{+}k_2$, then the pasting
\begin{equation*}
	k_4 \eqdef k_3 \cup h_1\colon e \cp{n-1} e^* \celto \eps{}{x} \cup \eps{}{x_1}
\end{equation*}
along $e_1 \cp{n-1} e_1^* \submol \bord{}{+}k_3$, and finally the pasting
\begin{equation*}
	k \eqdef k_4 \cp{n} \runitor{x_1}{\eps{}{x}}{-}\colon e \cp{n-1} e^* \celto \eps{}{x}.
\end{equation*}
In the other direction, let $k'_1$ be the unitor $\lunitor{y_1}{e_2}{-}$. We construct the pasting
\begin{equation*}
	k'_2 \eqdef k'_1 \cup h'_2\colon e^*_2 \cup (\eps{}{y_1} \cup e_2) \celto \eps{}{y}
\end{equation*}
along $e_2 \submol \bord{}{-}k'_1$, then the pasting 
\begin{equation*}
	k' \eqdef h'_1 \cup k'_2\colon e^* \cp{n-1} e \celto \eps{}{y}
\end{equation*}
along $\eps{}{y_1} \submol \bord{}{-}k'_2$. 

Both $k$ and $k'$ are constructed by iterated pasting of diagrams in $\pasting{\infty}{A}$ and diagrams in $\degg{X}$, so $k, k' \in \pasting{\infty}{A \cup \degg{X}}$. 
\end{proof}

\begin{prop} \label{prop:pasting_invertible}
$\winv{X} = \pasting{\infty}{\winv{X}}$.
\end{prop}
\begin{proof}
We prove by induction that $\pasting{n}{\winv{X}} \subseteq \winfun{\pasting{\infty}{\winv{X}}}$ for all $n \geq 0$. The base case is $\winv{X} = \winfun{\winv{X}} \subseteq \winfun{\pasting{\infty}{\winv{X}}}$. For the inductive step, we use Lemma \ref{lem:pasting_invertible} together with the fact that $\degg{X} \subseteq \equi{X} \subseteq \winv{X}$ by Lemma \ref{lem:equivalence_is_invertible}. 

It follows that $\pasting{\infty}{\winv{X}} \subseteq \winfun{\pasting{\infty}{\winv{X}}}$, and we conclude by coinduction.
\end{proof}

\begin{thm} \label{thm:equivalence_equals_invertible}
A composable diagram is an equivalence if and only if it is weakly invertible.
\end{thm}
\begin{proof}
We have already proved $\equi{X} \subseteq \winv{X}$. Let $e\colon x \celto y$ be a weakly invertible $n$\nbd diagram of shape $U$ and let $(x_0\colon U_0 \to X, \imath\colon \bord{}{\alpha}U \incl \bord{}{\alpha}U_0)$ be a dividend for $e$. Without loss of generality suppose $\alpha = -$. 

There is a weakly invertible diagram $h\colon e \cp{n-1} e^* \celto \eps{}{x}$. The pasting
\begin{equation*}
	h \cup \lunitor{x}{x_0}{-}\colon e \cup (e^* \cup x_0) \celto x_0
\end{equation*}
along $\eps{}{x} \submol \bord{}{-}\lunitor{x}{x_0}{-}$ is defined and is a lax division of $(x_0, \imath)$ by $e$. Because it is the pasting of two weakly invertible diagrams, by Proposition \ref{prop:pasting_invertible} it is weakly invertible. 

This proves that $\winv{X} \subseteq \equifun{}{\winv{X}}$, and by coinduction $\winv{X} \subseteq \equi{X}$. 
\end{proof}

\begin{prop} \label{prop:morphism_preserve_equivalence}
Let $f\colon X \to Y$ be a morphism of diagrammatic sets. Then $f$ sends equivalences to equivalences.
\end{prop}
\begin{proof}
By Theorem \ref{thm:equivalence_equals_invertible}, it suffices to show that $f$ sends weakly invertible cells to weakly invertible cells. Suppose $e\colon x \celto y$ is weakly invertible in $X$. There is a weak inverse $e^*\colon y \celto x$ and weakly invertible diagrams $h\colon e \cp{n-1} e^* \celto \eps{}{x}$ and $h'\colon e^* \cp{n-1} e \celto \eps{}{y}$. Every morphism of diagrammatic sets is compatible with units and pasting, so through $f$ we obtain composable diagrams
\begin{equation*}
	f(h)\colon f(e) \cp{n-1} f(e^*) \celto \eps{}{f(x)}, \quad \quad f(h')\colon f(e^*) \cp{n-1} f(e) \celto \eps{}{f(y)}
\end{equation*}
in $Y$. This proves that $f(\winv{X}) \subseteq \winfun{f(\winv{X})}$, and by coinduction we conclude that $f(\winv{X}) \subseteq \winv{Y}$. 
\end{proof}

\begin{cor} \label{cor:morphism_preserve_equivalence}
Let $f\colon X \to Y$ be a morphism of diagrammatic sets. If $x \simeq y$ in $X$, then $f(x) \simeq f(y)$ in $Y$.
\end{cor}
\begin{proof}
If $h\colon x \celto y$ exhibits the equivalence $x \simeq y$, then $f(h)\colon f(x) \celto f(y)$ exhibits $f(x) \simeq f(y)$. 
\end{proof}

\section{Diagrammatic $\omega$-categories} \label{sec:omegacats}

\subsection{Weak composites}

\begin{dfn}[Weak composite] 
Let $x$ be a composable diagram of shape $U$. A \emph{weak composite} of $x$ is a cell $\compos{x}$ of shape $\compos{U}$ such that $x \simeq \compos{x}$.

A diagrammatic set $X$ \emph{has weak composites} if every composable diagram in $X$ has a weak composite.
\end{dfn}

\begin{comm}
A diagrammatic set with weak composites is our candidate for a weak model of $\omega$\nbd categories. 

A good adjective for a diagrammatic set with weak composites would be \emph{representable}, modelled on Claudio Hermida's representable multicategories \cite{hermida2000representable}, the idea being that composable diagrams are ``represented'' by individual cells. We choose not to use the term to avoid confusion with diagrammatic sets that are representable as presheaves. 
\end{comm}

\begin{prop} \label{prop:basic_weak_composite}
Let $f\colon X \to Y$ be a morphism of diagrammatic sets and $x$ a composable diagram in $X$. For each weak composite $\compos{x}$ of $x$,
\begin{enumerate}[label=(\alph*)]
	\item if $\compos{x}'$ is another weak composite of $x$, then $\compos{x} \simeq \compos{x}'$,
	\item if $x$ is an equivalence, then $\compos{x}$ is an equivalence,
	\item $f(\compos{x})$ is a weak composite of $f(x)$.
\end{enumerate}
\end{prop}
\begin{proof}
Immediate consequences of Proposition \ref{prop:equiv_relation}, Proposition \ref{prop:higher_equivalence}, and Corollary \ref{cor:morphism_preserve_equivalence}, respectively.
\end{proof}

\begin{dfn}[Localisation] 
Let $X$ be a diagrammatic set and $A \subseteq \cdgm{X}$. We define $X\{\invrs{A}\}$ to be the diagrammatic set obtained from $X$ in the following two steps:
\begin{enumerate}
	\item attach an $n$\nbd cell $x^*\colon y^+ \celto y^-$ for each $n$\nbd diagram $x\colon y^- \celto y^+$ in $A$, then
	\item attach $(n+1)$\nbd cells $h(x)\colon x \cp{n-1} x^* \celto \eps{}{y^-}$ and $h'(x)\colon x^* \cp{n-1} x \celto \eps{}{y^+}$ for each $n$\nbd diagram $x\colon y^- \celto y^+$ in $A$.
\end{enumerate}
Let $X_0 \eqdef X$ and $A_0 \eqdef A$. For all $n > 0$, let
\begin{equation*}
	X_n \eqdef X_{n-1}\{\invrs{A_{n-1}}\}, \quad \quad A_n \eqdef \{h(x), h'(x) \mid x \in A_{n-1}\}.
\end{equation*}
We have a sequence $\{X_{n-1} \incl X_n\}_{n \in \mathbb{N}}$ of inclusions of diagrammatic sets. The \emph{localisation of $X$ at $A$} is the colimit $X[\invrs{A}]$ of this sequence. 

By construction, for all $n$ the images of the diagrams of $A_n$ through the inclusions $X_n \incl X[\invrs{A}]$ are weakly invertible. In particular, all diagrams in $A$ become equivalences in $X[\invrs{A}]$. 
\end{dfn}

\begin{prop} \label{prop:localisation_universal}
Let $f\colon X \to Y$ be a morphism of diagrammatic sets and $A \subseteq \cdgm{X}$. Suppose that
\begin{enumerate}
	\item $f(x)$ is an equivalence in $Y$ for all $x \in A$ and
	\item $Y$ has weak composites.
\end{enumerate}
Then there exists a morphism $\tilde{f}\colon X[\invrs{A}] \to Y$ such that the triangle
\begin{equation} \label{eq:localisation}
\begin{tikzpicture}[baseline={([yshift=-.5ex]current bounding box.center)}]
	\node (0) at (-1.5,1.25) {$X$};
	\node (1) at (1.5,1.25) {$Y$};
	\node (2) at (0,0) {$X[\invrs{A}]$};
	\draw[1c] (0) to node[auto,arlabel] {$f$} (1);
	\draw[1cinc] (0) to node[auto,swap,arlabel] {$\imath$} (2);
	\draw[1c] (2) to node[auto,swap,arlabel] {$\tilde{f}$} (1);
\end{tikzpicture}
\end{equation}
commutes.
\end{prop}
\begin{proof}
By definition $X[\invrs{A}]$ is the colimit of the sequence defined by the recursion $X_n \eqdef X_{n-1}\{\invrs{A}_{n-1}\}$, $A_n \eqdef \{h(x), h'(x) \mid x \in A_{n-1}\}$ and starting with 
\begin{equation*}
	X_0 \eqdef X, \quad \quad A_0 \eqdef A. 
\end{equation*}
We define $\tilde{f}$ by successive extensions from $X_{n-1}$ to $X_n$ with the property that $\restr{\tilde{f}}{X_n}(x)$ is an equivalence in $Y$ for all $x \in A_n$. 

We let $\restr{\tilde{f}}{X_0} \eqdef f$. By assumption, $\restr{\tilde{f}}{X_0}(x) = f(x)$ is an equivalence in $Y$ for all $x \in A_0$. If $n > 0$, assume that $g \eqdef \restr{\tilde{f}}{X_{n-1}}$ has been defined and that it sends all $x \in A_{n-1}$ to equivalences. Let $x \in A_{n-1}$ be an $m$\nbd diagram. Then $g(x)\colon y^- \celto y^+$ has a weak inverse $g(x)^*\colon y^+ \celto y^-$ and there are equivalences $k\colon g(x)\cp{m-1}g(x)^* \celto \eps{}{y^-}$ and $k'\colon g(x)^*\cp{m-1}g(x) \celto \eps{}{y^+}$. 

Take a weak composite $\compos{g(x)^*}$ and an equivalence $e\colon \compos{g(x)^*} \celto g(x)^*$. The pastings $e \cup k$ and $e \cup k'$ along $g(x)^* \submol \bord{}{-}k, \bord{}{-}k'$ are still equivalences, exhibiting $\compos{g(x)^*}$ as a weak inverse of $g(x)$. Finally, we take weak composites $\compos{e \cup k}$ and $\compos{e \cup k'}$. Letting
\begin{equation*}
	x^* \mapsto \compos{g(x)^*}, \quad \quad h(x) \mapsto \compos{e \cup k}, \quad \quad h'(x) \mapsto \compos{e \cup k'}
\end{equation*}
for each $x \in A_{n-1}$ determines an extension of $\restr{\tilde{f}}{X_{n-1}}$ to $X_n$, which by construction sends cells in $A_n$ to equivalences in $Y$. 

Passing to the colimit of the $\restr{\tilde{f}}{X_n}$, we obtain a morphism $\tilde{f}$ which extends $f$, so that the triangle (\ref{eq:localisation}) commutes. 
\end{proof}

\begin{dfn}
Let $U, V \in \cls{C}$ be molecules such that $U \celto V$ is defined. We let $U \simeq V$ be the localisation of $U \celto V$ at $\{\idd{U \celto V}\}$. 
\end{dfn}

\begin{lem} \label{lem:walking_equivalence}
Let $\imath$ be the inclusion of $U \celto V$ into its localisation $U \simeq V$ and let $\tilde{x}\colon (U \simeq V) \to X$ be a morphism of diagrammatic sets. Then $\imath;\tilde{x}$ is an equivalence in $X$.
\end{lem}
\begin{proof}
By construction the cell $\imath\colon U \celto V \incl U \simeq V$ is an equivalence in $U \simeq V$, so by Proposition \ref{prop:morphism_preserve_equivalence} $\imath;\tilde{x} = \tilde{x}(\imath)$ is an equivalence in $X$.
\end{proof}

\begin{prop} \label{prop:composites_fibrant}
Let $X$ be a diagrammatic set. The following are equivalent:
\begin{enumerate}
	\item $X$ has weak composites;
	\item for all composable diagrams $x$ of shape $U$ in $X$, there exists an equivalence $c(x)$ of shape $U \celto \compos{U}$ such that the triangle
\begin{equation} \label{eq:composition_filler}
\begin{tikzpicture}[baseline={([yshift=-.5ex]current bounding box.center)}]
	\node (0) at (0,1.5) {$U$};
	\node (1) at (2.5,1.5) {$X$};
	\node (2) at (0,0) {$U \celto \compos{U}$};
	\draw[1c] (0) to node[auto,arlabel] {$x$} (1);
	\draw[1cincl] (0) to node[auto,swap,arlabel] {$\imath^-$} (2);
	\draw[1c] (2) to node[auto,swap,arlabel] {$c(x)$} (1);
\end{tikzpicture}
\end{equation}
commutes;
	\item for all composable diagrams $x$ of shape $U$ in $X$, there exists a morphism $\tilde{c}(x)\colon (U \simeq \compos{U}) \to X$ such that the triangle 
\begin{equation} \label{eq:equivalence_filler}
\begin{tikzpicture}[baseline={([yshift=-.5ex]current bounding box.center)}]
	\node (0) at (0,1.5) {$U$};
	\node (1) at (2.5,1.5) {$X$};
	\node (2) at (0,0) {$U \simeq \compos{U}$};
	\draw[1c] (0) to node[auto,arlabel] {$x$} (1);
	\draw[1cincl] (0) to node[auto,swap,arlabel] {$\imath^-$} (2);
	\draw[1c] (2) to node[auto,swap,arlabel] {$\tilde{c}(x)$} (1);
\end{tikzpicture}
\end{equation}
commutes.
\end{enumerate}
\end{prop}
\begin{proof}
Let $x\colon U \to X$ be a composable diagram. If $X$ has weak composites, there is a cell $\compos{x}$ and an equivalence $e\colon x \celto \compos{x}$. In turn, $e$ has a weak composite $\compos{e}\colon x \celto \compos{x}$, which by Proposition \ref{prop:basic_weak_composite} is still an equivalence and has shape $U \celto \compos{U}$. Letting $c(x) \eqdef \compos{e}$ makes (\ref{eq:composition_filler}) commute. 

Now suppose that there exists an equivalence $c(x)$ that makes (\ref{eq:composition_filler}) commute. Applying Proposition \ref{prop:localisation_universal} with $f \eqdef c(x)$ and $A \eqdef \{\idd{U \celto \compos{U}}\}$, we find an extension $\tilde{c}(x)$ of $c(x)$ to $U \simeq \compos{U}$, which makes (\ref{eq:equivalence_filler}) commute.

Finally, suppose that $\tilde{c}(x)$ exists such that (\ref{eq:equivalence_filler}) commutes. By Lemma \ref{lem:walking_equivalence} its restriction to $U \celto \compos{U}$ is an equivalence, exhibiting its output boundary as a weak composite of $x$.
\end{proof}

\begin{rmk} \label{rmk:model_structure}
Proposition \ref{prop:composites_fibrant} displays the class of diagrammatic sets with weak composites as characterised by a right lifting property with respect to the set of morphisms $J \eqdef \{U \incl (U \simeq \compos{U})\}_{U \in \cls{C}}$. 

This suggests that diagrammatic sets with weak composites are the fibrant objects of a cofibrantly generated model structure on $\dgmset$, whose set of generating acyclic cofibrations is $J$. A natural candidate for a set of generating cofibrations is $I \eqdef \{\bord U \incl U\}$ with $U$ ranging over the atoms in $\cls{C}$. 

We leave the development of such a model structure to future work.
\end{rmk}

\begin{cor} \label{cor:product_closure}
Let $\{X_i\}_{i \in I}$ be a family of diagrammatic sets with weak composites. Then its coproduct $\coprod_{i \in I} X_i$ and its product $\prod_{i \in I} X_i$ have weak composites.
\end{cor}
\begin{proof}
Follows from general properties of classes of objects defined by a right lifting property.
\end{proof}

\begin{prop}
Let $\{X_i\}_{i \in I}$ be a family of diagrammatic sets and let $X \eqdef \prod_{i \in I} X_i$ be its product. For each composable diagram $x$ in $X$,
\begin{enumerate}[label=(\alph*)]
	\item if $x$ is an equivalence, then each projection $\pi_i(x)$ is an equivalence, and
	\item if the $X_i$ have weak composites, the converse also holds.
\end{enumerate}
\end{prop}
\begin{proof}
If $x$ is an equivalence in $X$, by Proposition \ref{prop:morphism_preserve_equivalence} each projection $\pi_i(x)$ is an equivalence in $X_i$. 

Conversely, let $A \subseteq \cdgm{X}$ be the set of all composable diagrams whose projections are all equivalences, and let $x \in A$. If $(x_0, \imath)$ is a dividend for $x$, then $(\pi_i(x_0), \imath)$ is a dividend for $\pi_i(x)$ for each $i \in I$. By assumption, there is a lax division $h_i$ of $\pi_i(x_0)$ by $\pi_i(x)$ which is an equivalence in $X_i$. 

If we replace $h_i$ with a weak composite $\compos{h_i}$, then all the $\compos{h_i}$ have the same shape $U$, so they induce a cell $(\compos{h_i})_{i\in I}\colon U \to X$ which is a lax division of $x_0$ by $x$. By construction $(\compos{h_i})_{i\in I} \in A$. This proves that $A \subseteq \equifun{}{A}$ and by coinduction $A \subseteq \equi{X}$. 
\end{proof}


\subsection{Inflate and merge}

\begin{dfn}
We write $\atomin$ for the subcategory of atoms and inclusions in $\atom$.
\end{dfn}

\begin{dfn}[Combinatorial $\cls{C}$\nbd polygraph] 
A \emph{combinatorial $\cls{C}$\nbd polygraph} is a presheaf on $\atomin$. Combinatorial $\cls{C}$\nbd polygraphs and their morphisms form a category $\cpol$. 

When there is no ambiguity, we will omit the adjective ``combinatorial'' and speak simply of $\cls{C}$\nbd polygraphs.
\end{dfn}

\begin{dfn}
We let $\Gamma$ be the set of colimit diagrams in $\dcpxin$ comprising the initial object diagram and all pushout diagrams of inclusions, preserved both by the inclusion $\dcpxin \incl \dcpx$ and by the inclusion $\dcpxin \incl \dcpxfun$.
\end{dfn}

\begin{dfn}
The embedding $\atomin \incl \cpol$ extends to an embedding $\dcpxin \incl \cpol$ through an equivalence between $\cpol$ and $\psh{\Gamma}{\dcpxin}$ as in the case of diagrammatic sets.

Every diagrammatic set has an underlying $\cls{C}$\nbd polygraph, via the restriction functor $\dgmset \to \cpol$. We carry over from diagrammatic sets to $\cls{C}$\nbd polygraphs all the terminology that does not involve surjective maps in $\atom$. Thus we may speak of diagrams, composable diagrams, cells in a $\cls{C}$\nbd polygraph, their boundaries, substitution, pasting, and attaching of cells. We may \emph{not} speak of degenerate diagrams, units, or unitors in a $\cls{C}$\nbd polygraph. 
\end{dfn}

\begin{dfn}
If $\monad{}$ is a monad, we let $\algmon{\monad{}}$ denote its Eilenberg-Moore category.
\end{dfn}

\begin{dfn}[Inflate monad]
Let $X$ be a $\cls{C}$\nbd polygraph. We define a $\cls{C}$\nbd polygraph $\imonad X$ as follows.
\begin{enumerate}
	\item Cells $U \to \imonad X$ are pairs $(p\colon U \surj V, x\colon V \to X)$ of a surjective map in $\atom$ and a cell in $X$.
	\item If $(p\colon U \surj V, x\colon V \to X)$ is a cell in $\imonad X$ and $\imath\colon U' \incl U$ an inclusion of atoms, let $p';\imath'$ be the unique factorisation of $\imath;p$ as a surjective map $p'\colon U' \surj V'$ followed by an inclusion $\imath'\colon V' \incl V$ in $\atom$. We define $\imath;(p,x) \eqdef (p', \imath';x)$ as a cell of $\imonad X$.
\end{enumerate}
The factorisation is strictly unique because $\atomin$ is skeletal and $V'$ is an atom, so $V'$ has no non-trivial automorphisms.

If $f\colon X \to Y$ is a morphism of $\cls{C}$\nbd polygraphs, we let $\imonad f\colon \imonad X \to \imonad Y$ be defined by $(p,x) \mapsto (p,f(x))$. This determines an endofunctor $\imonad$ on $\cpol$. 

We define natural transformations $\mu^\imonad\colon \imonad \imonad \to \imonad$ and $\eta^\imonad\colon \bigid{} \to \imonad$. For each $\cls{C}$\nbd polygraph $X$,
\begin{itemize}
	\item $\mu^\imonad_X$ sends a cell $(p_1, (p_2, x))$ in $\imonad \imonad X$ to the cell $(p_1;p_2, x)$ in $\imonad X$, and
	\item $\eta^\imonad_X$ sends a cell $x$ of shape $U$ in $X$ to the cell $(\idd{U}, x)$ in $\imonad X$.
\end{itemize}
Then $(\imonad, \mu^\imonad, \eta^\imonad)$ is a monad on $\cpol$. 
\end{dfn}

\begin{prop} \label{prop:imonad_monadic}
The forgetful functor $\dgmset \to \cpol$ is monadic. The induced monad is naturally isomorphic to $(\imonad, \mu^\imonad, \eta^\imonad)$. 
\end{prop}
\begin{proof}
Let $X$ be a diagrammatic set. We give its underlying $\cls{C}$\nbd polygraph an $\imonad$\nbd algebra structure $\imonad X \to X$ defined by $(p,x) \mapsto p;x$, the latter computed in $\dgmset$. If $f\colon X \to Y$ is a morphism of diagrammatic sets, its underlying morphism in $\cpol$ is compatible with this algebra structure.

Conversely, let $\alpha\colon \imonad X \to X$ be an $\imonad$\nbd algebra. We give the $\cls{C}$\nbd polygraph $X$ the structure of a diagrammatic set. If $x\colon V \to X$ is a cell of $X$ and $f\colon U \to V$ a map in $\atom$, let $f = p;i$ be the unique factorisation of $f$ as a surjective map followed by an inclusion. Then define $f;x \eqdef \alpha(p, \imath;x)$. Compatibility of a morphism of $\cls{C}$\nbd polygraphs with $\imonad$\nbd algebra structures amounts to compatibility with this action of maps in $\atom$. 

This defines opposite functors between $\dgmset$ and $\algmon{\imonad}$ over the respective forgetful functors to $\cpol$. We leave it to the reader to check that these functors are each other's weak inverse.
\end{proof}

\begin{dfn}[Non-unital $\cls{C}$\nbd $\omega$\nbd category]
A \emph{non-unital $\cls{C}$\nbd $\omega$\nbd category} is a $\Gamma$\nbd continuous presheaf on $\dcpxfun$. Non-unital $\cls{C}$\nbd $\omega$\nbd categories and their morphisms form a category $\nucat$.
\end{dfn}

\begin{dfn}
We identify $\dcpxfun$ with a subcategory of $\nucat$ via the Yoneda embedding. 

There is a restriction functor $\nucat \to \psh{\Gamma}{\dcpxin}$, which via the equivalence between $\cpol$ and $\psh{\Gamma}{\dcpxin}$ becomes a forgetful functor $\nucat \to \cpol$.
\end{dfn}

\begin{dfn}[Merge monad]
Let $X$ be a $\cls{C}$\nbd polygraph. We define a $\cls{C}$\nbd polygraph $\mmonad{X}$ as follows.
\begin{enumerate}
	\item Cells $U \to \mmonad X$ are pairs $(s\colon U \cfun V, x\colon V \to X)$ of a subdivision in $\dcpxfun$ and a composable diagram in $X$.
	\item If $(s\colon U \cfun V, x\colon V \to X)$ is a cell in $\mmonad X$ and $\imath\colon U' \incl U$ an inclusion of atoms, let $s';\imath'$ be the unique factorisation of $\imath;s$ as a subdivision $s'\colon U' \cfun V'$ followed by an inclusion $\imath'\colon V' \incl V$ in $\dcpxfun$. We define $\imath;(s,x) \eqdef (s', \imath';x)$ as a cell of $\mmonad X$.
\end{enumerate}
The factorisation is strictly unique because $\dcpxfun$ is skeletal and $V'$ is a molecule, so $V'$ has no non-trivial automorphisms.

Let $x$ be a composable diagram of shape $U$ in $\mmonad X$ and let $\{U_i\}_{i \in I}$ be the atoms of $U$. For each $i \in I$, the restriction $\restr{x}{U_i}$ is uniquely of the form $(s_i\colon U_i \cfun V_i, y_i\colon V_i \to X)$ for some subdivision $s_i$ and composable diagram $y_i$ in $X$. By construction, for each inclusion $\imath_{ij}\colon U_i \incl U_j$ we have an inclusion $\imath'_{ij}\colon V_i \incl V_j$ with $\imath_{ij};s_j = s_i;\imath'_{ij}$ and $y_i = \imath'_{ij};y_j$. 

Let $V$ be the colimit of the diagram of inclusions $V_i \incl V_j$ with $i, j \in I$. Then the $s_i$ can be uniquely collated to a subdivision $s\colon U \cfun V$ and the $y_i$ to a morphism $y\colon V \to X$. Because $U \in \cls{C}$ and $\cls{C}$ is algebraic, $V \in \cls{C}$, so $y$ is a composable diagram in $X$. 

It follows that not only cells, but also composable diagrams in $\mmonad X$ can be uniquely represented as pairs $(s\colon U \cfun V, y\colon V \to X)$ of a subdivision and a composable diagram in $X$.

If $f\colon X \to Y$ is a morphism of $\cls{C}$\nbd polygraphs, we let $\mmonad f\colon \mmonad X \to \mmonad Y$ be defined by $(s,x) \mapsto (s,f(x))$. This determines an endofunctor $\mmonad$ on $\cpol$. 

We define natural transformations $\mu^\mmonad\colon \mmonad \mmonad \to \mmonad$ and $\eta^\mmonad\colon \bigid{} \to \mmonad$. For each $\cls{C}$\nbd polygraph $X$,
\begin{itemize}
	\item $\mu^\mmonad_X$ sends a cell $(s_1, (s_2, x))$ in $\mmonad \mmonad X$ to the cell $(s_1;s_2, x)$ in $\mmonad X$, and
	\item $\eta^\mmonad_X$ sends a cell $x$ of shape $U$ in $X$ to the cell $(\idd{U}, x)$ in $\mmonad X$.
\end{itemize}
Then $(\mmonad, \mu^\mmonad, \eta^\mmonad)$ is a monad on $\cpol$. 
\end{dfn}

\begin{prop}
Let $X$ be a $\cls{C}$\nbd polygraph. There is a bijective correspondence between
\begin{enumerate}
	\item composable $n$\nbd diagrams in $X$, and
	\item cells of shape $O^n$ in $\mmonad X$.
\end{enumerate}
\end{prop}
\begin{proof}
Let $x$ be a composable $n$\nbd diagram of shape $U$ in $X$. By Example \ref{exm:unique_subdivision}, there is a unique subdivision $s\colon O^n \cfun U$. Then $(s,x)$ is a cell of shape $O^n$ in $\mmonad X$. Conversely, each cell of shape $O^n$ in $\mmonad X$ is uniquely of the form $(s,x)$ for some composable $n$\nbd diagram $x$ in $X$.
\end{proof}

\begin{prop}
The forgetful functor $\nucat \to \cpol$ is monadic. The induced monad is naturally isomorphic to $(\mmonad, \mu^\mmonad, \eta^\mmonad)$.
\end{prop}
\begin{proof}
Let $X$ be a non-unital $\cls{C}$\nbd $\omega$\nbd category. Then $(s, x) \mapsto s;x$, the latter computed in $\nucat$, gives its underlying $\cls{C}$\nbd polygraph an $\mmonad$\nbd algebra structure. If $f\colon X \to Y$ is a morphism in $\nucat$, its underlying morphism in $\cpol$ is compatible with this structure. 

Conversely, let $\alpha\colon \mmonad X \to X$ be an $\mmonad$\nbd algebra. We give the $\cls{C}$\nbd polygraph $X$ the structure of a non-unital $\cls{C}$\nbd $\omega$\nbd category. First of all, we identify $X$ with a $\Gamma$\nbd continuous presheaf on $\dcpxin$ by letting $X(P)$ be the set of morphisms $P \to X$ in $\cpol$ for each $\cls{C}$\nbd directed complex $P$.

Let $f\colon P \cfun Q$ be a $\cls{C}$\nbd functor and $x\colon Q \to X$ a morphism in $\cpol$. Because $P$ is constructed from its atoms by colimits in $\Gamma$, it suffices to define $\restr{f}{U}; x\colon U \to X$ for all atoms $U \subseteq P$. Now, $\restr{f}{U}$ factors uniquely as a subdivision $s\colon U \cfun V$ followed by an inclusion $\imath\colon V \incl Q$ where $V$ is a $\cls{C}$\nbd molecule. We define $\restr{f}{U}; x \eqdef \alpha(s, \imath;x)$. This is compatible with inclusions of atoms so it extends uniquely to a morphism $f;x\colon P \to X$. Compatibility of a morphism of $\cls{C}$\nbd polygraphs with $\mmonad$\nbd algebra structures amounts to compatibility with this action of $\cls{C}$\nbd functors. 

We have defined opposite functors between $\nucat$ and $\algmon{\mmonad}$ over the respective forgetful functors to $\cpol$. We leave it to the reader to check that these functors are each other's weak inverse.
\end{proof}

\begin{prop} \label{prop:finitary_lfp}
$\mmonad$ and $\imonad$ are finitary monads on a locally finitely presentable category.
\end{prop}
\begin{proof}
The category $\cpol$ is a category of presheaves on a small category, so by \cite[Theorem 1.11]{adamek1994locally} it is locally finitely presentable. 

By inspection, $\mmonad$ and $\imonad$ preserve monomorphisms. Moreover, every cell $(s,x)$ in $\mmonad X$ and $(p, x)$ in $\imonad X$ lies in $\mmonad Y \incl \mmonad X$ and $\imonad Y \incl \imonad X$, respectively, for a finite sub-presheaf $Y \incl X$: namely, the one containing only the cells $y \submol x$, of which there are finitely many. Then the criterion of \cite{adamek2019finitary} applies, proving that $\mmonad$ and $\imonad$ are finitary.
\end{proof}

\begin{dfn}
Let $\monad{1}$, $\monad{2}$ be monads on a category $\cat{C}$. A \emph{$\{\monad{1},\monad{2}\}$\nbd algebra} is a pair of a $\monad{1}$\nbd algebra $\alpha_1\colon \monad{1}X \to X$ and a $\monad{2}$\nbd algebra $\alpha_2\colon \monad{2}X \to X$ on the same object of $\cat{C}$. 

A \emph{morphism} $f\colon (X, \alpha_1, \alpha_2) \to (Y, \beta_1, \beta_2)$ of $\{\monad{1},\monad{2}\}$\nbd algebras is a morphism $f\colon X \to Y$ in $\cat{C}$ that is compatible with both algebra structures. $\{\monad{1},\monad{2}\}$\nbd algebras and their morphisms form a category $\algmon{\{\monad{1},\monad{2}\}}$. 

There is a pullback square of forgetful functors
\begin{equation} \label{eq:monad_forgetful}
\begin{tikzpicture}[baseline={([yshift=-.5ex]current bounding box.center)}]
	\node (0) at (0,1.5) {$\algmon{\{\monad{1},\monad{2}\}}$};
	\node (1) at (2.5,1.5) {$\algmon{\monad{1}}$};
	\node (2) at (0,0) {$\algmon{\monad{2}}$};
	\node (3) at (2.5,0) {$\cat{C}$};
	\draw[1c] (0) to (1);
	\draw[1c] (0) to (2);
	\draw[1c] (2) to (3);
	\draw[1c] (1) to (3);
	\draw[edge] (0.9,1.3) to (0.9,0.8) to (0.2,0.8);
\end{tikzpicture}
\end{equation}
forgetting either algebra structure.
\end{dfn}

\begin{dfn}
An $\{\mmonad, \imonad\}$\nbd algebra is, equivalently, a diagrammatic set $X$ with an additional structure of non-unital $\cls{C}$\nbd $\omega$\nbd category on its underlying $\cls{C}$\nbd polygraph.

If $X$ is a $\{\mmonad, \imonad\}$\nbd algebra, we write $p\cdot x$ and $s\cdot x$ for the action of a surjective map $p$ or a subdivision $s$ on a cell or composable diagram $x$ in $X$.
\end{dfn}

\begin{dfn} \label{par:algebraic_coproduct}
The results of \cite[Section VIII]{kelly1980unified} imply that, when $\cat{C}$ has small limits, the following are equivalent:
\begin{enumerate}
	\item $\monad{1}$ and $\monad{2}$ have a coproduct $\monad{1}\oplus \monad{2}$ in the category of monads on $\cat{C}$ and monad morphisms;
	\item the forgetful functor $\algmon{\{\monad{1},\monad{2}\}} \to \cat{C}$ has a left adjoint.
\end{enumerate}
When either condition holds, $\monad{1}\oplus \monad{2}$ is the monad induced by the adjunction. In particular $\algmon{\{\monad{1},\monad{2}\}}$ is equivalent to $\algmon{(\monad{1}\oplus \monad{2})}$.

G.M.\ Kelly provided sufficient conditions for the existence of colimits of diagrams of monads. As observed by Ji\v{r}\'{i} Ad\'{a}mek \cite{adamek2014colimits}, these are verified if all monads in the diagram are accessible monads on a locally presentable category. In this case the colimit is also accessible with the same index.
\end{dfn}

\begin{lem} \label{lem:coproduct_exists}
The coproduct $\mimonad$ of $\mmonad$ and $\imonad$ exists and is finitary. 
\end{lem}
\begin{proof}
Follows from Proposition \ref{prop:finitary_lfp} and the results discussed in \S \ref{par:algebraic_coproduct}.
\end{proof}

\begin{prop} \label{prop:mimonad_lfp}
The category $\algmon{\{\mmonad,\imonad\}}$ is locally finitely presentable.
\end{prop}
\begin{proof}
Because $\cpol$ is a presheaf topos, in particular it has small limits, so by \S \ref{par:algebraic_coproduct} and Lemma \ref{lem:coproduct_exists} the forgetful functor $\algmon{\{\mmonad,\imonad\}} \to \cpol$ has a left adjoint and $\algmon{\{\mmonad,\imonad\}}$ is equivalent to $\algmon{(\mimonad)}$. 

Then $\algmon{\{\mmonad,\imonad\}}$ is the Eilenberg-Moore category of a finitary monad on a locally finitely presentable category, and the statement follows from the Remark at the end of \cite[\S 2.78]{adamek1994locally}.
\end{proof}

\begin{dfn}
As an instance of (\ref{eq:monad_forgetful}), there is a pullback square of forgetful functors
\begin{equation} \label{eq:forgetful_mimonad}
\begin{tikzpicture}[baseline={([yshift=-.5ex]current bounding box.center)}]
	\node (0) at (0,1.5) {$\algmon{\{\mmonad,\imonad\}}$};
	\node (1) at (2.5,1.5) {$\nucat$};
	\node (2) at (0,0) {$\dgmset$};
	\node (3) at (2.5,0) {$\cpol$.};
	\draw[1c] (0) to (1);
	\draw[1c] (0) to (2);
	\draw[1c] (2) to (3);
	\draw[1c] (1) to (3);
	\draw[edge] (0.9,1.3) to (0.9,0.8) to (0.2,0.8);
\end{tikzpicture}
\end{equation}
\end{dfn}

\begin{prop} \label{prop:left_adjoints}
All functors in diagram (\ref{eq:forgetful_mimonad}) have left adjoints.
\end{prop}
\begin{proof}
We already know that the functors $\dgmset \to \cpol$ and $\nucat \to \cpol$ have left adjoints. 

By Proposition \ref{prop:mimonad_lfp}, $\algmon{\{\mmonad,\imonad\}}$ has all small colimits. Then every functor $\algmon{\{\mmonad,\imonad\}} \to \algmon{\monad{}}$ induced by a monad morphism $\monad{} \to \mimonad$ has a left adjoint \cite{linton1969coequalizers}. Since the forgetful functors $\algmon{\{\mmonad,\imonad\}} \to \nucat$ and $\algmon{\{\mmonad,\imonad\}} \to \dgmset$ are induced by coproduct inclusions of monads $\mmonad \to \mimonad$, $\imonad \to \mimonad$, we conclude.
\end{proof}


\subsection{Algebraic composites}

\begin{dfn}
Let $U \in \cls{C}$ be a molecule, $V \subseteq U$ a closed subset, and $s\colon U \cfun U'$ a subdivision. Then $s(V)$ is a closed subset of $U'$, and the subdivision 
\begin{equation*}
	\idd{} \gray s\colon O^1 \gray U \cfun O^1 \gray U'
\end{equation*}
descends through the quotient to a subdivision
\begin{equation*}
	s'\colon \slice{O^1 \gray U}{\sim_V} \cfun \slice{O^1 \gray U'}{\sim_{s(V)}}.
\end{equation*}
In particular, when $V = \bord U$, we write $\infl{s}$ for $s'\colon \infl{U} \cfun \infl{U'}$. Let
\begin{equation*}
	p\colon \slice{O^1 \gray U}{\sim_V} \surj U, \quad \quad p'\colon \slice{O^1 \gray U'}{\sim_{s(V)}} \surj U'
\end{equation*}
be the retractions obtained by factorising the natural maps $O^1 \gray U \surj U$ and $O^1 \gray U' \surj U'$ through the quotients. We represent this setup with the formal diagram
\begin{equation} \label{eq:cylinder_subdivision}
\begin{tikzpicture}[baseline={([yshift=-.5ex]current bounding box.center)}]
	\node (0) at (-1,1.5) {$\slice{O^1 \gray U}{\sim_V}$};
	\node (1) at (2.5,1.5) {$\slice{O^1 \gray U'}{\sim_{s(V)}}$};
	\node (2) at (-1,0) {$U$};
	\node (3) at (2.5,0) {$U'$.};
	\draw[1cloop] (0) to node[auto,arlabel] {$s'$} (1);
	\draw[1csurj] (0) to node[auto,arlabel,swap] {$p$} (2);
	\draw[1cloop] (2) to node[auto,arlabel,swap] {$s$} (3);
	\draw[1csurj] (1) to node[auto,arlabel] {$p'$} (3);
\end{tikzpicture}
\end{equation}
\end{dfn}

\begin{dfn}[Unit-merge compatibility]
Let $X$ be an $\{\mmonad, \imonad\}$\nbd algebra. We say that $X$ satisfies \emph{unit-merge compatibility} if for all atoms $U$, subdivisions $s\colon U \cfun U'$, closed subsets $V \subseteq \bord U$, and composable diagrams $x\colon U' \to X$, the equation
\begin{equation} \label{eq:unitmerge}
	p \cdot (s \cdot x) = s' \cdot (p' \cdot x)
\end{equation}
holds for $s', p, p'$ defined as in diagram (\ref{eq:cylinder_subdivision}).
\end{dfn}

\begin{dfn}[Diagrammatic $\omega$-category]
A \emph{diagrammatic $\omega$\nbd category} is an $\{\mmonad, \imonad\}$\nbd algebra $X$ that satisfies unit-merge compatibility.

A \emph{functor} of diagrammatic $\omega$\nbd categories is a morphism of $\{\mmonad, \imonad\}$\nbd algebras. Diagrammatic $\omega$\nbd categories and functors form a category $\dgmcat$.
\end{dfn}

\begin{lem} \label{lem:unitmerge}
Let $U \in \cls{C}$, let $s\colon U \cfun U'$ be a subdivision, and let $x$ be a composable diagram of shape $V$ in a diagrammatic $\omega$\nbd category. Then
\begin{equation*}
	\eps{}(s \cdot x) = \infl{s} \cdot (\eps{}x).
\end{equation*}
\end{lem}
\begin{proof}
Let $n \eqdef \dmn{U}$, let $\{U_i \incl U\}_{i \in I}$ be the $n$\nbd dimensional atoms in $U$, and let $s_i$ be the restriction of $s$ to $U_i$. Then the restriction of $\infl{s}$ to $U_i$ is of the form
\begin{equation*}
	s'_i\colon \slice{O^1 \gray U_i}{\sim_{V_i}} \cfun \slice{O^1 \gray s(U_i)}{\sim_{s(V_i)}}
\end{equation*}
where $V_i \eqdef U_i \cap \bord U \subseteq \bord U_i$.

Letting $p_i$ be the retraction $\slice{O^1 \gray U_i}{\sim_{V_i}} \surj U_i$ and $p'_i$ be the retraction $\slice{O^1 \gray s(U_i)}{\sim_{s(V_i)}} \surj s(U_i)$, we have
\begin{equation*}
	p_i \cdot (s_i \cdot x_i) = s'_i \cdot (p'_i \cdot x_i)
\end{equation*}
using (\ref{eq:unitmerge}). Since $U$ is pure, it is equal to the union of the $U_i$, and these equations can be collated to give $\eps{}(s \cdot x) = \infl{s} \cdot (\eps{}x)$.
\end{proof}

\begin{dfn}
An equivalence in a diagrammatic $\omega$\nbd category is an equivalence in its underlying diagrammatic set.
\end{dfn}

\begin{prop} \label{prop:unitmerge_dgmcat}
Let $X$ be a diagrammatic $\omega$\nbd category and $U, V \in \cls{C}$. If $s\colon U \cfun V$ is a subdivision and $x\colon V \to X$ an equivalence, then $s \cdot x$ is an equivalence.
\end{prop}
\begin{proof}
Let $A \subseteq \cdgm{X}$ be the set of composable diagrams equal to $s \cdot e$ for an equivalence $e\colon V \to X$ and a subdivision $s\colon U \cfun V$, and consider one such diagram. Then $e\colon x \celto y$ is weakly invertible, so there exist a weak inverse $e^*\colon y \celto x$ and equivalences $h\colon e \cp{n-1} e^* \celto \eps{}{x}$ and $h'\colon e^* \cp{n-1} e \celto \eps{}{y}$. 

Let $V^*$ be the shape of $e^*$, $W$ be the shape of $h$, and $W'$ be the shape of $h'$. We have the following relations:
\begin{itemize}
	\item$\bord{}{\alpha}V^* = \bord{}{-\alpha}V$ for all $\alpha \in \{+,-\}$,
	\item $\bord{}{-}W = V\cp{n-1}V^*$ and $\bord{}{+}W = \infl{\bord{}{-}V}$,
	\item $\bord{}{-}W' = V^* \cp{n-1} V$ and $\bord{}{+}W' = \infl{\bord{}{+}V}$. 
\end{itemize}
Let $U^* \eqdef \bord{}{+}U \celto \bord{}{-}U$. There is a unique subdivision $s'\colon U^* \cfun V^*$ such that $\restr{s'}{\bord{}{\alpha}U^*} \eqdef \restr{s}{\bord{}{-\alpha}U}$ for all $\alpha \in \{+,-\}$.

Let $W_s \eqdef (U \cp{n-1} U^*) \celto \infl{\bord{}{-}U}$ and $W'_s \eqdef (U^* \cp{n-1} U) \celto \infl{\bord{}{+}U}$. There are subdivisions $t\colon W_s \cfun W$ and $t'\colon W'_s \cfun W'$ defined by
\begin{enumerate}
	\item $\restr{t}{U} = \restr{t'}{U} \eqdef s$ on $U \incl \bord{}{-}W_s$ and $U \incl \bord{}{-}W'_s$,
	\item $\restr{t}{U^*} = \restr{t'}{U^*} \eqdef s'$ on $U^* \incl \bord{}{-}W_s$ and $U^* \incl \bord{}{-}W'_s$,
	\item $\restr{t}{\bord{}{+}W_s} \eqdef \infl{\restr{s}{\bord{}{-}U}}$ and $\restr{t}{\bord{}{+}W'_s} \eqdef \infl{\restr{s}{\bord{}{+}U}}$.
\end{enumerate}
Then $t \cdot h$ has type $(s \cdot e) \cp{n-1} (s' \cdot e^*) \celto \infl{\restr{s}{\bord{}{-}U}} \cdot \eps{}x$, while $t' \cdot h'$ has type $(s' \cdot e^*) \cp{n-1} (s \cdot e) \celto \infl{\restr{s}{\bord{}{+}U}} \cdot \eps{}y$. By Lemma \ref{lem:unitmerge}, these are equal to
\begin{equation*}
(s \cdot e) \cp{n-1} (s' \cdot e^*) \celto \eps{}(\restr{s}{\bord{}{-}U} \cdot x), \quad \quad (s' \cdot e^*) \cp{n-1} (s \cdot e) \celto \eps{}(\restr{s}{\bord{}{+}U} \cdot y),
\end{equation*} 
respectively. Now $t \cdot h$ and $t' \cdot h'$ belong to $A$, which proves that $s \cdot e \in \winfun{A}$. Then $A \subseteq \winfun{A}$, and by coinduction $A \subseteq \winv{X} = \equi{X}$. 
\end{proof}

\begin{dfn}[Algebraic composite]
Let $X$ be a diagrammatic $\omega$\nbd category and let $x$ be a composable diagram of shape $U$ in $X$. Let $\compos{-}$ be the unique subdivision $\compos{U} \cfun U$. The \emph{algebraic composite} of $x$ is the cell $\compos{-} \cdot x$.
\end{dfn}

\begin{prop} \label{prop:algebraic_weak_composite}
The algebraic composite of $x$ is a weak composite of $x$.
\end{prop}
\begin{proof}
There is a unique subdivision $c\colon (U \celto \compos{U}) \cfun \infl{U}$ restricting to the identity on $U \incliso \bord{}{-}(U \celto \compos{U})$. Then $c(x) \eqdef c \cdot (\eps{}{x})$ is a cell of type $x \celto (\compos{-} \cdot x)$. Because $\eps{}{x}$ is an equivalence, $c(x)$ is also an equivalence, exhibiting $\compos{-} \cdot x$ as a weak composite of $x$.
\end{proof}

\begin{cor}
Every diagrammatic $\omega$\nbd category has weak composites.
\end{cor}

\begin{rmk}
By Proposition \ref{prop:basic_weak_composite}, morphisms of diagrammatic sets only weakly preserve weak composites, but functors of diagrammatic $\omega$\nbd categories strictly preserve algebraic composites.
\end{rmk}

\begin{dfn}
Because it is defined equationally, $\dgmcat$ is a reflective subcategory of $\algmon{\{\mmonad, \imonad\}}$. Combined with Proposition \ref{prop:left_adjoints}, this implies that the forgetful functor $\dgmcat \to \dgmset$ has a left adjoint $\tmonad\colon \dgmset \to \dgmcat$. 

This left adjoint can be constructed stepwise as follows. For each diagrammatic set $X$, we define sequences $\imonad_n X$ of diagrammatic sets and $\mmonad_n X$ of non-unital $\cls{C}$\nbd $\omega$\nbd categories, with the property that
\begin{enumerate}
	\item for all $n > 0$, there are inclusions
\begin{equation} \label{eq:tmonad_inclusions}
	\imonad_{n-1} X \incl \mmonad_{n-1} X \incl \imonad_n X \incl \mmonad_n X \incl \ldots
\end{equation}
of the underlying $\cls{C}$\nbd polygraphs, and
	\item $\imonad_{n-1} X \incl \imonad_n X$ is a morphism in $\dgmset$ and $\mmonad_{n-1} X \incl \mmonad_n X$ is a morphism in $\nucat$.
\end{enumerate}
We let $\imonad_0 X \eqdef X$ and $\mmonad_0 \eqdef \mmonad X$ with the free $\mmonad$\nbd algebra structure. Let $n > 0$ and suppose $\imonad_{n-1} X$ and $\mmonad_{n-1} X$ are defined.
\begin{itemize}
	\item We let $\imonad_n X$ be the quotient of the free $\imonad$\nbd algebra $\imonad (\mmonad_{n-1} X)$ by the relations
	\begin{enumerate}
		\item $(p, x) \sim (\idd{}, p\cdot x)$ whenever $x$ factors through $\imonad_{n-1} X$, and
		\item $(p, s \cdot x) \sim (\idd{}, s' \cdot (p'\cdot x))$ whenever $x$ factors through $\imonad_{n-1} X$, $s\cdot x$ is defined in $\mmonad_{n-1} X$, and $s', p, p'$ are as in diagram (\ref{eq:cylinder_subdivision}).
	\end{enumerate}
	\item We let $\mmonad_n X$ be the quotient of the free $\mmonad$\nbd algebra $\mmonad (\imonad_{n} X)$ by the relation $(s, x) \sim (\idd{}, s \cdot x)$ whenever $x$ factors through $\mmonad_{n-1} X$. 
\end{itemize}
Let $\tmonad X$ be the colimit in $\cpol$ of the sequence of inclusions (\ref{eq:tmonad_inclusions}). Then $\tmonad X$ is both a colimit of $\imonad$\nbd algebras and of $\mmonad$\nbd algebras, so it has an $\{\mmonad, \imonad\}$\nbd algebra structure. By construction it satisfies unit-merge compatibility, so it is a diagrammatic $\omega$\nbd category.

We have a natural inclusion $X \incl \tmonad X$ of diagrammatic sets. It is an exercise to show that it is universal from $X$ to the forgetful functor $\dgmcat \to \dgmset$. 
\end{dfn}

\begin{dfn}[Weak equivalence] \label{dfn:weak_equivalence}
Let $X, Y$ be diagrammatic sets with weak composites. A morphism $f\colon X \to Y$ is a \emph{weak equivalence} if 
\begin{enumerate}
	\item for all 0-cells $y$ in $Y$, there exists a 0-cell $x$ in $X$ such that $y \simeq f(x)$, and
	\item for all $n > 0$, parallel composable $(n-1)$\nbd diagrams $x_1, x_2$ in $X$, and $n$\nbd cells $y\colon f(x_1) \celto f(x_2)$ in $Y$, there exists an $n$\nbd cell $x\colon x_1 \celto x_2$ in $X$ such that $y \simeq f(x)$.
\end{enumerate}
\end{dfn}

\begin{comm}
This is based on the definition of weak equivalences of strict $\omega$\nbd categories in the folk model structure on $\omegacat$ defined by Lafont, M\'etayer, and Worytkiewicz \cite{lafont2010folk}. We expect it to characterise weak equivalences of fibrant objects in the model structure that we sketched in Remark \ref{rmk:model_structure}.
\end{comm}

\begin{conj} \label{conj:weak_equivalence}
If $X$ is a diagrammatic set with weak composites, the unit $X \to \tmonad X$ is a weak equivalence.
\end{conj}

\begin{comm}
This is a semi-strictification conjecture \emph{\`a la} Simpson for diagrammatic sets with weak composites. If proven true, it would allow us to consider $\tmonad X$ as a kind of coherent completion of $X$.

We suspect Conjecture \ref{conj:weak_equivalence} to be true at least when $\cls{C}$ is $\cls{S}$, the class of molecules with spherical boundary. The addition of units by $\imonad$ and composites by $\mmonad$ seem to be separately homotopically trivial: the latter is indicated by Henry's results on ``regular polygraphs and regular $\infty$\nbd categories'' \cite[Section 6.2]{henry2018regular}, to which our $\cls{S}$\nbd polygraphs and non-unital $\cls{S}$\nbd $\omega$\nbd categories are a combinatorial counterpart. 

Free $\{\mmonad,\imonad\}$\nbd algebras are constructed by the interleaving of these two operations, and unit-merge compatibility seems to be instrumental in avoiding the creation of non-joinable parallel pairs between ``first merge, then inflate'' and ``first inflate, then merge''.

An attempt to prove the conjecture would start by imitating the proof of its counterpart in the theory of regular polygraphs, namely \cite[Proposition 6.2.4]{henry2018regular}. Every cell in $\tmonad X$ factors through the inclusion of $\mmonad_n X$ or $\imonad_n X$, and we can proceed by induction.

The first non-trivial case are cells in $\mmonad_0 X$, which are of the form $(s,x)$ where $s\colon U \surj V$ is a non-trivial subdivision and $x\colon V \to X$ a composable diagram in $X$. If the boundary of $(s,x)$ is in $X$, the restriction of $s$ to $\bord U$ must be the identity, so $U = \compos{V}$ and $s = \compos{-}$. Then $(s,x)$ is the algebraic composite of $x$ in $\tmonad X$, so by Proposition \ref{prop:algebraic_weak_composite} $(s,x) \simeq x$. If $X$ has weak composites, then $x \simeq \compos{x}$ for some cell $\compos{x}$ in $X$, and $(s,x) \simeq \compos{x}$ in $\tmonad X$. 

Along similar lines, we can prove that, if both $\imonad_{n-1} X$ and $\imonad_n X$ have weak composites, then $\imonad_{n-1} X \incl \imonad_n X$ is a weak equivalence. The remaining step is to prove that $\imonad_n X$ has weak composites whenever $\imonad_{n-1} X$ does. 

Unfortunately, we cannot at present rule out that the alternation of inflate and merge creates new composable diagrams that are not reducible to composable diagrams in $X$. Topological considerations suggest that this should not be the case, but we will need a better understanding of the combinatorics of $\tmonad$ in order to prove it.
\end{comm}

\section{Nerves of strict $\omega$-categories} \label{sec:nerves}

\subsection{Algebraically free classes}

\begin{dfn}
The inclusion $\omegacat \incl \pomegacat$ has a left adjoint $-^*\colon \pomegacat \to \omegacat$. If $X$ is a partial $\omega$\nbd category, then $X^*$ is the free $\omega$\nbd category on the underlying reflexive $\omega$\nbd graph of $X$, quotiented by all the equations involving compositions that are defined in $X$. 
\end{dfn}

\begin{dfn}
Let $\cls{C}$ be a class of molecules, and let $f\colon P \to Q$ be a map of $\cls{C}$\nbd directed complexes. Because $f$ is compatible with boundaries and $\mol{}{P}$ is composition-generated by the atoms of $P$, the assignment $\clos\{x\} \mapsto \clos\{f(x)\}$ extends uniquely to a functor of partial $\omega$\nbd categories $\mol{}{P} \to \mol{}{Q}^*$, which by adjointness gives a functor of $\omega$\nbd categories $\mol{}{f}^*\colon \mol{}{P}^* \to \mol{}{Q}^*$. This determines a functor $\mol{}{-}^*\colon \dcpx \to \omegacat$. 

If $\cls{C}$ is algebraic, the assignment of Corollary \ref{cor:cfunctor_functor}, post-composed with $-^*$, also determines a functor $\mol{}{-}^*\colon \dcpxfun \to \omegacat$, which fits into a commutative square
\begin{equation} \label{eq:mol_star_functors}
\begin{tikzpicture}[baseline={([yshift=-.5ex]current bounding box.center)}]
	\node (0) at (-.5,1.5) {$\dcpxin$};
	\node (1) at (2.5,1.5) {$\dcpxfun$};
	\node (2) at (-.5,0) {$\dcpx$};
	\node (3) at (2.5,0) {$\omegacat$.};
	\draw[1cinc] (0) to (1);
	\draw[1cincl] (0) to (2);
	\draw[1c] (2) to node[auto,arlabel,swap] {$\mol{}{-}^*$} (3);
	\draw[1c] (1) to node[auto,arlabel] {$\mol{}{-}^*$} (3);
\end{tikzpicture}
\end{equation}
\end{dfn}

\begin{dfn}[Skeleta]
Let $X$ be an $\omega$\nbd category, $n \in \mathbb{N}$. The \emph{$n$\nbd skeleton} $\skel{n}{X}$ of $X$ is the restriction of $X$ to the sub-$\omega$\nbd graph
\begin{equation*}
\begin{tikzpicture}
	\node (0) at (0,0) {$X_0$};
	\node (1) at (2,0) {$\ldots$};
	\node (2) at (4,0) {$X_n$};
	\node (3) at (6,0) {$\eps{}(X_n)$};
	\node (4) at (8,0) {$\ldots$};
	\node (5) at (10,0) {$\eps{m}(X_n)$};
	\node (6) at (12,0) {$\ldots$};
	\draw[1c] (1.west |- 0,.15) to node[auto,swap,arlabel] {$\bord{}{+}$} (0.east |- 0,.15);
	\draw[1c] (1.west |- 0,-.15) to node[auto,arlabel] {$\bord{}{-}$} (0.east |- 0,-.15);
	\draw[1c] (2.west |- 0,.15) to node[auto,swap,arlabel] {$\bord{}{+}$} (1.east |- 0,.15);
	\draw[1c] (2.west |- 0,-.15) to node[auto,arlabel] {$\bord{}{-}$} (1.east |- 0,-.15);
	\draw[1c] (3.west |- 0,.15) to node[auto,swap,arlabel] {$\bord{}{+}$} (2.east |- 0,.15);
	\draw[1c] (3.west |- 0,-.15) to node[auto,arlabel] {$\bord{}{-}$} (2.east |- 0,-.15);
	\draw[1c] (4.west |- 0,.15) to node[auto,swap,arlabel] {$\bord{}{+}$} (3.east |- 0,.15);
	\draw[1c] (4.west |- 0,-.15) to node[auto,arlabel] {$\bord{}{-}$} (3.east |- 0,-.15);
	\draw[1c] (5.west |- 0,.15) to node[auto,swap,arlabel] {$\bord{}{+}$} (4.east |- 0,.15);
	\draw[1c] (5.west |- 0,-.15) to node[auto,arlabel] {$\bord{}{-}$} (4.east |- 0,-.15);
	\draw[1c] (6.west |- 0,.15) to node[auto,swap,arlabel] {$\bord{}{+}$} (5.east |- 0,.15);
	\draw[1c] (6.west |- 0,-.15) to node[auto,arlabel] {$\bord{}{-}$} (5.east |- 0,-.15);
\end{tikzpicture}
\end{equation*}
whose $m$\nbd cells are all in the image of $\eps{}\colon X_{m-1} \to X_m$ if $m > n$.
\end{dfn}

\begin{dfn}[Polygraph]
Let $\bord O^n \eqdef \skel{n-1}{O^n}$. A \emph{polygraph} is an $\omega$\nbd category $X$ together with a set $A = \bigcup_{n \in \mathbb{N}}A_n$ of \emph{generating} cells such that, for all $n \in \mathbb{N}$, 
\begin{equation*}
\begin{tikzpicture}[baseline={([yshift=-.5ex]current bounding box.center)}]
	\node (0) at (-.5,1.5) {$\coprod_{x \in A_n} \bord O^n$};
	\node (1) at (2.5,0) {$\skel{n}{X}$};
	\node (2) at (-.5,0) {$\skel{n-1}{X}$};
	\node (3) at (2.5,1.5) {$\coprod_{x \in A_n} O^n$};
	\draw[1cinc] (0) to (3);
	\draw[1c] (0) to (2);
	\draw[1cinc] (2) to (1);
	\draw[1c] (3) to node[auto,arlabel] {$(x)_{x \in A_n}$} (1);
	\draw[edge] (1.6,0.2) to (1.6,0.7) to (2.3,0.7);
\end{tikzpicture}
\end{equation*}
is a pushout in $\omegacat$. 
\end{dfn}

\begin{dfn}[Algebraically free class of molecules]
Let $\cls{C}$ be a class of molecules. We say that $\cls{C}$ is \emph{algebraically free} if $(\mol{}{P}^*, \{\clos\{x\}\}_{x \in P})$ is a polygraph for all $\cls{C}$\nbd directed complexes $P$.
\end{dfn}

\begin{exm} \label{exm:lf_algebraically_free}
By \cite[Theorem 2.13]{steiner1993algebra}, combined with [Theorem 2.17, \emph{ibid.}], the class $\cls{LF}$ of totally loop-free molecules is algebraically free.
\end{exm}

\begin{conj} \label{conj:algebraically_free}
The class $\cls{S}$ of molecules with spherical boundary is algebraically free.
\end{conj}

\begin{rmk}
If $\cls{C}$ is algebraically free, then so is every class contained in $\cls{C}$. In particular, if Conjecture \ref{conj:algebraically_free} holds, then every convenient class is algebraically free.
\end{rmk}

\begin{comm}
This conjecture generalises and implies \cite[Conjecture 3]{hadzihasanovic2018combinatorial}. Via Proposition \ref{prop:cw_poset}, we think of it as a directed version of the classical result that regular CW complexes are determined up to cellular homeomorphism by their face poset \cite[Theorem 1.7]{lundell1969topology}: it implies that, if the \emph{oriented} face poset of a polygraph is a regular directed complex, we can use it to reconstruct the original polygraph up to isomorphism.

We believe the conjecture to be true and know no counterexamples, but the normalisation strategy used in proofs of special cases, as in \cite{steiner1993algebra} or \cite{forest2019describing}, does not extend to the general case. If, on the other hand, the conjecture turned out to be false, there are two possible scenarios:
\begin{enumerate}
	\item the conjecture is false for $\cls{S}$, but true for another convenient class $\cls{C}$;
	\item no convenient class is algebraically free. 
\end{enumerate}
In the first scenario, we may want to work with the restricted class, with no apparent loss. The second scenario would point to a deeper mismatch between diagrammatic sets and $\omega$\nbd categories. To this, we could respond in two ways.

One would be to try to realign them by intervening on the underlying combinatorics of diagrammatic sets. A counterexample may suggest what other information is needed to reconstruct a unique polygraph besides its oriented face poset, and we could try to reformulate the theory starting from oriented graded posets with this additional structure. 

On the other hand, regular CW complexes and their face posets are well-established as a combinatorial notion of space, and it seems perfectly justifiable to see \emph{face posets of regular CW complexes with an orientation, compatible with composition}, as a valid combinatorial notion of directed space. From this perspective, a counterexample could be seen as a failure in the algebra of $\omega$\nbd categories to identify what should be identical directed spaces, and we would have more reason to drop strict $\omega$\nbd categories and use the combinatorial framework.
\end{comm}

\begin{prop} \label{prop:free_preserve_gamma}
The following are equivalent:
\begin{enumerate}
	\item $\cls{C}$ is algebraically free;
	\item the functor $\mol{}{-}^*\colon \dcpxin \to \omegacat$ preserves pushouts and the initial object.
\end{enumerate}
\end{prop}
\begin{proof}
Let $P$ be a $\cls{C}$\nbd directed complex and let $\skel{n}{P} \subseteq P$ be the subset of elements $x \in P$ with $\dmn{x} \leq n$. We will use the following general facts, whose proof we leave to the reader:
\begin{enumerate}
\item for all $n \in \mathbb{N}$, the diagram 
\begin{equation} \label{eq:poset_skeleton_pushout}
\begin{tikzpicture}[baseline={([yshift=-.5ex]current bounding box.center)}]
	\node (0) at (-1,1.5) {$\coprod_{\dmn{x} = n} \bord x$};
	\node (1) at (2.5,0) {$\skel{n}{P}$};
	\node (2) at (-1,0) {$\skel{n-1}{P}$};
	\node (3) at (2.5,1.5) {$\coprod_{\dmn{x} = n} \clos\{x\}$};
	\draw[1cinc] (0) to (3);
	\draw[1c] (0) to (2);
	\draw[1cinc] (2) to (1);
	\draw[1c] (3) to (1);
	\draw[edge] (1.6,0.2) to (1.6,0.7) to (2.3,0.7);
\end{tikzpicture}
\end{equation}
is a pushout in $\dcpx$ and can be built by coproducts and pushouts of inclusions;
\item $\mol{}{(\skel{n}{P})}^*$ and $\skel{n}{\mol{}{P}^*}$ are isomorphic $\omega$\nbd categories for all $n \in \mathbb{N}$;
\item for all $x \in P$ with $\dmn{x} = n$, the diagram 
\begin{equation} \label{eq:add_greatest_element}
\begin{tikzpicture}[baseline={([yshift=-.5ex]current bounding box.center)}]
	\node (0) at (-.5,1.5) {$\bord O^n$};
	\node (1) at (2.5,0) {$\mol{}{(\clos\{x\})}^*$};
	\node (2) at (-.5,0) {$\mol{}{(\bord x)}^*$};
	\node (3) at (2.5,1.5) {$O^n$};
	\draw[1cinc] (0) to (3);
	\draw[1c] (0) to (2);
	\draw[1cinc] (2) to (1);
	\draw[1c] (3) to node[auto,arlabel] {$\clos\{x\}$} (1);
	\draw[edge] (1.6,0.2) to (1.6,0.7) to (2.3,0.7);
\end{tikzpicture}
\end{equation}
is a pushout in $\omegacat$. 
\end{enumerate}

Suppose that $\mol{}{-}^*\colon \dcpxin \to \omegacat$ preserves pushouts and the initial object. Then the pushout (\ref{eq:poset_skeleton_pushout}) is preserved by $\mol{}{-}^*$. Pasting it together with the pushout diagrams (\ref{eq:add_greatest_element}), we obtain that $(\mol{}{P}^*, \{\clos\{x\}\}_{x \in P})$ is a polygraph.

Conversely, suppose that $\cls{C}$ is algebraically free. Using the pasting law for pushouts with (\ref{eq:add_greatest_element}) and the pushout diagrams
\begin{equation*}
\begin{tikzpicture}[baseline={([yshift=-.5ex]current bounding box.center)}]
	\node (0) at (-1,1.5) {$\coprod_{\dmn{x} = n} \bord O^n$};
	\node (1) at (2.5,0) {$\mol{}{(\skel{n}{P})}^*$};
	\node (2) at (-1,0) {$\mol{}{(\skel{n-1}{P})}^*$};
	\node (3) at (2.5,1.5) {$\coprod_{\dmn{x} = n} O^n$};
	\draw[1cinc] (0) to (3);
	\draw[1c] (0) to (2);
	\draw[1cinc] (2) to (1);
	\draw[1c] (3) to node[auto,arlabel] {$(\clos\{x\})_{\dmn{x} = n}$} (1);
	\draw[edge] (1.6,0.2) to (1.6,0.7) to (2.3,0.7);
\end{tikzpicture}
\end{equation*}
for all $x \in P$ and $n \in \mathbb{N}$, we find that $\mol{}{-}^*$ preserves the pushout diagrams (\ref{eq:poset_skeleton_pushout}). An inductive argument then shows that $\mol{}{-}^*$ preserves the canonical diagram exhibiting $P$ as the colimit of inclusions of its atoms. This suffices to conclude.
\end{proof}


\subsection{Diagrammatic nerve}

\begin{center}
\setlength{\fboxsep}{.6em}
\colorbox{gray!20}{In this section we assume that $\cls{C}$ is convenient and algebraically free.}
\end{center}

\begin{dfn}
By \cite[Theorem 2.13]{steiner1993algebra}, for each directed complex $P$ the adjunction unit $\mol{}{P} \to \mol{}{P}^*$ is injective on cells: since $P$ can be reconstructed from $\mol{}{P}$, it follows that all functors in (\ref{eq:mol_star_functors}) are faithful and injective on objects. We may then identify $\dcpxfun$ and $\dcpx$ with subcategories of $\omegacat$, and write $f\colon P \to Q$ instead of $\mol{}{f}^*\colon \mol{}{P}^* \to \mol{}{Q}^*$ for a map or $\cls{C}$\nbd functor seen as a functor of $\omega$\nbd categories.
\end{dfn}

\begin{dfn}
Let $\dcpxomega$ be the full subcategory of $\omegacat$ on $\cls{C}$\nbd directed complexes. By Proposition \ref{prop:free_preserve_gamma}, since $\cls{C}$ is algebraically free we have a square 
\begin{equation*}
\begin{tikzpicture}[baseline={([yshift=-.5ex]current bounding box.center)}]
	\node (0) at (0,1.5) {$\dcpxin$};
	\node (1) at (2.5,1.5) {$\dcpxfun$};
	\node (2) at (0,0) {$\dcpx$};
	\node (3) at (2.5,0) {$\dcpxomega$};
	\draw[1cinc] (0) to (1);
	\draw[1cincl] (0) to (2);
	\draw[1cinc] (2) to (3);
	\draw[1cincl] (1) to (3);
\end{tikzpicture}
\end{equation*}
of inclusions of subcategories which preserve the set $\Gamma$ of colimit diagrams comprising the initial object diagram and all pushout diagrams of inclusions. Applying $\psh{\Gamma}{-}$, we obtain a square
\begin{equation*}
\begin{tikzpicture}[baseline={([yshift=-.5ex]current bounding box.center)}]
	\node (0) at (-.5,1.5) {$\psh{\Gamma}{\dcpxomega}$};
	\node (1) at (2.5,1.5) {$\nucat$};
	\node (2) at (-.5,0) {$\dgmset$};
	\node (3) at (2.5,0) {$\cpol$};
	\draw[1c] (0) to (1);
	\draw[1c] (0) to (2);
	\draw[1c] (2) to (3);
	\draw[1c] (1) to (3);
\end{tikzpicture}
\end{equation*}
of restriction functors, all with left adjoints. By the universal property of the pullback square (\ref{eq:forgetful_mimonad}), this produces a functor $\psh{\Gamma}{\dcpxomega} \to \algmon{\{\mmonad,\imonad\}}$, which actually factors through a functor
\begin{equation} \label{eq:omega_restrict}
	\psh{\Gamma}{\dcpxomega} \to \dgmcat.
\end{equation}
This follows immediately from the fact that the formal diagram (\ref{eq:cylinder_subdivision}) is an actual commutative diagram in $\dcpxomega$.
\end{dfn}

\begin{dfn}
The category $\dcpxomega$ is small: it has countably many objects, and countably many morphisms between any pair of them. Then $\psh{\Gamma}{\dcpxomega}$ is locally small. Finally, $\omegacat$ has small colimits. 

By \cite[Theorem 1.2.1]{riehl2014categorical}, the left Kan extension of $\dcpxomega \incl \omegacat$ along the Yoneda embedding $\dcpxomega \incl \psh{\Gamma}{\dcpxomega}$ exists and has a right adjoint $N\colon \omegacat \to \psh{\Gamma}{\dcpxomega}$.
\end{dfn}

\begin{rmk}
$\dcpxomega$ contains a full subcategory isomorphic to Joyal's category $\Theta$ \cite{joyal1997disks}, namely, the one obtained by restricting first to $\cls{O}$\nbd directed complexes, where $\cls{O}$ is the class of all globes, and then to the molecules therein. The isomorphism is best seen through Makkai and Zawadowski's description of $\Theta$ as the full subcategory of $\omegacat$ on the \emph{simple} $\omega$\nbd categories \cite{makkai2001duality}.

By \cite[Theorem 1.12]{berger2002cellular}, $\Theta$ is a dense subcategory of $\omegacat$. Then $\dcpxomega$ is also a dense subcategory, which implies that $N$ is full and faithful. 
\end{rmk}

\begin{dfn}[Diagrammatic nerve]
The \emph{diagrammatic nerve} is the functor 
\begin{equation*}
	\nerve{}\colon \omegacat \to \dgmcat
\end{equation*}
obtained as the composite of $N\colon \omegacat \to \psh{\Gamma}{\dcpxomega}$ and (\ref{eq:omega_restrict}).
\end{dfn}

\begin{dfn}
Concretely, diagrams of shape $U$ in $\nerve{X}$ are functors $U \to X$, that is, $\mol{}{U}^* \to X$ in $\omegacat$. The action of a surjective map $p$ or a subdivision $s$ on a cell or composable diagram is precomposition with $p$ or $s$ in $\omegacat$.
\end{dfn}

\begin{lem} \label{lem:omegagph_faithful}
Let $X, Y$ be $\omega$\nbd categories and let $f, g\colon \nerve{X} \to \nerve{Y}$ be functors. Suppose that $f(x) = g(x)$ for all $n \in \mathbb{N}$ and cells $x$ of shape $O^n$. Then $f = g$.
\end{lem}
\begin{proof}
Let $x$ be a cell of shape $U$ in $\nerve{X}$. We will prove that $f(x) = g(x)$ by induction on $n \eqdef \dmn{U}$. If $n = 0$ or $n = 1$, then $U = O^n$ and $f(x) = g(x)$ holds by assumption.
 
Suppose $n > 1$. The cells $f(x)$ and $g(x)$ correspond uniquely to functors $f(x), g(x)\colon U \to Y$, and it suffices to show that they are equal on the atoms of $U$. For atoms of dimension $k < n$ this holds by the inductive hypothesis, so we only need to show that $f(x)(U) = g(x)(U)$. This is equivalent to $s;f(x) = s;g(x)$ where $s$ is the unique subdivision $O^n \cfun U$. But $s;f(x) = f(s\cdot x)$ and $s;g(x) = g(s\cdot x)$ because $f$ and $g$ are functors of diagrammatic $\omega$\nbd categories. Since $s \cdot x$ is a cell of shape $O^n$, we conclude.
\end{proof}

\begin{dfn} \label{dfn:omega_compositors}
For all $k \in \mathbb{N}$ and $n > k$, we define molecules $C_{n,k} \in \cls{C}$ with inclusions $\imath_{n,k}\colon O^n \cp{k} O^n \incl C_{n,k}$ and retractions $p_{n,k}\colon C_{n,k} \surj O^n \cp{k} O^n$. 

For each $k$, we let $C_{k+1,k} \eqdef O^{k+1} \cp{k} O^{k+1}$, and let both $\imath_{k+1,k}$ and $p_{k+1,k}$ be identities. These belong to $\cls{C}$ by axiom \ref{ax:submolecule} of convenient classes. 

If $n > k+1$, suppose that $C_{n-1,k}$, $\imath_{n-1, k}$, $p_{n-1,k}$ are defined and $C_{n-1,k} \in \cls{C}$. We let $C_{n,k}$ be the pushout
\begin{equation*}
\begin{tikzpicture}[baseline={([yshift=-.5ex]current bounding box.center)}]
	\node (0) at (-.5,1.5) {$O^{n-1} \cp{k} O^{n-1}$};
	\node (1) at (2.5,1.5) {$O^n \cp{k} O^n$};
	\node (2) at (-.5,0) {$\infl{C_{n-1,k}}$};
	\node (3) at (2.5,0) {$C_{n,k}$};
	\draw[1cinc] (0) to node[auto,arlabel] {$\imath^+$} (1);
	\draw[1cincl] (0) to node[auto,swap,arlabel] {$\imath_{n-1,k};\imath^-$} (2);
	\draw[1cinc] (2) to (3);
	\draw[1cincl] (1) to node[auto,arlabel] {$\imath_{n,k}$} (3);
	\draw[edge] (1.6,0.2) to (1.6,0.7) to (2.3,0.7);
\end{tikzpicture}
\end{equation*}
in $\dcpx$, where the $\imath^\alpha$ are inclusions of $(n-1)$\nbd boundaries. This also defines $\imath_{n,k}$. Then we let $p_{n,k}$ be the unique map such that
\begin{enumerate}
	\item $p_{n,k}$ is the identity on $O^n \cp{k} O^n$, and
	\item $p_{n,k}$ is the composite of $\tau\colon \infl{C_{n-1,k}} \surj C_{n-1,k}$ and $p_{n-1,k}$ on $\infl{C_{n-1,k}}$.
\end{enumerate}
The construction of $C_{n,k}$ can be factored as the pasting of $\infl{C_{n-1,k}}$ and one copy of $O^n$ along a submolecule of $\bord{}{-}\infl{C_{n-1,k}}$, followed by the pasting of another copy of $O^n$. By axioms \ref{ax:submolecule} and \ref{ax:inflate} of convenient classes, $C_{n,k} \in \cls{C}$.
\end{dfn}

\begin{thm}
The diagrammatic nerve is full and faithful.
\end{thm}
\begin{proof}
Let $X, Y$ be $\omega$\nbd categories and $f\colon \nerve{X} \to \nerve{Y}$ a functor. Passing to the underlying diagrammatic sets, then restricting to presheaves on the full subcategory of $\atom$ whose objects are the globes, we obtain a morphism $f'$ of the underlying reflexive $\omega$\nbd graphs of $X$ and $Y$, which by Lemma \ref{lem:omegagph_faithful} uniquely determines $f$.

To conclude, it suffices to show that $f'$ is compatible with compositions in $X$ and $Y$. Let $x_1, x_2$ be $n$\nbd cells in $X$ such that $x_1 \cp{k} x_2$ is defined. This composition is classified by a functor $x\colon O^n \cp{k} O^n \to X$ with the property that $x(O^n \cp{k} O^n) = x_1 \cp{k} x_2$. 

The functor classifying $O^n \cp{k} O^n$ as a cell of $\mol{}{(O^n \cp{k} O^n)}^*$ factors as the unique subdivision $s\colon O^n \cfun C_{n,k}$ followed by the map $p_{n,k}\colon C_{n,k} \surj O^n \cp{k} O^n$ as defined in \S \ref{dfn:omega_compositors}. Then
\begin{align*}
	f'(x_1 \cp{k} x_2) & = f'(x(O^n \cp{k} O^n)) = f(s \cdot (p_{n,k} \cdot x)) = \\
	& = s \cdot (p_{n,k} \cdot f(x)) = f(x)(O^n \cp{k} O^n) = f'(x_1) \cp{k} f'(x_2).
\end{align*}
This proves that $f'\colon X \to Y$ is a functor of $\omega$\nbd categories and $f = \nerve{f'}$. 
\end{proof}

\begin{rmk} \label{rmk:dgmset_nerve}
The composite of $\nerve{}$ with the forgetful functor to $\dgmset$ is faithful, but it is not full. For example, consider a 2-category $X$ generated by three 0-cells $x, y, z$, three 1-cells $f\colon x \celto y$, $g\colon y \celto z$, and $h: x \celto z$, and an isomorphism $\alpha: f \cp{0} g \celto h$. There is an involution on the diagrammatic set $\nerve{X}$ which swaps $f \cp{0} g$ and $h$, but it is not the image of any endofunctor of $X$ because it does not respect composition.
\end{rmk}

\section{Homotopy hypothesis} \label{sec:homotopy}

\subsection{Simplices as molecules}

We assume some basic knowledge of homotopy theory and simplicial sets. We mostly base our notation on Goerss and Jardine \cite{goerss2009simplicial}.

\begin{dfn}[Simplex category]
The category $\atom$ contains a subcategory $\deltacat$ isomorphic to the \emph{simplex category}. The $n$\nbd simplex in $\deltacat$ is $\Delta^n$ as defined in \S \ref{dfn:cubes_simplices}. For $0 \leq k \leq n$, if $!$ denotes the unique map to the terminal object $1$,
\begin{itemize}
	\item the $k$\nbd th co-face map $d^k\colon \Delta^{n-1} \incl \Delta^n$ is the inclusion 
	\begin{equation*}
		\idd{} \join ! \join \idd{}\colon \Delta^{k-1}\join \emptyset \join \Delta^{n-k-1} \incl \Delta^{k-1}\join 1 \join \Delta^{n-k-1},
	\end{equation*}
	\item the $k$\nbd th co-degeneracy map $s^k\colon \Delta^{n+1} \surj \Delta^n$ is the surjective map
	\begin{equation*}
		\idd{} \join ! \join \idd{}\colon \Delta^{k-1} \join \Delta^1 \join \Delta^{n-k-1} \surj \Delta^{k-1}\join 1 \join \Delta^{n-k-1}.
	\end{equation*}
\end{itemize}
\end{dfn}

\begin{dfn}
We write $\sset$ for the category of simplicial sets, that is, presheaves on $\deltacat$. Corresponding to the inclusion $\deltacat \incl \atom$, there is a restriction functor $\delres{-}\colon \dgmset \to \sset$ with a left adjoint $\imath_\deltacat\colon \sset \to \dgmset$. 

The functor $\imath_\deltacat$ is the left Kan extension of $\deltacat \incl \atom \incl \dgmset$ along the Yoneda embedding $\deltacat \incl \sset$. By the coend formula for left Kan extensions, cells of shape $U$ in $\imath_\deltacat X$ are represented by pairs $(f\colon U \to \Delta^n, x\colon \Delta^n \to X)$ where $f$ is a map in $\atom$, quotiented by the relation
\begin{equation*}
	(f\colon U \to \Delta^n, g;x\colon \Delta^n \to X) \sim (f;g\colon U \to \Delta^m, x\colon \Delta^m \to X)
\end{equation*}
for all maps $g\colon \Delta^n \to \Delta^m$ in $\deltacat$.
\end{dfn}

\begin{prop} \label{prop:delta_is_full}
$\deltacat$ is a full subcategory of $\atom$.
\end{prop}
\begin{proof}
Let $[n]$ be the linear order $\{0 < \ldots < n\}$. Morphisms $f\colon \Delta^n \to \Delta^m$ in $\deltacat$ corresponds bijectively to order-preserving functions $[n] \to [m]$. In particular, $d^k\colon [n-1] \to [n]$ identifies $[n-1]$ with the suborder 
\begin{equation*}
	\{0 < \ldots < k-1 < k+1 < \ldots < n\}.
\end{equation*}
Our strategy will be to prove that there are no maps $f\colon \Delta^n \to \Delta^m$ in $\atom$ besides these.

First of all, each atom of a simplex is a simplex, so each map $f\colon \Delta^n \to \Delta^m$ factors uniquely as a surjective map $p\colon \Delta^n \surj \Delta^k$ followed by an inclusion $\imath\colon \Delta^k \incl \Delta^m$ with $k \leq n,m$. Thus we can prove separately that surjective maps and inclusions are in $\Delta$.

Since $\Delta^n$ has exactly $(n+1)$\nbd elements of dimension $n -1$, the co-face maps $d^k$ exhaust the inclusions $\Delta^{n-1} \incl \Delta^n$, and the identity is the only automorphism of $\Delta^n$. This suffices to prove that all inclusions of simplices are in $\deltacat$.

Let $p\colon \Delta^n \surj \Delta^k$ be a surjective map. If $n = 0$ or $n = 1$ we can handle the few possibilities separately. 

If $n > 1$, we can assume the inductive hypothesis that, for all $j < n$, all maps $f\colon \Delta^{j} \to \Delta^m$ are in $\deltacat$. Then all the faces $d^i;p\colon \Delta^{n-1} \to \Delta^k$ are in $\deltacat$, so they correspond to a system of order-preserving functions $f_i\colon [n-1] \to [k]$ with the property that $d^i; f_j = d^{j-1}; f_i$ when $i < j$. This system can be collated to a unique function $f\colon [n] \to [k]$ such that $d^i;f = f_i$. 

Because $n > 1$, given any pair of elements $k_1 \leq k_2$ in $[n]$ we can pick $i \in [n] \setminus \{k_1,k_2\}$. Then $k_1 = d^i(k'_1)$ and $k_2 = d^i(k'_2)$ for some $k'_1 \leq k'_2$ in $[n-1]$, and $f(k_1) = f_i(k'_1) \leq f_i(k'_2) = f(k_2)$. This proves that $f$ is order-preserving.

Then $f$ determines a map $\Delta^n \surj \Delta^k$ in $\deltacat$ with the property that $d^i;f = d^i;p$ for $0 \leq i \leq n$, so $\restr{f}{\bord \Delta^n} = \restr{p}{\bord \Delta^n}$. Finally, both $f$ and $p$ must send the greatest element of $\Delta^n$ to the greatest element of $\Delta^k$, so $f = p$. 
\end{proof}

\begin{cor}
The functor $\imath_\deltacat\colon \sset \to \dgmset$ is full and faithful.
\end{cor}
\begin{proof}
Let $K$ be a simplicial set. Since $\deltacat$ is a full subcategory of $\atom$, each $n$\nbd simplex in $\delres{(\imath_\deltacat K)}$ is uniquely represented by a pair $(\idd{\Delta_n}, x\colon \Delta^n \to K)$. It follows that each unit component $K \to \delres{(\imath_\deltacat K)}$ is an isomorphism, and \cite[Theorem IV.3.1]{maclane1971cats} applies.
\end{proof}

\begin{dfn}
Let $p\colon U \surj V$ be a surjective map of atoms of the same dimension. Then $p(\bord U) = \bord V$ and $\invrs{p}(\bord V) = \bord U$, so the map $\idd{} \gray p\colon O^1 \gray U \surj O^1 \gray V$ descends through the quotient to a map 
\begin{equation*}
	\infl{p}\colon \infl{U} \surj \infl{V}.
\end{equation*}
This is also a surjective map of atoms of the same dimension. The construction can then be iterated to produce a sequence $\{O^n(p)\colon O^n(U) \surj O^n(V)\}_{n \in \mathbb{N}}$.
\end{dfn}

\begin{dfn} \label{dfn:fattening}
Let $p\colon U \surj V$ be a surjective map of atoms with $\dmn{U} = \dmn{V} + 1$. There is a unique surjective map $p_\prec\colon U \surj \infl{V}$ such that the triangle
\begin{equation*}
\begin{tikzpicture}[baseline={([yshift=-.5ex]current bounding box.center)}]
	\node (0) at (-1.5,1.25) {$U$};
	\node (1) at (0,0) {$\infl{V}$};
	\node (2) at (1.5,1.25) {$V$};
	\draw[1csurj] (0) to node[auto,arlabel] {$p$} (2);
	\draw[1csurj] (0) to node[auto,swap,arlabel] {$p_\prec$} (1);
	\draw[1csurj] (1) to node[auto,swap,arlabel] {$\tau$} (2);
\end{tikzpicture}
\end{equation*}
commutes. This map sends the greatest element of $U$ to the greatest element of $\infl{V}$, and sends each $x \in \bord{}{\alpha}U$ to $\imath^\alpha(p(x)) \in \bord{}{\alpha}\infl{V}$, where $\imath^\alpha$ is the inclusion of $V$ as $\bord{}{\alpha}\infl{V}$ for each $\alpha \in \{+,-\}$.
\end{dfn}

\begin{dfn} \label{dfn:simplex_encoding}
Every element of $\Delta^n$ can be uniquely encoded as a string $b_0\ldots b_n$ of $(n+1)$ bits at least one of which is 1. The encoding identifies the string $1^{j_1}0^{k_1}\ldots 1^{j_m}0^{k_m}$ with the element
\begin{equation*}
	\underbrace{\top \gray \ldots \gray \top}_{j_1} \gray \underbrace{\bot \gray \ldots \gray \bot}_{k_1} \gray \ldots \gray \underbrace{\top \gray \ldots \gray \top}_{j_m} \gray \underbrace{\bot \gray \ldots \gray \bot}_{k_m}
\end{equation*}
of the $(n+1)$\nbd fold Gray product of $1_\bot$, used in the definition of $\Delta^n$ as an iterated join, where $\top$ denotes the unique element of $1$. 

We have $\dmn{b_0\ldots b_n} = \sum_{i=0}^n b_i - 1$ and
\begin{align*}
	\sbord{}{-}(b_0 \ldots b_n) & = \{b_0 \ldots b_{k-1} 0 b_{k+1} \ldots b_n \mid \text{$b_k = 1$ and $\sum_{i=0}^{k-1} b_i$ is odd} \}, \\
	\sbord{}{+}(b_0 \ldots b_n) & = \{b_0 \ldots b_{k-1} 0 b_{k+1} \ldots b_n \mid \text{$b_k = 1$ and $\sum_{i=0}^{k-1} b_i$ is even} \}.
\end{align*}
\end{dfn}

\begin{dfn}[Folding of simplices onto globes] \label{dfn:folding}
For all $n \in \mathbb{N}$, we define surjective maps $a_n\colon \Delta^n \surj O^n$. 

We have $O^0 = \Delta^0$ and we let $a_0 \eqdef \idd{}$. If $a_{n-1}$ is defined, we let $a_n$ be the composite of $s^0_\prec\colon \Delta^n \surj \infl{\Delta^{n-1}}$ and of $\infl{a_{n-1}}\colon \infl{\Delta^{n-1}} \to O^n$.

In terms of the encoding of \S \ref{dfn:simplex_encoding}, the map $a_n$ is defined by
\begin{alignat*}{3}
	& 1^{n+1} && \mapsto \; \undl{n}, && \\
	& 0^k 1^{n-k+1} && \mapsto \; \undl{n-k}^+ && \; \text{if $k > 0$}, \\
	& \ldots 10^k 1^j && \mapsto \; \undl{j}^- && \; \text{if $k > 0$}.
\end{alignat*}
Using this explicit definition, the reader can check that the diagrams 
\begin{equation} \label{eq:folding_faces}
\begin{tikzpicture}[baseline={([yshift=-.5ex]current bounding box.center)}]
	\node (0) at (0,1.5) {$\Delta^n$};
	\node (1) at (2.5,0) {$O^{n+1}$,};
	\node (2) at (0,0) {$\Delta^{n+1}$};
	\node (3) at (2.5,1.5) {$O^n$};
	\draw[1csurj] (0) to node[auto,arlabel] {$a_n$} (3);
	\draw[1cincl] (0) to node[auto,swap,arlabel] {$d^0$} (2);
	\draw[1csurj] (2) to node[auto,swap,arlabel] {$a_{n+1}$} (1);
	\draw[1cincl] (3) to node[auto,arlabel] {$\imath^+$} (1);
\end{tikzpicture} \quad \quad 
\begin{tikzpicture}[baseline={([yshift=-.5ex]current bounding box.center)}]
	\node (0) at (0,1.5) {$\Delta^n$};
	\node (1) at (2.5,0) {$O^{n+1}$};
	\node (2) at (0,0) {$\Delta^{n+1}$};
	\node (3) at (2.5,1.5) {$O^n$};
	\draw[1csurj] (0) to node[auto,arlabel] {$a_n$} (3);
	\draw[1cincl] (0) to node[auto,swap,arlabel] {$d^1$} (2);
	\draw[1csurj] (2) to node[auto,swap,arlabel] {$a_{n+1}$} (1);
	\draw[1cincl] (3) to node[auto,arlabel] {$\imath^-$} (1);
\end{tikzpicture}
\end{equation} 
commute for all $n \in \mathbb{N}$.
\end{dfn}

\begin{dfn} \label{dfn:globe_compositors}
For all $n > 0$, let $\Phi^{n+1}$ be the atom $O^n \celto (O^n \cp{n-1} O^n)$. 

For all $\alpha \in \{+,-\}$, $\bord{n-1}{\alpha}\Phi^{n+1}$ is isomorphic to $\bord{n-1}{\alpha}O^{n+1}$ and we name the elements of $\bord{n-1}{}\Phi^{n+1}$ accordingly. We write $\undl{n}^-$ for the single element of $\sbord{}{-}\Phi^{n+1}$ and $\undl{n}_1^+, \undl{n}_2^+$ for the elements of $\sbord{}{+}\Phi^{n+1}$ corresponding to the first and second copy of $O^n$ in $O^n \cp{n-1} O^n$. Finally, we write $\undl{n-1}^0$ for the unique element of $\sbord{}{+}\undl{n}_1^+ \cap \sbord{}{-}\undl{n}_2^+$ and $\undl{n+1}$ for the greatest element of $\Phi^{n+1}$.

For all $i \in \{1,2\}$, we write $\imath_i^+\colon O^n \incl \Phi^{n+1}$ for the inclusion defined by $\undl{n} \mapsto \undl{n}^+_i$.
\end{dfn}

\begin{rmk}
By Lemma \ref{lem:all_shapes} and axioms \ref{ax:shift} and \ref{ax:submolecule}, every convenient class contains all the $\Phi^{n}$. 
\end{rmk}

\begin{dfn}[Folding of compositors] \label{dfn:folding_compositor}
For all $n > 0$, we define surjective maps $c_{n+1}\colon \Delta^{n+1} \surj \Phi^{n+1}$. In terms of the encoding of \S \ref{dfn:simplex_encoding}, $c_{n+1}$ is defined by
\begin{alignat*}{3}
	& 1101^{n-1} && \mapsto \; \undl{n}^+_1, && \\
	& 0111^{n-1} && \mapsto \; \undl{n}^+_2, && \\
	& 0101^{n-1} && \mapsto \; \undl{n-1}^0, &&
\end{alignat*}
and the same as $a_{n+1}$ on all other elements. The reader can check that the diagrams
\begin{equation} \label{eq:cn_faces}
\begin{tikzpicture}[baseline={([yshift=-.5ex]current bounding box.center)}]
	\node (0) at (0,1.5) {$\Delta^n$};
	\node (1) at (2.5,0) {$\Phi^{n+1}$,};
	\node (2) at (0,0) {$\Delta^{n+1}$};
	\node (3) at (2.5,1.5) {$O^n$};
	\draw[1csurj] (0) to node[auto,arlabel] {$a_n$} (3);
	\draw[1cincl] (0) to node[auto,swap,arlabel] {$d^0$} (2);
	\draw[1csurj] (2) to node[auto,swap,arlabel] {$c_{n+1}$} (1);
	\draw[1cincl] (3) to node[auto,arlabel] {$\imath_2^+$} (1);
\end{tikzpicture} \quad
\begin{tikzpicture}[baseline={([yshift=-.5ex]current bounding box.center)}]
	\node (0) at (0,1.5) {$\Delta^n$};
	\node (1) at (2.5,0) {$\Phi^{n+1}$,};
	\node (2) at (0,0) {$\Delta^{n+1}$};
	\node (3) at (2.5,1.5) {$O^n$};
	\draw[1csurj] (0) to node[auto,arlabel] {$a_n$} (3);
	\draw[1cincl] (0) to node[auto,swap,arlabel] {$d^1$} (2);
	\draw[1csurj] (2) to node[auto,swap,arlabel] {$c_{n+1}$} (1);
	\draw[1cincl] (3) to node[auto,arlabel] {$\imath^-$} (1);
\end{tikzpicture} \quad
\begin{tikzpicture}[baseline={([yshift=-.5ex]current bounding box.center)}]
	\node (0) at (0,1.5) {$\Delta^n$};
	\node (1) at (2.5,0) {$\Phi^{n+1}$};
	\node (2) at (0,0) {$\Delta^{n+1}$};
	\node (3) at (2.5,1.5) {$O^n$};
	\draw[1csurj] (0) to node[auto,arlabel] {$a_n$} (3);
	\draw[1cincl] (0) to node[auto,swap,arlabel] {$d^2$} (2);
	\draw[1csurj] (2) to node[auto,swap,arlabel] {$c_{n+1}$} (1);
	\draw[1cincl] (3) to node[auto,arlabel] {$\imath_1^+$} (1);
\end{tikzpicture}
\end{equation}  
commute for all $n > 0$.
\end{dfn}

\begin{dfn} \label{dfn:extr}
For all $n > 1$ and $k \in \mathbb{N}$, we define molecules $\extr{k}{n}$ with spherical boundary, with the property that $\bord{}{\alpha}\extr{k}{n}$ is isomorphic to $\bord{}{\alpha}O^k(\Delta^n)$ for all $\alpha \in \{+,-\}$, together with inclusions of submolecules 
\begin{equation*}
	j_{k,n}\colon O^{k+1}(\Delta^{n-1}) \incl \extr{k}{n}.
\end{equation*}

For all $n > 1$, we let $\extr{0}{n}$ be the pasting of $\infl{\Delta^{n-1}}$ and $\Delta^n$ along $d^0(\Delta^{n-1}) \submol \bord{}{+}\Delta^n$, that is, the pushout
\begin{equation*}
\begin{tikzpicture}[baseline={([yshift=-.5ex]current bounding box.center)}]
	\node (0) at (0,1.5) {$\Delta^{n-1}$};
	\node (1) at (2.5,0) {$\extr{0}{n}$};
	\node (2) at (0,0) {$\Delta^n$};
	\node (3) at (2.5,1.5) {$\infl{\Delta^{n-1}}$};
	\draw[1cinc] (0) to node[auto,arlabel] {$\imath^-$} (3);
	\draw[1cincl] (0) to node[auto,swap,arlabel] {$d^0$} (2);
	\draw[1cinc] (2) to (1);
	\draw[1cincl] (3) to node[auto,arlabel] {$j_{0,n}$} (1);
	\draw[edge] (1.6,0.2) to (1.6,0.7) to (2.3,0.7);
\end{tikzpicture}
\end{equation*}
in $\rdcpxin$. Its boundaries are isomorphic to those of $O^0(\Delta^n) = \Delta^n$.

If $k > 0$, suppose that $\extr{k-1}{n}$ and $j_{k-1,n}$ have been defined. We define $\extr{k}{n}$ to be the colimit of the diagram
\begin{equation*} 
\begin{tikzpicture}[baseline={([yshift=-.5ex]current bounding box.center)}]
	\node (0) at (-2,-1.5) {$O^{k-1}(\Delta^n) \celto \extr{k-1}{n}$};
	\node (1) at (0,0) {$O^{k}(\Delta^{n-1})$};
	\node (2) at (2,-1.5) {$O^{k+1}(\Delta^{n-1})$};
	\node (3) at (4,0) {$O^{k}(\Delta^{n-1})$};
	\node (4) at (6,-1.5) {$\extr{k-1}{n} \celto O^{k-1}(\Delta^n)$};
	\draw[1cincl] (1) to node[auto,swap,arlabel] {$j_{k-1,n};\imath^+$} (0);
	\draw[1cinc] (1) to node[auto,arlabel] {$\imath^-$} (2);
	\draw[1cincl] (3) to node[auto,swap,arlabel] {$\imath^+$} (2);
	\draw[1cinc] (3) to node[auto,arlabel] {$j_{k-1,n};\imath^-$} (4);
\end{tikzpicture}
\end{equation*}
in $\rdcpxin$. This can be constructed as a sequence of two pastings along submolecules, and $\bord{}{\alpha}\extr{k}{n}$ is isomorphic to $O^{k-1}(\Delta^n) = \bord{}{\alpha}O^k(\Delta^n)$. We let $j_{k,n}$ be the inclusion of $O^{k+1}(\Delta^{n-1})$ into the colimit. 
\end{dfn}

\begin{rmk} \label{rmk:extr_convenient}
The $\extr{k}{n}$ are constructed from simplices using only the $\infl{-}$ construction, the $- \celto -$ construction, and pasting along submolecules. By Lemma \ref{lem:all_shapes} and axioms \ref{ax:shift}, \ref{ax:submolecule}, and \ref{ax:inflate}, they are contained in every convenient class.
\end{rmk}

\begin{dfn}
For all $n > 1$ and $k \in \mathbb{N}$, we define a retraction
\begin{equation*}
	 r_{k,n}\colon \extr{k}{n} \surj O^{k+1}(\Delta^{n-1})
\end{equation*}
such that $j_{k,n};r_{k,n} = \idd{}$ and the diagram
\begin{equation} \label{eq:retract_rkn}
\begin{tikzpicture}[baseline={([yshift=-.5ex]current bounding box.center)}]
	\node (0) at (-.5,1.5) {$\bord \extr{k}{n}$};
	\node (1) at (2.5,0) {$O^{k+1}(\Delta^{n-1})$};
	\node (2) at (-.5,0) {$\extr{k}{n}$};
	\node (3) at (2.5,1.5) {$O^{k}(\Delta^n)$};
	\draw[1cinc] (0) to (3);
	\draw[1cincl] (0) to (2);
	\draw[1csurj] (2) to node[auto,swap,arlabel] {$r_{k,n}$} (1);
	\draw[1csurj] (3) to node[auto,arlabel] {$O^k(s^0_\prec)$} (1);
\end{tikzpicture}
\end{equation}
commutes in $\rdcpx$. 

First, $\extr{0}{n}$ decomposes as $\Delta^n \cup \infl{\Delta^{n-1}}$. Then $r_{0,n}$ is defined by
\begin{enumerate}
	\item $\restr{r_{0,n}}{\Delta^n} \eqdef s^0; \imath^-$ and
	\item $\restr{r_{0,n}}{\infl{\Delta^{n-1}}} \eqdef \idd{}$.
\end{enumerate}
If $k > 0$, suppose $r_{k-1,n}$ is defined. Now $\extr{k}{n}$ decomposes as
\begin{equation*}
	(O^{k-1}(\Delta^n) \celto \extr{k-1}{n}) \cup O^{k+1}(\Delta^{n-1}) \cup (\extr{k-1}{n} \celto O^{k-1}(\Delta^n)).
\end{equation*}
Then $r_{k,n}$ is defined by
\begin{enumerate}
	\item $\restr{r_{k,n}}{O^{k-1}(\Delta^n) \celto \extr{k-1}{n}}$ sends the greatest element to the greatest element of $\bord{}{-}O^{k+1}(\Delta^{n-1})$, is equal to $O^{k-1}(s^0_\prec);\imath^-$ on the input boundary, and is equal to $r_{k-1,n};\imath^-$ on the output boundary,
	\item $\restr{r_{k,n}}{O^{k+1}(\Delta^{n-1})} \eqdef \idd{}$, and
	\item $\restr{r_{k,n}}{\extr{k-1}{n} \celto O^{k-1}(\Delta^n)}$ sends the greatest element to the greatest element of $\bord{}{+}O^{k+1}(\Delta^{n-1})$, is equal to $r_{k-1,n};\imath^+$ on the input boundary, and is equal to $O^{k-1}(s^0_\prec);\imath^+$ on the output boundary.
\end{enumerate}
\end{dfn}

\begin{dfn} \label{dfn:extrtil}
For all $n > 1$ and $k \in \mathbb{N}$, we define a molecule $\extrtil{k}{n}$ such that $\bord{}{\alpha}\extrtil{k}{n}$ is isomorphic to $\bord{}{\alpha}\extr{k}{n}$. We let
\begin{align*}
	\extrtil{k}{2} & \eqdef \extr{k}{2}, \\
	\extrtil{k}{n} & \eqdef \extr{k}{n}[\extrtil{k+1}{n-1}/O^{k+1}(\Delta^{n-1})] \; \text{for $n > 2$}.
\end{align*}
Here $O^{k+1}(\Delta^{n-1})$ is seen as a submolecule of $\extr{k}{n}$ through $j_{k,n}$. 

By construction we have a sequence of inclusions of submolecules
\begin{equation*}
	\extrtil{k}{n} \supmol \extrtil{k+1}{n-1} \supmol \ldots \supmol \extrtil{k+n-2}{2} \supmol O^{k+n-1}(\Delta^1) = O^{k+n},
\end{equation*}
and the $r_{k,n}$ induce a sequence of retractions
\begin{equation*}
	\extrtil{k}{n} \surj \extrtil{k+1}{n-1} \surj \ldots \surj \extrtil{k+n-2}{2} \surj O^{k+n}
\end{equation*}
whose composite we call $r'_{k,n}\colon \extrtil{k}{n} \surj O^{k+n}$.

In particular, for all $n > 1$ we have an inclusion $O^n \submol \extrtil{0}{n}$ and a retraction of $\extrtil{0}{n}$ onto $O^n$. By definition of $a_n$ and commutativity of (\ref{eq:retract_rkn}), the diagram
\begin{equation} \label{eq:retract_r0n}
\begin{tikzpicture}[baseline={([yshift=-.5ex]current bounding box.center)}]
	\node (0) at (0,1.5) {$\bord{}{}\Delta^n$};
	\node (1) at (2.5,0) {$O^n$};
	\node (2) at (0,0) {$\extrtil{0}{n}$};
	\node (3) at (2.5,1.5) {$\Delta^n$};
	\draw[1cinc] (0) to (3);
	\draw[1cincl] (0) to (2);
	\draw[1csurj] (2) to node[auto,swap,arlabel] {$r'_{0,n}$} (1);
	\draw[1csurj] (3) to node[auto,arlabel] {$a_n$} (1);
\end{tikzpicture}
\end{equation}
commutes in $\rdcpx$.

By Lemma \ref{lem:substitution_convenient} and Remark \ref{rmk:extr_convenient}, every convenient class contains all the $\extrtil{k}{n}$. 
\end{dfn}


\subsection{Combinatorial homotopy groups}

\begin{dfn}[Horn]
Let $V \incl U$ be an inclusion of atoms in $\dcpx$ such that $\dmn{U} = \dmn{V} + 1$, and let	$\horn{V}{U} \eqdef \bord U \setminus (V \setminus \bord V)$. We call the inclusion $\horn{V}{U} \incl U$ a \emph{horn of $U$}.
\end{dfn}

\begin{dfn}[Filler]
Let $X$ be a diagrammatic set and $U \in \cls{C}$ an atom. A \emph{horn of $U$ in $X$} is a pair of a horn $\horn{V}{U} \incl U$ and a morphism $h\colon \horn{V}{U} \to X$. A \emph{filler} for the horn is a cell $f\colon U \to X$ such that
the triangle
\begin{equation*}
\begin{tikzpicture}[baseline={([yshift=-.5ex]current bounding box.center)}]
	\node (0) at (0,1.5) {$\horn{V}{U}$};
	\node (1) at (2.5,1.5) {$X$};
	\node (2) at (0,0) {$U$};
	\draw[1c] (0) to node[auto,arlabel] {$h$} (1);
	\draw[1cincl] (0) to (2);
	\draw[1c] (2) to node[auto,swap,arlabel] {$f$} (1);
\end{tikzpicture}
\end{equation*}
commutes in $\dgmset$.
\end{dfn}

\begin{dfn}[Kan diagrammatic set]
A diagrammatic set $X$ is \emph{Kan} if every horn in $X$ has a filler.
\end{dfn}

\begin{rmk}
The horns of $\Delta^n$ as an atom correspond bijectively to the standard horns $\{\Lambda_k^n \incl \Delta^n\}_{k=0}^n$ of the $n$\nbd simplex. Thus the defining property of Kan diagrammatic sets extends that of Kan complexes from $\deltacat$ to $\atom$. \emph{A fortiori}, if $X$ is a Kan diagrammatic set, then $\delres{X}$ is a Kan complex.
\end{rmk}

\begin{prop}
Let $X$ be a Kan diagrammatic set. Then
\begin{enumerate}[label=(\alph*)]
	\item all composable diagrams in $X$ are equivalences, and
	\item $X$ has weak composites.
\end{enumerate}
\end{prop}
\begin{proof}
Let $x$ be a composable diagram of shape $U$ in $X$ and let $(x_0\colon U_0 \to X,$ $\imath\colon \bord{}{\alpha}U \incl \bord{}{\alpha}U_0)$ be a dividend for $x$. Without loss of generality, suppose $\alpha = -$. Then 
\begin{equation*}
	V \eqdef \bord{}{-}U_0[\bord{}{+}U/\bord{}{-}U] \celto \bord{}{+}U_0
\end{equation*}
is defined, and so are the pasting $U \cup V$ along the submolecule $\bord{}{+}U \submol \bord{}{-}V$ and $W \eqdef (U \cup V) \celto U_0$. By Lemma \ref{lem:substitution_convenient} and axioms \ref{ax:shift} and \ref{ax:submolecule} of convenient classes, $W \in \cls{C}$. 

The horn $\horn{V}{W}$ of $W$ is the union of $U \incl \bord{}{-}W$ and $U_0 \incliso \bord{}{+}W$ in $W$, so there is a horn $h\colon \horn{V}{W} \to X$ defined by $\restr{h}{U} \eqdef x$ and $\restr{h}{U_0} \eqdef x_0$. By assumption, this has a filler $f\colon W \to X$, which is a cell of type $x \cup x_- \celto x_0$, that is, a lax division of $(x_0, \imath)$ by $x$.

This proves that $\cdgm{X} \subseteq \equifun{}{\cdgm{X}}$, and it follows by coinduction that $\cdgm{X} = \equi{X}$.

For the second statement, observe that for all molecules $U \in \cls{C}$ the pair of the inclusion $\imath^-\colon U \incl (U \celto \compos{U})$ together with a composable diagram $x\colon U \to X$ is a horn of $U \celto \compos{U}$ in $X$. This has a filler $c(x)\colon (U \celto \compos{U}) \to X$, which by the first part is an equivalence, and Proposition \ref{prop:composites_fibrant} applies.
\end{proof}

\begin{cor}
Let $e\colon x \celto y$ be a composable diagram in a Kan diagrammatic set. Then $x \simeq y$.
\end{cor}

\begin{dfn}[Homotopy groups of a Kan diagrammatic set]
Let $X$ be a Kan diagrammatic set. We define $\pin{0}{}{X}$ to be the set of 0\nbd cells $v\colon 1 \to X$, quotiented by $\simeq$.

Let $v$ be a 0\nbd cell in $X$. For all $n > 0$, we define $\pin{n}{}{X,v}$ to be the set of $n$\nbd cells $x\colon O^n \to X$ such that $\bord x =\ !;v$, where $!$ denotes the unique morphism to $1$, quotiented by $\simeq$.

This becomes a group with the following structure.
\begin{enumerate}
	\item Let $[x], [y]$ be classes with representatives $x, y$. Then $x \cp{n-1} y$ is a composable diagram in $X$ and has a weak composite $\compos{x \cp{n-1} y}$. We define $[x] * [y] \eqdef [\compos{x \cp{n-1} y}]$.
	\item The unit is the class $[!;v]$. 
	\item Let $[x]$ be a class with representative $x$. Then $x$ has a weak inverse cell $x^*$, and we define $\invrs{[x]} \eqdef [x^*]$. 
\end{enumerate}
Let $f\colon X \to Y$ be a morphism of Kan diagrammatic sets. For all $n > 0$, we define functions
\begin{equation*}
	\pin{0}{}{f}\colon \pin{0}{}{X} \to \pin{0}{}{Y}, \quad \quad \pin{n}{}{f}\colon \pin{n}{}{X,v} \to \pin{n}{}{Y,f(v)}
\end{equation*}
by $[x] \mapsto [f(x)]$. These are well-defined as functions of sets by Corollary \ref{cor:morphism_preserve_equivalence}.
\end{dfn}

\begin{prop}
Let $f\colon X \to Y$ be a morphism of Kan diagrammatic sets and $v$ a 0\nbd cell in $X$. For all $n > 0$,
\begin{enumerate}[label=(\alph*)]
	\item $\pin{n}{}{X,v}$ is well-defined as a group, and
	\item $\pin{n}{}{f}\colon \pin{n}{}{X,v} \to \pin{n}{}{Y,f(v)}$ is a homomorphism of groups.
\end{enumerate}
\end{prop}
\begin{proof}
There are several things to check.
\begin{itemize}
	\item \emph{$[x] * [y]$ is independent of the representatives $x, y$.} Suppose that $x \simeq x'$ and $y \simeq y'$. By Proposition \ref{prop:pasting_equivalence}, 
	\begin{equation*}
		x \cp{n-1} y \simeq ((x \cp{n-1} y)[y'/y])[x'/x] = x' \cp{n-1} y',
	\end{equation*}
	so $\compos{x \cp{n-1} y} \simeq \compos{x' \cp{n-1} y'}$ for any choice of weak composites.
	\item \emph{Multiplication is associative.} Using Proposition \ref{prop:pasting_equivalence} again,
	\begin{equation*}
		\compos{\compos{x \cp{n-1} y}\cp{n-1}z} \simeq x \cp{n-1} y \cp{n-1} z \simeq \compos{x \cp{n-1} \compos{y\cp{n-1}z}}.
	\end{equation*}
	\item \emph{$[!;v]$ is a unit.} Immediate from Proposition \ref{prop:unitors} since $\eps{}{(!;v)} =\ !;v$.
	\item \emph{$\invrs{[x]}$ is the inverse of $[x]$.} Because $x^*$ is a weak inverse of $x$, there are equivalences exhibiting $x \cp{n-1} x^* \simeq\ !;v$ and $x^* \cp{n-1} x \simeq\ !;v$.
\end{itemize}
This proves that $\pin{n}{}{X,v}$ is a group. Finally, Proposition \ref{prop:basic_weak_composite} implies that $\pin{n}{}{f}$ is compatible with the group multiplication, and $f(!;v) =\ !;f(v)$ that it is compatible with the unit.
\end{proof}

\begin{dfn}
Let $x\colon \Delta^n \to K$ be an $n$\nbd simplex in a simplicial set. For $0 \leq k \leq n$, we write $d_k x \eqdef d^k;x$ and $\bord x = (d_0 x, \ldots, d_n x)$. 
\end{dfn}

\begin{dfn}[Homotopy groups of a Kan complex]
Let $K$ be a Kan complex. We define $\pin{0}{\Delta}{K}$ to be the set of 0\nbd simplices $v\colon \Delta^0 \to X$, quotiented by the relation $v \sim w$ if and only if there exists a 1\nbd simplex $x$ with $\bord x = (w,v)$.

Let $v$ be a 0\nbd simplex in $K$. For all $n > 0$, we define $\pin{n}{\Delta}{X,v}$ to be the set of $n$\nbd simplices $x\colon \Delta^n \to K$ such that $\bord x =\ (!;v,\ldots,!;v)$, where $!$ denotes the unique morphism to $\Delta^0$, quotiented by the relation $x \sim y$ if and only if there exists a morphism $h\colon \Delta^0 \times \Delta^n \to K$ such that the diagrams
\begin{equation*}
\begin{tikzpicture}[baseline={([yshift=-.5ex]current bounding box.center)}]
	\node (0) at (0,1.5) {$\Delta^0 \times \bord \Delta^n$};
	\node (1) at (2.5,1.5) {$\Delta^0$};
	\node (2) at (0,0) {$\Delta^0 \times \Delta^n$};
	\node (3) at (2.5,0) {$K$,};
	\draw[1c] (0) to node[auto,arlabel] {$!$} (1);
	\draw[1c] (2) to node[auto,swap,arlabel] {$h$} (3);
	\draw[1cincl] (0) to (2);
	\draw[1c] (1) to node[auto,arlabel] {$v$} (3);
\end{tikzpicture} \quad \quad \quad
\begin{tikzpicture}[baseline={([yshift=-.5ex]current bounding box.center)}]
	\node (0) at (-.5,1.5) {$\Delta^n$};
	\node (1) at (2.5,0) {$K$};
	\node (2) at (2.5,1.5) {$\Delta^1 \times \Delta^n$};
	\node (3) at (5.5,1.5) {$\Delta^n$};
	\draw[1cinc] (0) to node[auto,arlabel] {$(!;d^1,\idd{})$} (2);
	\draw[1cincl] (3) to node[auto,swap,arlabel] {$(!;d^0,\idd{})$} (2);
	\draw[1c] (0) to node[auto,swap,arlabel] {$x$} (1);
	\draw[1c] (2) to node[auto,arlabel] {$h$} (1);
	\draw[1c] (3) to node[auto,arlabel] {$y$} (1);
\end{tikzpicture}
\end{equation*}
commute in $\sset$. This becomes a group with the following structure.
\begin{enumerate}
	\item Let $[x], [y]$ be classes with representatives $x, y$. There is a horn $\Lambda^{n+1}_1 \to K$ equal to $y$ on $d^0$, to $x$ on $d^2$, and to $!;v$ on $d^k$ for all $k > 2$. This horn has a filler $f\colon \Delta^{n+1} \to K$, and we define $[x] * [y] \eqdef [d_1 f]$.
	\item The unit is the class $[!;v]$.
	\item Let $[x]$ be a class with representative $x$. There is a horn $\Lambda^{n+1}_0 \to K$ equal to $x$ on $d^2$ and to $!;v$ on all other faces. This horn has a filler $f\colon \Delta^{n+1} \to K$, and we define $\invrs{[x]} \eqdef [d_0 f]$.
\end{enumerate}
Let $f\colon K \to L$ be a morphism of Kan complexes. For all $n > 0$, we define functions
\begin{equation*}
	\pin{0}{\Delta}{f}\colon \pin{0}{\Delta}{K} \to \pin{0}{}{L}, \quad \quad \pin{n}{\Delta}{f}\colon \pin{n}{\Delta}{K,v} \to \pin{n}{\Delta}{L,f(v)}
\end{equation*}
by $[x] \mapsto [f(x)]$. 
\end{dfn}

\begin{comm}
We adopted the opposite convention with respect to \cite[Section I.7]{goerss2009simplicial}, by having the relevant faces of a simplex be the first few, as opposed to the last few. This is to avoid oscillations between the input and output boundary of the simplex seen as an atom. 
\end{comm}

\begin{prop}
Let $f\colon K \to L$ be a morphism of Kan complexes and $v$ a 0\nbd simplex in $K$. Then
\begin{enumerate}[label=(\alph*)]
	\item $\pin{n}{\Delta}{K,v}$ is well-defined as a group for all $n > 0$,
	\item $\pin{n}{\Delta}{K,v}$ is abelian for all $n > 1$, and
	\item $\pin{n}{\Delta}{f}$ is well-defined for all $n \in \mathbb{N}$ and a homomorphism of groups for all $n > 0$.
\end{enumerate}
\end{prop}
\begin{proof}
This is \cite[Theorem I.7.2]{goerss2009simplicial}.
\end{proof}

\begin{prop} \label{prop:alpha_0_isomorphism}
Let $X$ be a Kan diagrammatic set. The sets $\pin{0}{}{X}$ and $\pin{0}{\Delta}{\delres{X}}$ are naturally isomorphic.
\end{prop}
\begin{proof}
If $n \in \{0,1\}$, the $n$\nbd cells in $X$ are in bijection with the $n$\nbd simplices in $\delres{X}$, and a 1\nbd cell $x\colon v \celto w$ corresponds to a 1\nbd simplex with $\bord x = (w,v)$.
\end{proof}

\begin{dfn}
Let $X$ be a Kan diagrammatic set and $v$ a 0\nbd cell in $X$. For all $n > 0$, we define a function
\begin{align*}
	\alpha_n\colon \pin{n}{}{X,v} & \to \pin{n}{\Delta}{\delres{X},v}, \\
	[x] & \mapsto [a_n;x],
\end{align*}
where $a_n\colon \Delta^n \surj O^n$ is the folding map of \S \ref{dfn:folding}.

We show that this is well-defined, that is, independent of the representative of $[x]$. If $x \simeq y$, there exists a cell $e\colon x \celto y$ of shape $O^{n+1}$. Then $a_{n+1};e$ is an $(n+1)$\nbd simplex in $\delres{X}$ with
\begin{equation*}
	\bord(a_{n+1};e) = (a_n;y, a_n;x, !;v, \ldots, !;v)
\end{equation*}
by the commutativity of the two diagrams in (\ref{eq:folding_faces}). This exhibits $[a_n;x]$ as $[!;v]*[a_n;y] = [a_n;y]$ in $\pin{n}{\Delta}{\delres{X},v}$, so $\alpha_n[x] = \alpha_n[y]$.

We show that $\alpha_n$ is a homomorphism of groups. First of all, $a_n;!;v =\ !;v$ so $\alpha_n$ preserves the unit. 

Suppose $[x]*[y] = [z]$ in $\pin{n}{}{X,v}$. Then $z \simeq x \cp{n-1} y$, so there exists a cell $e\colon z \celto x \cp{n-1} y$ which has shape $\Phi^{n+1}$. 

Let $c_{n+1}\colon \Delta^{n+1} \surj \Phi^{n+1}$ be the folding map of \S \ref{dfn:folding_compositor}. Then $c_{n+1}; e$ is an $(n+1)$\nbd simplex in $\delres{X}$ with
\begin{equation*}
	\bord(c_{n+1};e) = (a_n;y, a_n;z, a_n;x, !;v, \ldots, !;v)
\end{equation*}
by the commutativity of the three diagrams in (\ref{eq:cn_faces}). This exhibits $[a_n;z]$ as $[a_n;x] * [a_n;y]$ in $\pin{n}{\Delta}{\delres{X},v}$, so $\alpha_n[x]*\alpha_n[y] = \alpha_n[z]$.

Finally, the $\alpha_n$ are natural in the sense that, if $f\colon X \to Y$ is a morphism of Kan diagrammatic sets, then the square
\begin{equation} \label{eq:naturality_alpha_n}
\begin{tikzpicture}[baseline={([yshift=-.5ex]current bounding box.center)}]
	\node (0) at (0,1.5) {$\pin{n}{}{X,v}$};
	\node (1) at (4,1.5) {$\pin{n}{}{Y,f(v)}$};
	\node (2) at (0,0) {$\pin{n}{\Delta}{\delres{X},v}$};
	\node (3) at (4,0) {$\pin{n}{\Delta}{\delres{Y},\delres{f}(v)}$};
	\draw[1c] (0) to node[auto,arlabel] {$\pin{n}{}{f}$} (1);
	\draw[1c] (2) to node[auto,swap,arlabel] {$\pin{n}{\Delta}{\delres{f}}$} (3);
	\draw[1c] (0) to node[auto,swap,arlabel] {$\alpha_n$} (2);
	\draw[1c] (1) to node[auto,arlabel] {$\alpha_n$} (3);
\end{tikzpicture}
\end{equation}
commutes in the category of groups and homomorphisms.
\end{dfn}

\begin{thm} \label{thm:alpha_n_isomorphism}
Let $X$ be a Kan diagrammatic set and $v$ a 0\nbd cell in $X$. Then $\alpha_n\colon \pin{n}{}{X,v} \to \pin{n}{\Delta}{\delres{X},v}$ is an isomorphism for all $n > 0$.
\end{thm}
\begin{proof}
We will first prove that $\alpha_n$ is surjective, that is, for all classes $[x]$ in $\pin{n}{\Delta}{\delres{X},v}$, represented by an $n$\nbd simplex $x$ in $\delres{X}$, there exists an $n$\nbd cell $x'\colon O^n \to X$ such that $[a_n;x'] = [x]$. 

Then we will prove that $\alpha_n$ is injective, by showing it has a trivial kernel, that is, if $[a_n;x] = [!;v]$ in $\pin{n}{\Delta}{\delres{X},v}$ then $[x] = [!;v]$ in $\pin{n}{}{X,v}$.

We will repeatedly use the molecules and maps defined in \S \ref{dfn:extr} and \S \ref{dfn:extrtil}.
\begin{itemize}
	\item \emph{Surjectivity.} If $n = 1$, we have $O^1 = \Delta^1$ and $a_1 = \idd{}$ so there is nothing to prove. \\
	If $n > 1$, let $x$ be an $n$\nbd simplex in $\delres{X}$. Let $U \eqdef \Delta^n \celto \extrtil{0}{n}$, and let $O^n \incl U$ be the inclusion of $O^n$ as a submolecule of $\extrtil{0}{n} = \bord{}{+}U$. Because $\bord{}{-}x =\ !;v$, there is a horn $h\colon \horn{O^n}{U} \to X$ equal to $x$ on $\bord{}{-}U$ and to $!;v$ everywhere else. This horn has a filler $f\colon U \to X$ and we define $x' \eqdef \restr{f}{O^n}$. \\
	Let $U' \eqdef \extrtil{0}{n} \celto \Delta^n$, and $p\colon U' \surj O^n$ be the map defined by
	\begin{enumerate}
		\item $\restr{p}{\bord{}{-}U} \eqdef r'_{0,n}$,
		\item $\restr{p}{\bord{}{+}U} \eqdef a_n$, and
		\item $p$ sends the greatest element of $U'$ to $\undl{n}$.
	\end{enumerate}
	This is well-defined by the commutativity of (\ref{eq:retract_r0n}). Consider the $(n+1)$\nbd cell $p;x'$. Because $r'_{0,n}$ is a retraction onto $O^n$, the diagram $\bord{}{-}(p;x')$ is equal to $x'$ on $O^n$ and to $!;v$ everywhere else, so $\bord{}{-}(p;x') = \bord{}{+}f$. By construction, $\bord{}{+}(p;x')$ is equal to $a_n;x'$. \\
	It follows that the pasting $f \cp{n} (p;x')$ is defined. Let $k$ be a weak composite. Then $k$ has shape $\infl{\Delta^n}$ and type $x \celto a_n;x'$. \\
	Let $s^0_\prec\colon \Delta^{n+1} \surj \infl{\Delta^n}$ be defined as in \S \ref{dfn:fattening}. Then
	\begin{equation*}
		\bord (s^0_\prec; k) = (a_n;x', x, !;v, \ldots, !;v),
	\end{equation*}
	which exhibits $[x] = [!;v] * [a_n;x'] = [a_n;x']$ in $\pin{n}{\Delta}{\delres{X},v}$.
	\item \emph{Injectivity.} Let $x\colon O^n \to X$ be such that $[a_n;x] = [!;v]$ in $\pin{n}{\Delta}{\delres{X},v}$. By \cite[Lemma I.7.4]{goerss2009simplicial}, there is an $(n+1)$\nbd simplex $y$ in $\delres{X}$ with 
	\begin{equation*}
	\bord y = (a_n;x, !;v, \ldots, !;v). 
	\end{equation*}
	Let $U \eqdef \Delta^{n+1} \celto \extr{0}{n+1}$, and let $\infl{\Delta^n} \incl U$ be $j_{0,n};\imath^+$. There is a horn $h\colon \horn{\infl{\Delta^n}}{U} \to X$ equal to $y$ on $\bord{}{-}U$ and to $!;v$ everywhere else. This horn has a filler $f\colon U \to X$ and we let $e \eqdef \restr{f}{\infl{\Delta^n}}$. This is an $(n+1)$\nbd cell of type $!;v \celto a_n;x$. \\
	Let $W$ be the colimit in $\dcpxin$ of the diagram
	\begin{equation*} 
	\begin{tikzpicture}[baseline={([yshift=-.5ex]current bounding box.center)}]
	\node (0) at (-2,-1.5) {$\Delta^n \celto \extrtil{0}{n}$};
	\node (1) at (0,0) {$O^n$};
	\node (2) at (2,-1.5) {$O^{n+1}$};
	\node (3) at (4,0) {$O^n$};
	\node (4) at (6,-1.5) {$\extrtil{0}{n} \celto \Delta^n$};
	\draw[1cincl] (1) to node[auto,swap,arlabel] {$j^+$} (0);
	\draw[1cinc] (1) to node[auto,arlabel] {$\imath^-$} (2);
	\draw[1cincl] (3) to node[auto,swap,arlabel] {$\imath^+$} (2);
	\draw[1cinc] (3) to node[auto,arlabel] {$j^-$} (4);
	\end{tikzpicture}
	\end{equation*}
	where $j^+, j^-$ are the inclusions of $O^n$ as a submolecule of $\extrtil{0}{n}$. This can be constructed as a sequence of two pastings along submolecules, so $W \in \cls{C}$. Then we define $U' \eqdef \infl{\Delta^n} \celto W$, and let $O^{n+1} \incl U'$ be the inclusion into the colimit $W = \bord{}{+}U'$. \\
	There is a horn $h'\colon \horn{O^{n+1}}{U'} \to X$ that is equal to $e$ on $\bord{}{-}U'$, to $!;v$ on $(\Delta^n \celto \extrtil{0}{n}) \incl \bord{}{+}U'$, and to $p;x$ on $(\extrtil{0}{n} \celto \Delta^n) \incl \bord{}{+}U'$, where $p$ is the map defined in the proof of surjectivity. This is well-defined because $\bord{}{-}(!;v) =\ !;v = \bord{}{-}e$ and $\bord{}{+}(p;x) = a_n;x = \bord{}{+}e$. \\
	This horn has a filler $f'\colon U' \to X$ and we define $e' \eqdef \restr{f'}{O^{n+1}}$. Then $e'$ has type $!;v \celto x$, so it exhibits $[!;v] = [x]$ in $\pin{n}{}{X,v}$.
\end{itemize}
This completes the proof.
\end{proof}

\begin{dfn}
In the classical model structure on $\sset$, a weak equivalence of Kan complexes is a morphism $f\colon K \to L$ with the property that $\pin{n}{\Delta}{f}$ is an isomorphism for all $n \in \mathbb{N}$ and all choices of $v\colon \Delta^0 \to K$. 

By analogy, we define a weak equivalence of Kan diagrammatic sets to be a morphism $f\colon X \to Y$ such that $\pin{n}{}{f}$ is an isomorphism for all $n \in \mathbb{N}$ and all choices of $v\colon 1 \to X$. A weak equivalence in the sense of \S \ref{dfn:weak_equivalence} is also a weak equivalence in this sense.
\end{dfn}

\begin{cor}
If $f\colon X \to Y$ is a weak equivalence of Kan diagrammatic sets, then $\delres{f}\colon \delres{X} \to \delres{Y}$ is a weak equivalence of Kan complexes.
\end{cor}
\begin{proof}
Follows from the combination of Proposition \ref{prop:alpha_0_isomorphism}, Theorem \ref{thm:alpha_n_isomorphism}, and commutativity of (\ref{eq:naturality_alpha_n}).
\end{proof}


\subsection{Geometric realisation}

\begin{dfn}[Nerve of a poset]
Let $P$ be a poset. The \emph{nerve} of $P$ is the simplicial set $NP$ whose $n$\nbd simplices are chains $(x_0 \leq \ldots \leq x_n)$ in $P$ and
\begin{align*}
	d_k(x_0 \leq \ldots \leq x_n) & \eqdef (x_0 \leq \ldots \leq x_{k-1} \leq x_{k+1} \leq \ldots \leq x_n), \\
	s_k(x_0 \leq \ldots \leq x_n) & \eqdef (x_0 \leq \ldots \leq x_{k} \leq x_{k} \leq \ldots \leq x_n)
\end{align*}
for $0 \leq k \leq n$. This extends to a functor $N\colon \pos \to \sset$ defined by
\begin{equation*}
	Nf\colon (x_0 \leq \ldots \leq x_n) \mapsto (f(x_0) \leq \ldots \leq f(x_n))
\end{equation*}
for each order-preserving function $f$ of posets.
\end{dfn}

\begin{dfn} \label{dfn:exfun}
Precomposing $N\colon \pos \to \sset$ with the forgetful functor $\atom \to \pos$, we obtain a functor $k\colon \atom \to \sset$. 

By \cite[Theorem 1.2.1]{riehl2014categorical} the left Kan extension of this functor along the Yoneda embedding $\atom \incl \dgmset$ exists, producing a functor 
\begin{equation*}
	k\colon \dgmset \to \sset
\end{equation*}
with a right adjoint $p\colon \sset \to \dgmset$.

Thus we have a pair of adjunctions
\begin{equation*}
\begin{tikzpicture}[baseline={([yshift=-.5ex]current bounding box.center)}]
	\node (0) at (0,0) {$\sset$};
	\node (1) at (3,0) {$\dgmset$};
	\node (2) at (6,0) {$\sset$};
	\draw[1c, out=30,in=150] (0.east |- 0,.15) to node[auto,arlabel] {$\imath_\Delta$} (1.west |- 0,.15);
	\draw[1c, out=-150,in=-30] (1.west |- 0,-.15) to node[auto,arlabel] {$\delres{-}$} (0.east |- 0,-.15);
	\node at (1.5,0) {$\bot$};
	\draw[1c, out=30,in=150] (1.east |- 0,.15) to node[auto,arlabel] {$k$} (2.west |- 0,.15);
	\draw[1c, out=-150,in=-30] (2.west |- 0,-.15) to node[auto,arlabel] {$p$} (1.east |- 0,-.15);
	\node at (4.5,0) {$\bot$};
\end{tikzpicture}
\end{equation*}
and we claim that $k\imath_\Delta$ is naturally isomorphic to the barycentric subdivision endofunctor $\subdiv$, while $\delres{(p-)}$ is naturally isomorphic to its right adjoint $\exfun$ \cite[Section 4.6]{fritsch1990cellular}.

First of all, the restriction of the forgetful functor $\atom \to \pos$ to $\deltacat$ is precisely the functor sending the $n$\nbd simplex to its poset of non-degenerate simplices ordered by inclusion. By definition, its post-composition with $N$ is the barycentric subdivision $\subdiv\colon \Delta \to \sset$. Then the diagram of functors 
\begin{equation*}
\begin{tikzpicture}[baseline={([yshift=-.5ex]current bounding box.center)}]
	\node (0) at (-1.5,1.5) {$\sset$};
	\node (1) at (1,1.5) {$\dgmset$};
	\node (0b) at (-1.5,0) {$\deltacat$};
	\node (1b) at (1,0) {$\atom$};
	\node (2b) at (3.5,0) {$\sset$};
	\draw[1c] (0) to node[auto,arlabel] {$\imath_\Delta$} (1);
	\draw[1c,out=0,in=105] (1) to node[auto,arlabel] {$k$} (2b);
	\draw[1cinc] (0b) to (1b);
	\draw[1c] (1b) to node[auto,arlabel] {$k$} (2b);
	\draw[1cinc] (0b) to (0);
	\draw[1cinc] (1b) to (1);
	\draw[1c,out=-30,in=-150] (0b) to node[auto,swap,arlabel] {$\subdiv$} (2b);
\end{tikzpicture}
\end{equation*}
commutes up to natural isomorphism. 

Because $\imath_\deltacat$ is the left Kan extension of $\deltacat \incl \dgmset$ along $\deltacat \incl \sset$, and $k$ as a left adjoint preserves left Kan extensions, $k\imath_\Delta$ is the left Kan extension of $\subdiv\colon \deltacat \to \sset$ along $\deltacat \incl \sset$. By definition $k\imath_\Delta$ is equal to $\subdiv\colon \sset \to \sset$ up to natural isomorphism. 

Finally, $\exfun$ is defined up to natural isomorphism as the right adjoint of $\subdiv$, and $\delres{(p-)}$ is right adjoint to $k\imath_\Delta$.
\end{dfn}

\begin{dfn}[Last vertex map]
Let $[n]$ be the linear order $\{0 < \ldots < n\}$. There is an order-preserving map $\gamma_n\colon \Delta^n \to [n]$ defined by
\begin{equation*}
	\ldots 1^j 0^k \mapsto n-k \quad \text{if $j > 0$}
\end{equation*}
with the encoding of \S \ref{dfn:simplex_encoding}. 

The family $\{N(\gamma_n)\colon \subdiv \Delta^n \to \Delta^n\}_{n\in \mathbb{N}}$ of morphisms in $\sset$ is natural over $\deltacat$, so it extends uniquely to a natural transformation $d\colon \subdiv \to \bigid{\sset}$. By adjointness this produces a natural transformation $e\colon \bigid{\sset} \to \exfun$. 
\end{dfn}

\begin{prop} \label{prop:exfun_weak_equivalence}
Let $K$ be a Kan complex. Then $\exfun K$ is a Kan complex and $e_K\colon K \to \exfun K$ is a weak equivalence.
\end{prop}
\begin{proof}
This is \cite[Corollary 4.6.21]{fritsch1990cellular}.
\end{proof}

\begin{dfn}
We write $\cghaus$ for the category of compactly generated Hausdorff spaces and continuous maps. The geometric realisation of simplicial sets is a functor $\realis{-}_\deltacat\colon \sset \to \cghaus$ with a right adjoint $S_\deltacat\colon \cghaus \to \sset$.
\end{dfn}

\begin{dfn}[Geometric realisation] The \emph{geometric realisation} of diagrammatic sets is the functor
\begin{equation*}
	\realis{-}\colon \dgmset \to \cghaus
\end{equation*}
obtained as the composite of $k\colon \dgmset \to \sset$ and $\realis{-}_\deltacat\colon \sset \to \cghaus$. This functor has a right adjoint
\begin{equation*}
	\sing{}\colon \cghaus \to \dgmset
\end{equation*}
obtained as the composite of $S_\deltacat$ and $p$. Diagrams of shape $U$ in $\sing{X}$ are continuous maps $\realis{U} \to X$.
\end{dfn}

\begin{dfn}
For each $n \in \mathbb{N}$, we let $D^n$ be the topological closed $n$\nbd ball and $\bord D^n$ the topological $(n-1)$\nbd sphere. 
\end{dfn}

\begin{comm}
In what follows we use some results in combinatorial topology that refer to the geometric realisation of the \emph{order complex} of a poset, which is an ordered simplicial complex, rather than the nerve, a simplicial set. It seems to be folklore that the two realisations are equal up to homeomorphism. 

We will also use the fact that these realisations are compatible with unions and intersections of closed subsets.
\end{comm}

\begin{prop} \label{prop:realisation_sphere}
Let $U$ be an $n$\nbd molecule with spherical boundary. Then
\begin{enumerate}[label=(\alph*)]
	\item $\realis{U}$ is homeomorphic to $D^n$ and
	\item $\realis{\bord U}$ is homeomorphic to $\bord D^n$.
\end{enumerate}
If $U$ is an atom, $n > 0$, and $\imath\colon \horn{V}{U} \incl U$ is a horn of $U$, then $\realis{\imath}$ is an embedding of $D^{n-1}$ into the boundary of $D^n$.
\end{prop}
\begin{proof}
If $n = 0$ this is clear. If $n > 0$ we may assume that $\realis{\bord{}{+}U}$ and $\realis{\bord{}{-}U}$ are $(n-1)$\nbd balls and, since $U$ has spherical boundary, their intersection $\realis{\bord{n-2}{}U}$ is an $(n-2)$\nbd sphere. It follows from \cite[Theorem 2]{zeeman1966seminar} that $\realis{\bord U}$ is an $(n-1)$\nbd sphere. 

If $U$ is an atom, by \cite[Proposition 3.1]{bjorner1984posets} this suffices to prove that $\realis{U}$ is an $n$\nbd ball. Otherwise, decompose $U$ as $V_1 \cp{n-1} \ldots \cp{n-1} V_m$, where $V_i$ contains a single atom $U_i$, as by Lemma \ref{lem:spherical_moves}. Let
\begin{align*}
	U'_0 & \eqdef \bord{}{-}U \celto \bord{}{-}U, \\
	U'_k & \eqdef U'_{k-1} \cp{n-1} V_k \quad \text{for $1 \leq k \leq m$}.
\end{align*}
We prove by induction on $k$ that $\realis{U'_k}$ is an $n$\nbd ball. First, $U'_0$ is an $n$\nbd dimensional atom with spherical boundary, so $\realis{U'_0}$ is an $n$\nbd ball. 

If $k > 0$, the molecule $U'_k$ is the pasting of $U'_{k-1}$ and the atom $U_k$ along a submolecule of $\bord{}{+}U'_{k-1}$. By the inductive hypothesis, $\realis{U'_k}$ is the union of the closed $n$\nbd balls $\realis{U'_{k-1}}$ and $\realis{U_k}$ along a closed $(n-1)$\nbd ball in their boundaries. Then $\realis{U'_k}$ is a closed $n$\nbd ball, again by \cite[Theorem 2]{zeeman1966seminar}.

Let $W \eqdef (\bord{}{-}U \celto \bord{}{+}U) \celto U'_m$. We show that $\realis{\bord W}$ is an $n$\nbd sphere as in the first part of the proof. Because $U$ has spherical boundary, the intersection of the atoms $(\bord{}{-}U \celto \bord{}{-}U) \incl \bord{}{+}W$ and $(\bord{}{-}U \celto \bord{}{+}U) \incliso \bord{}{-}W$ in $W$ is their common input boundary $\bord{}{-}U$. Let $V$ be their union. 

As the union of two closed $n$\nbd balls along a closed $(n-1)$\nbd ball in their boundaries, $\realis{V}$ is an $n$\nbd ball. Then $\realis{V} \incl \realis{\bord W}$ is the embedding of a closed $n$\nbd ball into an $n$\nbd sphere. By \cite[Theorem 3]{zeeman1966seminar} its closed complement $\realis{U}$ is a closed $n$\nbd ball. The statement about horns of atoms holds as a consequence of the same result.
\end{proof}

\begin{dfn}[Face poset]
Let $X$ be a CW complex with a set $A$ of attaching maps. The \emph{face poset} of $X$ is the poset $\face{X}$  whose set of elements is $A$, ordered by the relation $(x\colon D^n \to X) \leq (y\colon D^m \to X)$ if and only if $x(D^n) \subseteq y(D^m)$.
\end{dfn}

\begin{dfn}[Regular CW complex]
A CW complex is \emph{regular} if each attaching map is a homeomorphism onto its image.
\end{dfn}

\begin{prop} \label{prop:cw_poset}
Let $P$ be a regular directed complex. The underlying poset of $P$ is the face poset of a regular CW complex.
\end{prop}
\begin{proof}
Follows from Proposition \ref{prop:realisation_sphere} and \cite[Proposition 3.1]{bjorner1984posets}.
\end{proof}

\begin{prop}
For all spaces $X$, the diagrammatic set $\sing{X}$ is Kan.
\end{prop}
\begin{proof}
Follows by adjointness from Proposition \ref{prop:realisation_sphere} and the fact that all maps $D^{n-1} \to X$ admit extensions along embeddings $D^{n-1} \incl D^n$.
\end{proof}

\begin{prop} \label{prop:simplicial_realisation_iso}
Let $X$ be a Kan diagrammatic set and $v$ a 0\nbd cell in $X$. There are natural isomorphisms 
\begin{equation*}
	\pin{0}{}{X} \to \pin{0}{}{\realis{\delres{X}}_\Delta}, \quad \quad \pin{n}{}{X,v} \to \pin{n}{}{\realis{\delres{X}}_\Delta,\realis{v}_\Delta} \text{ for all $n$ > 0}.
\end{equation*}
\end{prop}
\begin{proof}
It suffices to compose the natural isomorphisms 
\begin{equation*}
	\pin{0}{}{X} \to \pin{0}{\Delta}{\delres{X}}, \quad \quad \pin{n}{}{X,v} \to \pin{n}{\Delta}{\delres{X},v}
\end{equation*}
from the previous section with the natural isomorphisms 
\begin{equation*}
	\pin{0}{\Delta}{K} \to \pin{0}{}{\realis{K}_\Delta}, \quad \quad \pin{n}{\Delta}{K,v} \to \pin{n}{}{\realis{K}_\Delta,\realis{v}_\Delta}
\end{equation*}
defined after \cite[Proposition I.11.1]{goerss2009simplicial}.
\end{proof}

\begin{cor} \label{cor:weak_equivalence_dgmset_space}
If $f\colon X \to Y$ is a weak equivalence of Kan diagrammatic sets, then $\realis{\delres{f}}_\Delta\colon \realis{\delres{X}}_\Delta \to \realis{\delres{Y}}_\Delta$ is a weak equivalence of spaces.
\end{cor}

\begin{thm} \label{thm:span_weak_equivalences}
Each space $X$ is weakly equivalent to $\realis{\delres{(\sing{X})}}_\Delta$ via a natural span of weak equivalences.
\end{thm}
\begin{proof}
Let $\epsilon^\Delta$ be the counit of the adjunction between $\realis{-}_\Delta$ and $S_\Delta$. Because this is a Quillen equivalence, all the components of $\epsilon^\Delta$ are weak equivalences. 

By \S \ref{dfn:exfun}, for all spaces $X$ the Kan complexes $\exfun (S_\Delta X)$ and $(\sing{X})_\Delta$ are naturally isomorphic. By Proposition \ref{prop:exfun_weak_equivalence}, we have a family 
\begin{equation*}
	e_{S_\Delta X}\colon S_\Delta X \to (\sing{X})_\Delta
\end{equation*}
of weak equivalences of Kan complexes, natural in $X$. It follows that
\begin{equation*}
	\begin{tikzpicture}[baseline={([yshift=-.5ex]current bounding box.center)}]
	\node (0) at (-2,-1.5) {$X$};
	\node (1) at (0,0) {$\realis{S_\Delta X}_\Delta$};
	\node (2) at (2,-1.5) {$\realis{\delres{(\sing{X})}}_\Delta$};
	\draw[1c] (1) to node[auto,swap,arlabel] {$\epsilon^\Delta_X$} (0);
	\draw[1c] (1) to node[auto,arlabel] {$\realis{e_{S_\Delta X}}_\Delta$} (2);
	\end{tikzpicture}
\end{equation*}
is a span of weak equivalences in $\cghaus$, natural in $X$. 
\end{proof}

\begin{cor} \label{cor:singular_dgmset_iso}
Let $X$ be a space and $v$ a point in $X$. There are natural isomorphisms 
\begin{equation*}
	\pin{0}{}{X} \to \pin{0}{}{\sing{X}}, \quad \quad \pin{n}{}{X,v} \to \pin{n}{}{\sing{X}, \sing{v}} \text{ for all $n$ > 0}.
\end{equation*}
\end{cor}
\begin{proof}
Compose the natural isomorphisms of Proposition \ref{prop:simplicial_realisation_iso} with those induced by the natural span of Theorem \ref{thm:span_weak_equivalences}. 
\end{proof}

\begin{cor} \label{cor:weak_equivalence_space_dgmset}
If $f\colon X \to Y$ is a weak equivalence of spaces, then $\sing{f}\colon \sing{X} \to \sing{Y}$ is a weak equivalence of Kan diagrammatic sets.
\end{cor}

\begin{comm}
It follows from Corollary \ref{cor:weak_equivalence_dgmset_space} and Corollary \ref{cor:weak_equivalence_space_dgmset} that $\realis{\delres{-}}_\Delta$ and $\sing{}$ descend to functors between the homotopy categories of spaces and Kan diagrammatic sets, that is, their localisations at the weak equivalences. By Theorem \ref{thm:span_weak_equivalences}, $\sing{}$ is a homotopical right inverse to $\realis{\delres{-}}_\Delta$. This proves Simpson's homotopy hypothesis for Kan diagrammatic sets as $\infty$\nbd groupoids.

We believe that $\sing{}$ is in fact a two-sided homotopical inverse both to $\realis{\delres{-}}_\Delta$ and to $\realis{-}$, and that the adjunction of $\realis{-}$ and $\sing{}$ is a Quillen equivalence between an appropriate model structure on $\dgmset$ and the classical model structure on $\cghaus$. 

Proving this stronger result with the same combinatorial approach that we have adopted so far seems to require a theory of oriented simplicial subdivisions of atoms. We leave its development to future work.
\end{comm}

\begin{dfn}[Coskeleta] 
Let $n \in \mathbb{N}$, and let $\atom_n$ be the full subcategory of $\atom$ on the atoms $U$ with $\dmn{U} \leq n$. The restriction functor $-_{\leq n}\colon \dgmset \to \psh{}{\atom_n}$ has a right adjoint. 

Let $\coskel{n}{}$ be the monad induced by this adjunction. The \emph{$n$\nbd coskeleton} of a diagrammatic set $X$ is the unit $\eta_X\colon X \to \coskel{n}{X}$.

The coskeleta of $X$ form a sequence of morphisms under $X$
\begin{equation} \label{eq:coskel-tower}
	\ldots \to \coskel{n}{X} \to \coskel{n-1}{X} \to \ldots \to \coskel{0}{X}.
\end{equation}
\end{dfn}

\begin{thm}
Let $X$ be a Kan diagrammatic set and $v$ a 0\nbd cell in $X$. For all $n \in \mathbb{N}$,
\begin{enumerate}[label=(\alph*)]
	\item $\coskel{n+1}{X}$ is a Kan diagrammatic set,
	\item $\pin{m}{}{\coskel{n+1}{X},v}$ is trivial for all $m > n$, and
	\item $\pin{k}{}{\eta_X}$ is an isomorphism for all $k \leq n$.
\end{enumerate}
\end{thm}
\begin{proof}
Same as the proof for Kan complexes \cite[Proposition 8.8]{may1992simplicial}.
\end{proof}

\begin{comm}
It follows that (\ref{eq:coskel-tower}) is a Postnikov tower for a Kan diagrammatic set $X$. Combined with Proposition \ref{prop:simplicial_realisation_iso} and Theorem \ref{thm:span_weak_equivalences}, this implies that, for each space $X$, the image through $\realis{\delres{-}}_\Delta$ of the Postnikov tower of $\sing{X}$ is a Postnikov tower for $X$ up to natural weak equivalence.

This completes the proof of Simpson's homotopy hypothesis for Kan diagrammatic sets. 
\end{comm}

\begin{dfn} \label{cons:kv_mistake}
We conclude by exhibiting an explicit mistake in the proof of \cite[Lemma 3.4]{kapranov1991infty}.

The functor $\Pi$ in Kapranov--Voevodsky corresponds to the left adjoint to the functor $\omegacat \to \dgmset$ of Remark \ref{rmk:dgmset_nerve}. This can be computed by the coend formula for the left Kan extension of $\atom \incl \omegacat$ along $\atom \incl \dgmset$, which given a diagrammatic set $X$ presents $\Pi X$ as the colimit of a diagram of $\omega$\nbd categories $\mol{}{U}^*$ indexed by cells $x\colon U \to X$. 

If $U$ is a molecule and $\mol{}{-}^*$ preserves the colimit diagram of inclusions of the atoms of $U$, for example if we are working with an algebraically free class of molecules, then a diagram $x\colon U \to X$ determines a diagram $x'\colon \mol{}{U}^* \to \Pi X$. Our counterexample only uses totally loop-free molecules, so this is guaranteed by Example \ref{exm:lf_algebraically_free}. The diagram $x$ is what Kapranov and Voevodsky call a \emph{materialisation} of the cell $x'(U)$ in $\Pi X$. 

The crucial step in the proof may be rephrased as follows. Suppose that there is a cell $x$ in $\Pi X$ with two different materialisations $x_1$ and $x_2$ in $X$ such that $\bord x_1$ is equal to $\bord x_2$. Then Kapranov and Voevodsky claim that there exists a materialisation $e\colon x_1 \celto x_2$ in $X$ of the unit cell $\eps{}{x}$ in $\Pi X$. Because $X$ is an arbitrary diagrammatic set, this materialisation must be a diagram of degenerate cells, but we will not use this hypothesis. 

We produce a counterexample to this claim. Unsurprisingly, it amounts in essence to a pair of braidings collapsed by the strict Eckmann--Hilton argument, which is also the source of Simpson's refutation. Let $X$ be a diagrammatic set with a 0\nbd cell $v$ and two 2\nbd cells $a, b$ of type $!;v \celto\ !;v$. We construct two ``braiding'' 3\nbd diagrams $x_1, x_2$ with spherical boundary and type $a \cp{1} b \celto b \cp{1} a$ whose 3\nbd cells are all degenerate. 

The diagram $x_1$ has the form
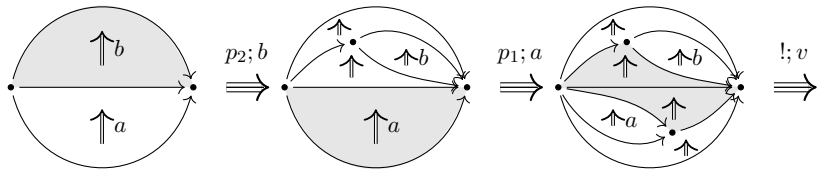
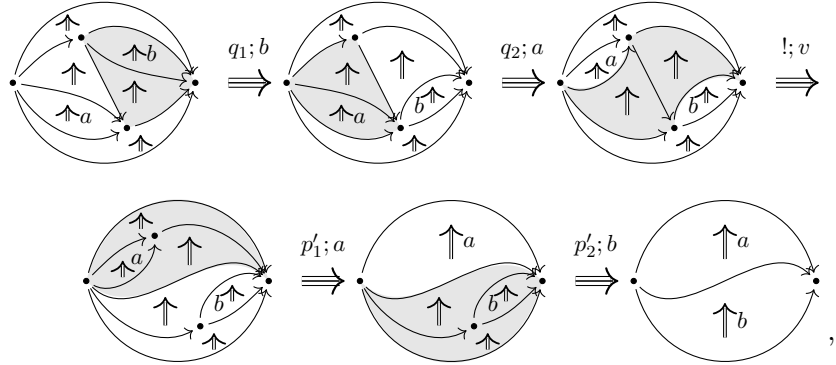
\begingroup
\allowdisplaybreaks
\begin{align*}
\begin{tikzpicture}[baseline={([yshift=-.5ex]current bounding box.center)},scale=.8]
\begin{scope}
	\path[fill, color=gray!20] (-1.5,0) to [out=75,in=105,looseness=1.6] (1.5,0) -- cycle;
	\node[0c] (0) at (-1.5,0) {};
	\node[0c] (1) at (1.5,0) {};
	\draw[1c, out=75, in=105, looseness=1.5] (0) to (1);
	\draw[1c, out=-75, in=-105, looseness=1.5] (0) to (1);
	\draw[1c] (0) to (1);
	\draw[2c] (0,-1.2) to node[auto,swap,arlabel] {$\,a$} (0,-.1);
	\draw[2c] (0,.1) to node[auto,swap,arlabel] {$\,b$} (0,1.2);
	\draw[3c1] (1.65,0) to node[above=6pt,arlabel] {$p_2;b$} (3.1,0);
	\draw[3c2] (1.65,0) to (3.1,0);
	\draw[3c3] (1.65,0) to (3.1,0);
\end{scope}
\begin{scope}[shift={(4.5,0)}]
	\path[fill, color=gray!20] (-1.5,0) to [out=-75,in=-105,looseness=1.6] (1.5,0) -- cycle;
	\node[0c] (0) at (-1.5,0) {};
	\node[0c] (1) at (1.5,0) {};
	\node[0c] (m1) at (-.375,.75) {};
	\draw[1c, out=75, in=105, looseness=1.5] (0) to (1);
	\draw[1c, out=-75, in=-105, looseness=1.5] (0) to (1);
	\draw[1c] (0) to (1);
	\draw[1c, out=45,in=-165] (0) to (m1);
	\draw[1c, out=30,in=120] (m1) to (1);
	\draw[1c, out=-45,in=170] (m1) to (1);
	\draw[2c] (0,-1.2) to node[auto,swap,arlabel] {$\,a$} (0,-.1);
	\draw[2c] (.5,.1) to node[auto,swap,arlabel] {$\,b$} (.5,.9);
	\draw[2c] (-.4,-.1) to (-.4,.8);
	\draw[2c] (-.6,.6) to (-.6,1.4);
	\draw[3c1] (1.65,0) to node[above=6pt,arlabel] {$p_1;a$} (3.1,0);
	\draw[3c2] (1.65,0) to (3.1,0);
	\draw[3c3] (1.65,0) to (3.1,0);
\end{scope}
\begin{scope}[shift={(9,0)}]
	\path[fill, color=gray!20] (-1.5,0) to [out=-10,in=135] (.375,-.75) to [out=15,in=-135] (1.5,0) to [out=170,in=-45] (-.375,.75) to [out=-165,in=45] (-1.5,0);
	\node[0c] (0) at (-1.5,0) {};
	\node[0c] (1) at (1.5,0) {};
	\node[0c] (m1) at (-.375,.75) {};
	\node[0c] (m2) at (.375,-.75) {};
	\draw[1c, out=75, in=105, looseness=1.5] (0) to (1);
	\draw[1c, out=-75, in=-105, looseness=1.5] (0) to (1);
	\draw[1c] (0) to (1);
	\draw[1c, out=45,in=-165] (0) to (m1);
	\draw[1c, out=30,in=120] (m1) to (1);
	\draw[1c, out=-45,in=170] (m1) to (1);
	\draw[1c, out=15,in=-135] (m2) to (1);
	\draw[1c, out=-60,in=-150] (0) to (m2);
	\draw[1c, out=-10,in=135] (0) to (m2);
	\draw[2c] (-.6,-.9) to node[auto,swap,pos=.4,arlabel] {$\,a$} (-.6,-.1);
	\draw[2c] (.5,.1) to node[auto,swap,arlabel] {$\,b$} (.5,.9);
	\draw[2c] (-.4,-.1) to (-.4,.8);
	\draw[2c] (-.6,.6) to (-.6,1.4);
	\draw[2c] (.4,-.8) to (.4,.1);
	\draw[2c] (.6,-1.4) to (.6,-.6);
	\draw[3c1] (1.65,0) to node[above=6pt,arlabel] {$!;v$} (3.1,0);
	\draw[3c2] (1.65,0) to (3.1,0);
	\draw[3c3] (1.65,0) to (3.1,0);
\end{scope}
\end{tikzpicture}
 \\
\begin{tikzpicture}[baseline={([yshift=-.5ex]current bounding box.center)},scale=.8]
\begin{scope}
	\path[fill, color=gray!20] (-.375,.75) to (.375,-.75) to [out=15,in=-135] (1.5,0) to [out=120,in=30, looseness=1.1] (-.375,.75);
	\node[0c] (0) at (-1.5,0) {};
	\node[0c] (1) at (1.5,0) {};
	\node[0c] (m1) at (-.375,.75) {};
	\node[0c] (m2) at (.375,-.75) {};
	\draw[1c, out=75, in=105, looseness=1.5] (0) to (1);
	\draw[1c, out=-75, in=-105, looseness=1.5] (0) to (1);
	\draw[1c] (m1) to (m2);
	\draw[1c, out=45,in=-165] (0) to (m1);
	\draw[1c, out=30,in=120] (m1) to (1);
	\draw[1c, out=-45,in=170] (m1) to (1);
	\draw[1c, out=15,in=-135] (m2) to (1);
	\draw[1c, out=-60,in=-150] (0) to (m2);
	\draw[1c, out=-10,in=135] (0) to (m2);
	\draw[2c] (-.6,-.9) to node[auto,swap,pos=.4,arlabel] {$\,a$} (-.6,-.1);
	\draw[2c] (.5,.1) to node[auto,swap,arlabel] {$\,b$} (.5,.9);
	\draw[2c] (-.5,-.3) to (-.5,.6);
	\draw[2c] (-.6,.6) to (-.6,1.4);
	\draw[2c] (.5,-.6) to (.5,.3);
	\draw[2c] (.6,-1.4) to (.6,-.6);
	\draw[3c1] (1.65,0) to node[above=6pt,arlabel] {$q_1;b$} (3.1,0);
	\draw[3c2] (1.65,0) to (3.1,0);
	\draw[3c3] (1.65,0) to (3.1,0);
\end{scope}
\begin{scope}[shift={(4.5,0)}]
	\path[fill, color=gray!20] (-1.5,0) to [out=-60,in=-150, looseness=1.1] (.375,-.75) to (-.375,.75) to [out=-165, in=45] (-1.5,0);
	\node[0c] (0) at (-1.5,0) {};
	\node[0c] (1) at (1.5,0) {};
	\node[0c] (m1) at (-.375,.75) {};
	\node[0c] (m2) at (.375,-.75) {};
	\draw[1c, out=75, in=105, looseness=1.5] (0) to (1);
	\draw[1c, out=-75, in=-105, looseness=1.5] (0) to (1);
	\draw[1c] (m1) to (m2);
	\draw[1c, out=45,in=-165] (0) to (m1);
	\draw[1c, out=30,in=120] (m1) to (1);
	\draw[1c, out=15,in=-135] (m2) to (1);
	\draw[1c, out=90,in=150] (m2) to (1);
	\draw[1c, out=-60,in=-150] (0) to (m2);
	\draw[1c, out=-10,in=135] (0) to (m2);
	\draw[2c] (-.6,-.9) to node[auto,swap,pos=.4,arlabel] {$\,a$} (-.6,-.1);
	\draw[2c] (.9,-.6) to node[auto,pos=.3,arlabel] {$b$} (.9,.2);
	\draw[2c] (-.5,-.3) to (-.5,.6);
	\draw[2c] (.4,-.2) to (.4,.8);
	\draw[2c] (-.6,.6) to (-.6,1.4);
	\draw[2c] (.6,-1.4) to (.6,-.6);
	\draw[3c1] (1.65,0) to node[above=6pt,arlabel] {$q_2;a$} (3.1,0);
	\draw[3c2] (1.65,0) to (3.1,0);
	\draw[3c3] (1.65,0) to (3.1,0);
\end{scope}
\begin{scope}[shift={(9,0)}]
	\path[fill, color=gray!20] (-1.5,0) to [out=-60,in=-150, looseness=1.1] (.375,-.75) to [out=90,in=150, looseness=1.1] (1.5,0) to [out=120,in=30, looseness=1.1] (-.375,.75) to [out=-90, in=-30, looseness=1.1] (-1.5,0);
	\node[0c] (0) at (-1.5,0) {};
	\node[0c] (1) at (1.5,0) {};
	\node[0c] (m1) at (-.375,.75) {};
	\node[0c] (m2) at (.375,-.75) {};
	\draw[1c, out=75, in=105, looseness=1.5] (0) to (1);
	\draw[1c, out=-75, in=-105, looseness=1.5] (0) to (1);
	\draw[1c] (m1) to (m2);
	\draw[1c, out=45,in=-165] (0) to (m1);
	\draw[1c, out=-30,in=-90] (0) to (m1);
	\draw[1c, out=30,in=120] (m1) to (1);
	\draw[1c, out=15,in=-135] (m2) to (1);
	\draw[1c, out=90,in=150] (m2) to (1);
	\draw[1c, out=-60,in=-150] (0) to (m2);
	\draw[2c] (-.9,-.2) to node[auto,swap,pos=.7,arlabel] {$a$} (-.9,.6);
	\draw[2c] (.9,-.6) to node[auto,pos=.3,arlabel] {$b$} (.9,.2);
	\draw[2c] (-.4,-.8) to (-.4,.2);
	\draw[2c] (.4,-.2) to (.4,.8);
	\draw[2c] (-.6,.6) to (-.6,1.4);
	\draw[2c] (.6,-1.4) to (.6,-.6);
	\draw[3c1] (1.65,0) to node[above=6pt,arlabel] {$!;v$} (3.1,0);
	\draw[3c2] (1.65,0) to (3.1,0);
	\draw[3c3] (1.65,0) to (3.1,0);
\end{scope}
\end{tikzpicture}
 \\[-10pt]
\begin{tikzpicture}[baseline={([yshift=-.5ex]current bounding box.center)},scale=.8]
\begin{scope}
	\path[fill, color=gray!20] (-1.5,0) to [out=75,in=105,looseness=1.6] (1.5,0) to [out=135,in=-45, looseness=1.1] (-1.5,0);
	\node[0c] (0) at (-1.5,0) {};
	\node[0c] (1) at (1.5,0) {};
	\node[0c] (m1) at (-.375,.75) {};
	\node[0c] (m2) at (.375,-.75) {};
	\draw[1c, out=75, in=105, looseness=1.5] (0) to (1);
	\draw[1c, out=-75, in=-105, looseness=1.5] (0) to (1);
	\draw[1c, out=-45, in=135] (0) to (1);
	\draw[1c, out=45,in=-165] (0) to (m1);
	\draw[1c, out=-30,in=-90] (0) to (m1);
	\draw[1c, out=30,in=120] (m1) to (1);
	\draw[1c, out=15,in=-135] (m2) to (1);
	\draw[1c, out=90,in=150] (m2) to (1);
	\draw[1c, out=-60,in=-150] (0) to (m2);
	\draw[2c] (-.9,-.2) to node[auto,swap,pos=.7,arlabel] {$a$} (-.9,.6);
	\draw[2c] (.9,-.6) to node[auto,pos=.3,arlabel] {$b$} (.9,.2);
	\draw[2c] (-.2,-1) to (-.2,0);
	\draw[2c] (.2,-0) to (.2,1);
	\draw[2c] (-.6,.6) to (-.6,1.4);
	\draw[2c] (.6,-1.4) to (.6,-.6);
	\draw[3c1] (1.65,0) to node[above=6pt,arlabel] {$p_1';a$} (3.15,0);
	\draw[3c2] (1.65,0) to (3.1,0);
	\draw[3c3] (1.65,0) to (3.1,0);
\end{scope}
\begin{scope}[shift={(4.5,0)}]
	\path[fill, color=gray!20] (-1.5,0) to [out=-75,in=-105,looseness=1.6] (1.5,0) to [out=135,in=-45, looseness=1.1] (-1.5,0);
	\node[0c] (0) at (-1.5,0) {};
	\node[0c] (1) at (1.5,0) {};
	\node[0c] (m2) at (.375,-.75) {};
	\draw[1c, out=75, in=105, looseness=1.5] (0) to (1);
	\draw[1c, out=-75, in=-105, looseness=1.5] (0) to (1);
	\draw[1c, out=-45, in=135] (0) to (1);
	\draw[1c, out=15,in=-135] (m2) to (1);
	\draw[1c, out=90,in=150] (m2) to (1);
	\draw[1c, out=-60,in=-150] (0) to (m2);
	\draw[2c] (0,.1) to node[auto,swap,arlabel] {$\,a$} (0,1.2);
	\draw[2c] (.9,-.6) to node[auto,pos=.3,arlabel] {$b$} (.9,.2);
	\draw[2c] (-.2,-1) to (-.2,0);
	\draw[2c] (.6,-1.4) to (.6,-.6);
	\draw[3c1] (1.65,0) to node[above=6pt,arlabel] {$p_2';b$} (3.1,0);
	\draw[3c2] (1.65,0) to (3.1,0);
	\draw[3c3] (1.65,0) to (3.1,0);
\end{scope}
\begin{scope}[shift={(9,0)}]
	\node[0c] (0) at (-1.5,0) {};
	\node[0c] (1) at (1.5,0) {};
	\draw[1c, out=75, in=105, looseness=1.5] (0) to (1);
	\draw[1c, out=-75, in=-105, looseness=1.5] (0) to (1);
	\draw[1c, out=-45, in=135] (0) to (1);
	\draw[2c] (0,.1) to node[auto,swap,arlabel] {$\,a$} (0,1.2);
	\draw[2c] (0,-1.2) to node[auto,swap,arlabel] {$\,b$} (0,-.1);
	\node at (1.75,-1) {,};
\end{scope}
\end{tikzpicture}

\end{align*}
\endgroup 
where the shaded area in each diagram indicates the input boundary of the following 3\nbd cell. The unlabelled cells are all equal to $!;v$. The maps $p_1, p_2,$ $q_1, q_2,$ $p'_1, p'_2$ are all surjective maps over $O^2$ with $p'_i = \rev{p_i}$ for each $i \in \{1,2\}$. The reader can reconstruct explicit expressions for these maps.

The diagram $x_2$ is constructed dually by exchanging the r\^oles of $a$ and $b$ and reversing all the 3\nbd cells in $x_1$. 

In $\Pi X$, let $a$ and $b$ also denote the cells whose materialisations are $a$ and $b$. Then we have $a \cp{1} b = b \cp{1} a$ by the strict Eckmann--Hilton argument. Because degenerate cells in $X$ are materialisations of units in $\Pi X$ and composites of units are units, both $x_1$ and $x_2$ are materialisations of $\eps{}(a \cp{1} b)$. We claim that there need not be any diagram of type $x_1 \celto x_2$ in $X$. 

Construct the topological 2-sphere $S^2$ as the quotient of $\realis{O^2}$ by the subspace $\realis{\bord O^2}$. By adjointness, the quotient map $f\colon \realis{O^2} \to S^2$ corresponds to a 2-cell $f\colon !;v \celto\ !;v$ of shape $O^2$ in $\sing{S^2}$, where $v$ corresponds to the image of $\realis{\bord O^2}$ through $f$. 

Let $a = b \eqdef f$. The diagrams $x_1$ and $x_2$ correspond to 3\nbd cells in $S^2$. If there were a diagram of type $x_1 \celto x_2$ in $\sing{S^2}$ it would correspond to a homotopy between these cells relative to their boundary. Such a homotopy would cause the Whitehead product $[f,f]$ to vanish in $S^2$, which is not the case. 

This proves that there is no materialisation $e\colon x_1 \celto x_2$ in $\sing{S^2}$ of the unit $\eps{4}(f \cp{1} f)$ in $\Pi(\sing{S^2})$.
\end{dfn}

\subsection*{Acknowledgements}

This work was supported by a JSPS Postdoctoral Research Fellowship, by JSPS KAKENHI Grant Number 17F17810, and by an FSMP Postdoctoral Research Fellowship.

I am grateful to Pierre-Louis Curien, Yves Guiraud, Masahito Hasegawa, Simon Henry, Alex Kavvos, Chaitanya Leena-Subramaniam, Sophie Turner, and Jamie Vicary for various forms of help.

\bibliographystyle{alpha}
\small \bibliography{main}

\end{document}